\documentclass[12pt,reqno]{book}
\usepackage{microtype}
\usepackage[leqno]{mathtools}
\usepackage{amsthm,amssymb}
\usepackage{hyperref}
\usepackage{hyphenat}
\usepackage[all]{xy}
\usepackage[latin1]{inputenc}
\usepackage[T1]{fontenc}
\usepackage{lmodern}
\usepackage{ifthen}
\usepackage{fullpage}

\newcommand{\KK}{{\mathbb{K}}}  %positive orthant
\newcommand{\NN}{{\mathbb{N}}}  %natural numbers
\newcommand{\QQ}{{\mathbb{Q}}}  %rational numbers
\newcommand{\RR}{{\mathbb{R}}}  %real numbers or cartesian space
\renewcommand{\SS}{{\mathbb{S}}} %sphere
\newcommand{\ZZ}{{\mathbb{Z}}}  %integers

\newcommand{\id}{{\operatorname{id}}}  %identity
\newcommand{\supp}{{\operatorname{supp}}}  %support
\newcommand{\dom}{{\operatorname{dom}}} %domain
\newcommand{\inter}{{\operatorname{int}}} %interior
\newcommand{\limit}[1]{{\underset{#1}{\lim}}} %limit
\newcommand{\pr}{{\operatorname{pr}}} %projection

\newcommand{\CIN}{{C^\infty}}   %infinitely differentiable
\newcommand{\Der}{{\operatorname{Der}\,}}  %derivations functor
\newcommand{\Vect}{{\operatorname{Vect}}}  %vector fields
\newcommand{\Diff}{{\operatorname{Diff}}}  %diffeomorphisms
\newcommand{\ham}{{\operatorname{Ham}}} %hamiltonian vector fields
\renewcommand{\O}{{\operatorname{O}}} %orthogonal
\newcommand{\SO}{{\operatorname{SO}}}  %special orthogonal
\newcommand{\g}{{\mathfrak{g}}} %generic Lie algebra

\newcommand{\Stab}{{\operatorname{Stab}}} %stabiliser
\newcommand{\Ad}{{\operatorname{Ad}}}  %adjoint operator
\newcommand{\pois}[2]{{\{#1,#2\}}}  %poisson bracket
\newcommand{\hook}{{\lrcorner\,}} %hook or interior multiplication
\newcommand{\hypref}[2]{{\hyperref[#1]{#2~\ref{#1}}}}

\newcommand{\ifwork}[1]{\ifthenelse{\boolean{workmode}}{#1}{}}
\newcommand{\comment}[1]{}
\newcommand{\mute}[1]{}
\newcommand{\printname}[1]{}

\newcommand{\labell}[1] {\label{#1} \printname{#1}}

\theoremstyle{plain}
\newtheorem{theorem}{Theorem}[chapter]

\newtheorem{proposition}[theorem]{Proposition}

\newtheorem{corollary}[theorem]{Corollary}
\newtheorem{lemma}[theorem]{Lemma}

\theoremstyle{definition}
\newtheorem{definition}[theorem]{Definition}
\newtheorem{example}[theorem]{Example}
\newtheorem{remark}[theorem]{Remark}

\newtheorem{notation}[theorem]{Notation}
\newtheorem{question}[theorem]{Question} 

\author{Jordan Watts}
\title{Diffeologies, Differential Spaces, and Symplectic Geometry}

\date{Modified: \today}

\begin{document}

\maketitle

\newpage

\section*{\textbf{\underline{Abstract}}}
Diffeological and differential spaces are generalisations of smooth structures on manifolds.  We show that the ``intersection'' of these two categories is isomorphic to Fr\"olicher spaces, another generalisation of smooth structures. We then give examples of such spaces, as well as examples of diffeological and differential spaces that do not fall into this category.\\

We apply the theory of diffeological spaces to differential forms on a geometric quotient of a compact Lie group.  We show that the subcomplex of basic forms is isomorphic to the complex of diffeological forms on the geometric quotient.  We apply this to symplectic quotients coming from a regular value of the momentum map, and show that diffeological forms on this quotient are isomorphic as a complex to Sjamaar differential forms.  We also compare diffeological forms to those on orbifolds, and show that they are isomorphic complexes as well.\\

We apply the theory of differential spaces to subcartesian spaces equipped with families of vector fields.  We use this theory to show that smooth stratified spaces form a full subcategory of subcartesian spaces equipped with families of vector fields.  We give families of vector fields that induce the orbit-type stratifications induced by a Lie group action, as well as the orbit-type stratifications induced by a Hamiltonian group action.

\newpage

\section*{\textbf{\underline{Acknowledgements}}}
First and foremost, I would like to thank my advisor, Yael Karshon, for all of her helpful advice, support, faith, and patience over the past number of years.  From what I hear, very few advisors actually read their graduate students' work, but I have had the opposite experience!  During my time here in Toronto, under her guidance, I have learned an incredible amount, and have matured substantially as a mathematician.  At least I would like to think so.\\

Next I would like to thank my supervisory committee, comprising Marco Gualtieri and Boris Khesin, for all of their suggestions throughout my PhD programme.  I also would like to thank my defence committee, including Marco Gualtieri, Lisa Jeffrey, Paul Selick, and the external examiner, Richard Cushman.  In particular, I would like thank Richard for his enthusiasm while reading my thesis and at the defence itself, and his careful checking of my bibliography!\\

I would like to thank my co-author and collaborators involved in parts of this thesis.  These include Jiayong Li, Augustin Batubenge, Patrick Iglesias-Zemmour, and Yael Karshon.  I would also like to thank Patrick for many stimulating discussions about my paper with Jiayong, and about diffeology in general.\\

I would like to thank J\k{e}drzej \'Sniatycki, my former advisor for my MSc thesis at the University of Calgary, for his continued interest in and support of my work, and the discussions we have had during my PhD programme.  Part of this thesis rests on a theoretical foundation established by him.\\

I would like to thank George Elliott for his generous support during my last year of my PhD programme.  Without this, life would have been incredibly harsh, and the degree might have been very difficult to complete.\\

Next, a special thank you to Ida Bulat, for her deep kindness, helpfulness, and vast knowledge of the inner workings of the PhD programme and beyond.  If Ida was not here, we would all be terribly lost, students and faculty alike.\\

Finally, I would like to thank my family for their unwavering support and confidence in me.  Thanks to my friends who have encouraged me, helped me through tough times, and have given me a number of sanity/reality checks over the years which were much-needed (even if in the end they failed on the sanity part).  And last, thanks to the Department of Mathematics at the University of Toronto for creating a stimulating environment where I have felt free to talk about the thing I love, mathematics, with others who also love it.  I will sorely miss this place. 

\newpage

\tableofcontents

\chapter{Introduction}\labell{ch:introduction}

Differential topology on manifolds is a well understood subject.  But the category of manifolds is not closed under subsets and quotients, among many other categorical constructs, like the topological category is.  Moreover, in the literature, there are many theories generalising the notion of smooth structure.  Stratified spaces, orbifolds, and geometric quotients of Lie group actions are just some specific examples, and many categories such as differential spaces, diffeological spaces, and Fr\"olicher spaces have been defined in order to give a general framework for these examples (see \cite{sikorski67}, \cite{sikorski71}, \cite{mostow79}, \cite{souriau81}, \cite{iglesias}, \cite{frolicher82}).\\

Many of these generalisations are related, and the relationships between them, along with some of their categorical properties, have been studied.  For example, see \cite{stacey10}, \cite{baez-hoffnung11}.\\

In this thesis, I explore two categories: diffeological spaces and differential spaces.  A diffeological space is a set equipped with a family of maps \emph{into} it.  The family contains all the constant maps, enjoys a locality condition, and satisfies a smooth compatibility condition.  Diffeological spaces retain a lot of information when taking quotients.  For example, when looking at the action of $\O(n)$ on $\RR^n$ via rotations, the diffeological structure on the geometric quotient remembers which $n$ is in the original set-up, even though all of the geometric quotients are homeomorphic.\\

A differential space is a set equipped with a family of real-valued functions, which satisfies a locality condition and a smooth compatibility condition.  Over the weakest topology such that the family is continuous, these functions induce a sheaf.  These spaces retain a lot of information when taking subsets.  For example, consider the $x$ and $y$-axes in $\RR^2$, along with another line given by $y=mx$ ($m>0$).  The differential space structure can tell the difference between different choices of $m$, even though all of the subsets are homeomorphic.\\

Both categories yield smooth structures on arbitrary subsets and quotients, and both contain manifolds, manifolds with boundary, and manifolds with corners as full subcategories.  In this thesis I am interested in comparing these two categories, along with providing examples and applications of each.  There will be a lot of focus on geometric quotients coming from compact Lie group actions, and symplectic quotients coming from compact Hamiltonian group actions.  I will now begin to outline each chapter.\\

In \hypref{ch:smoothstructures}{Chapter} I review the basic theory of diffeological and differential spaces, and give examples of each.  I then compare these two categories.  In particular, I show that the ``intersection'' of these two categories is isomorphic to another smooth category that appears often in the literature known as Fr\"olicher spaces (see \hypref{t:reflexiveisom}{Theorem}, \hypref{t:frolicher}{Theorem}, and \hypref{c:frolicher}{Corollary}).  I then in the following section give examples of such spaces, along with examples that are not.  This part is joint work with Augustin Batubenge, Patrick Iglesias-Zemmour, and Yael Karshon.\\

In \hypref{ch:forms}{Chapter} I study diffeological forms on geometric and symplectic quotients.  I show that on a geometric quotient coming from a compact Lie group, the diffeological forms are isomorphic to the basic forms on the manifold upstairs (see \hypref{t:geomquot}{Theorem}).  I use this theorem to show that the definition of differential forms on orbifolds that appears in the literature is isomorphic to the diffeological forms on the orbifold diffeological structure (see \cite{ikz10} for more on the orbifold diffeological structure, and \hypref{s:orbifolds}{Section} for the results).  Also, I show that differential forms defined on symplectic quotients by Sjamaar are isomorphic to diffeological forms in the case of symplectic reduction at a regular value of the momentum map. In other words, Sjamaar forms are intrinsic to the diffeological structure on the symplectic quotient in this case (see \hypref{t:sjamaar}{Theorem}).  I would like to thank Yael Karshon for making me aware of \hypref{l:techn1}{Lemma}, \hypref{l:techn2}{Lemma}, and \hypref{c:techn2}{Corollary}, along with their proofs.  These pieces of the puzzle are crucial in showing that basic forms are equal to pullbacks from the orbit space, and are due to her.\\

Subcartesian spaces are differential spaces that are locally diffeomorphic to arbitrary subsets of Euclidean spaces. \hypref{ch:vectorfields}{Chapter} is a detailed analysis of vector fields on subcartesian spaces, enhancing the theory developed by \'Sniatycki (see \cite{sniatycki03}).  A vector field is a smooth section of the Zariski tangent bundle that admits a local flow.  Besides reviewing the relevant theory, I prove a characterisation of a vector field in terms of certain ideals of the ring of smooth functions on vector field orbits (see \hypref{p:charvect}{Proposition}).  This immediately implies that vector fields form a Lie subalgebra of derivations of the ring of smooth functions on a subcartesian space (see \hypref{c:charvect}{Corollary}).  I study subcartesian spaces that are also stratified spaces (in a way compatible with the smooth structure), along with stratified maps.  I prove that these form a full subcategory of subcartesian spaces equipped with locally complete families of vector fields, along with so-called orbital maps (see \hypref{c:orbitalcat}{Corollary}).\\

I apply this theory to both geometric quotients and symplectic quotients.  In particular, given a compact Lie group action on a manifold, I construct a family of vector fields whose orbits are exactly the orbit-type strata (see \hypref{d:A}{Definition} and \hypref{t:singfolA}{Theorem}), and obtain that the quotient map is orbital with respect to this family of vector fields and the family of all vector fields on the quotient (see \hypref{c:piorbital}{Corollary}).  In the case of a Hamiltonian action, on the zero set of the momentum map, I construct a family of vector fields whose orbits are contained in the orbit-type strata, and are equal to the orbit-type strata in the case of a connected group (see \hypref{d:AZ}{Definition} and \hypref{p:vectZorbits}{Proposition}).  I then show that the quotient map restricted to the zero set is orbital with respect to this family of vector fields and the family of all vector fields on the symplectic quotient (see \hypref{p:piZorbit}{Proposition}).

\hypref{ch:sphere}{Chapter} is a report on a joint paper with Jiayong Li in which we show that orientation-preserving diffeomorphisms of $\SS^2$ have a diffeologically smooth strong deformation retraction onto $SO(3)$.  The continuous case with respect to the $C^k$-topology ($1<k\leq\infty$) is proven by Smale in his 1959 paper.  While much of this work is translating the work of Smale into the diffeological language, there are pieces in Smale's paper that simply do not translate at all.  The use of some Sobolev inequalities, as well as some careful modifications to Smale's arguments, are required to achieve smoothness.  (To view the actual publication, see \cite{li-watts10} or \texttt{<http://arxiv.org/abs/0912.2877>}.)\\

At the end, in an appendix, I give two different definitions of ``subcartesian space'' that appear in the literature.  I do not know of any place in the literature where the two definitions are shown to be equivalent, and so this appendix is dedicated to showing that these two categories are in fact isomorphic (see \hypref{t:subcart}{Theorem}).\\

I will give some final post-defence thoughts before I end this introduction.  This thesis is mainly about extending our notion of smoothness beyond manifolds to more general spaces.  The immediate question is, what exactly do we mean by ``smoothness'' in this context?  In the last three chapters we mostly stick to either the language of diffeology, or to differential spaces.  However, the last couple of sections of Chapter 2 show that the category of diffeological spaces, and that of differential spaces, interact with each other in some way; in particular, we examine their ``intersection''.  That said, I think it would be more natural to write down one category that includes diffeological spaces and differential spaces as full subcategories instead of switching between the two languages whenever it seems convenient.\\

I suggest the following: let $\mathcal{C}$ be the category whose objects are triples $(X,\mathcal{D},\mathcal{F})$ where $X$ is a set, $\mathcal{D}$ is a diffeology on $X$, and $\mathcal{F}$ is a differential structure on $X$.  There is a compatibility requirement such that for any plot $p\in\mathcal{D}$ and any function $f\in\mathcal{F}$, the composition $f\circ p$ is smooth in the usual sense on the domain of $p$.  Maps are required to push forward plots to plots and pull back functionally smooth functions to functionally smooth functions.  Note that such a category is much bigger than the category of Fr\"olicher spaces, and includes Fr\"olicher spaces as a full subcategory by the results of Chapter 2.  It also includes such things as the category of smooth stratified spaces as full subcategories.  Indeed, replace the stratification with a diffeology in which each plot has its image contained in a stratum.  Maps between these spaces are then automatically smooth stratified maps.\\

More details on this idea can be found in the paper that inspired it: Stacey's paper \cite{stacey10}.  The introduction to this paper is an excellent overview on a possible answer to the question, "What does it mean to be smooth?" and how mathematicians have attempted to answer it. 
\chapter{Diffeological and Differential Spaces}\labell{ch:smoothstructures}

In the literature, there are many different categories defined in order to generalise the concept of ``smooth structure''.  Some of these form a chain, each category within containing the previous ones as full subcategories:\\

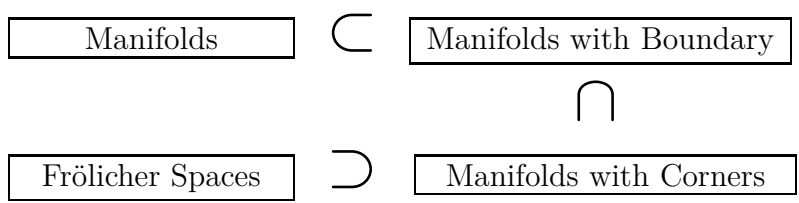
\begin{center}
\begin{tabular}{ccc}

\framebox[1.5in][c]{Manifolds} & \huge{$\subset$} & \framebox[2in][c]{Manifolds with Boundary} \vspace{0.1in} \\
 & & \huge{$\cap$} \vspace{0.1in} \\
\framebox[1.5in][c]{Fr\"olicher Spaces} &
\huge{$\supset$} &
\framebox[2in][c]{Manifolds with Corners} \\

\end{tabular}
\end{center}
\vspace{0.1in}

I define Fr\"olicher spaces later on in \hypref{d:frolicher}{Definition}.  There many other objects in mathematics that obtain natural ``smooth structures'' that do not belong in any of the categories above, however.  I discuss two more categories, diffeological spaces and differential spaces, each of which contains the four categories above as full subcategories.  For a more complete picture of these categories of smooth objects appearing in the literature, with details on the functors between them, see \cite{stacey10}.\\

\hypref{s:diffeology}{Section} reviews the basic theory of diffeology, and \hypref{s:differentialspaces}{Section} reviews that of differential spaces.  The new results in this chapter lie within the last two sections.  \hypref{s:reflexivity}{Section} is dedicated to showing that the ``intersection'' of the categories of diffeological spaces and differential spaces is isomorphic to Fr\"olicher spaces (see \hypref{t:reflexiveisom}{Theorem}, \hypref{t:frolicher}{Theorem}, and \hypref{c:frolicher}{Corollary}).  \hypref{s:examples}{Section} is comprised of a set of examples of spaces appearing in this intersection, and a set of counterexamples not inside of it.

\section{Diffeology}\labell{s:diffeology}

Diffeology was developed by Souriau (see \cite{souriau81}) in the early 1980's.  The theory of diffeology was then further developed by others, in particular Iglesias-Zemmour (see \cite{iglesias}), whose text will be our primary reference for the theory.  A similar theory was developed by Chen in the 1970's.  The definition of what is now referred to as a ``Chen space'' went through many revisions in a series of papers, the end result being the same as the definition of a diffeology below, but with open subsets of Euclidean spaces as domains of plots replaced with convex subsets of Euclidean spaces.  See \cite{chen73}, \cite{chen75}, \cite{chen77}, \cite{chen82}.\\

\begin{definition}[Diffeology]
Let $X$ be a nonempty set.  A \emph{parametrisation} of $X$ is a function $p:U\to X$ where $U$ is an open subset of $\RR^n$ for some $n$.  A \emph{diffeology} $\mathcal{D}$ on $X$ is a set of parametrisations satisfying the following three conditions.
\begin{enumerate}
\item (Covering) For every $x\in X$ and every nonnegative integer $n\in\NN$, the constant function $p:\RR^n\to\{x\}\subseteq X$ is in $\mathcal{D}$.
\item (Locality) Let $p:U\to X$ be a parametrisation such that for every $u\in U$ there exists an open neighbourhood $V\subseteq U$ of $u$ satisfying $p|_V\in\mathcal{D}$. Then $p\in\mathcal{D}$.
\item (Smooth Compatibility) Let $p:U\to X$ be a plot in $\mathcal{D}$.  Then for every $n\in\NN$, every open subset $V\subseteq\RR^n$, and every smooth map $F:V\to U$, we have $p\circ F\in\mathcal{D}$.
\end{enumerate}

$X$ equipped with a diffeology $\mathcal{D}$ is called a \emph{diffeological space}, and is denoted by $(X,\mathcal{D})$.  When the diffeology is understood, we will drop the symbol $\mathcal{D}$.  The parametrisations $p\in\mathcal{D}$ are called \emph{plots}.
\end{definition}

\begin{example}[Standard Diffeology on a Manifold]
Let $M$ be a manifold.  The \emph{standard diffeology} on $M$ is the set of all smooth maps from open subsets of $\RR^n$, $n\in\NN$, to $M$.  In particular, any plot $p:U\to M$ has the property that for any $u\in U$, there is an open neighbourhood $V\subseteq U$ of $u$, a chart $q:W\to M$, and a smooth map $F:V\to W$ satisfying $$p|_V=q\circ F.$$
\end{example}

\begin{definition}[Diffeologically Smooth Maps]
Let $(X,\mathcal{D}_X)$ and $(Y,\mathcal{D}_Y)$ be two diffeological spaces, and let $F:X\to Y$ be a map.  Then we say that $F$ is \emph{diffeologically smooth} if for any plot $p\in\mathcal{D}_X$, $$F\circ p\in\mathcal{D}_Y.$$  Denote the set of all diffeologically smooth maps between $X$ and $Y$ by $\mathcal{D}(X,Y)$.
\end{definition}

\begin{example}[Smooth Maps Between Manifolds]
Any smooth map between two smooth manifolds is diffeologically smooth with respect to the standard diffeologies on the manifolds.  That is, if $M$ and $N$ are manifolds, then $\mathcal{D}(M,N)=\CIN(M,N)$.
\end{example}

\begin{remark}
Diffeological spaces, along with diffeologically smooth maps, form a category.
\end{remark}

\begin{definition}[Generating Diffeologies]
Let $\mathcal{D}_0$ be a family of parametrisations into a set $X$.  Then the diffeology \emph{generated} by $\mathcal{D}_0$ is the set of all plots $p:U\to X$ satisfying the following property.  For any $u\in U$ there is an open neighbourhood $V\subseteq U$ of $u$ such that either $p|_V$ is a constant parametrisation, or there is a parametrisation $(q:W\to X)\in\mathcal{D}_0$ and a smooth map $F:V\to W$ satisfying $p|_V=q\circ F$.  Equivalently, the diffeology $\mathcal{D}$ generated by $\mathcal{D}_0$ is the smallest diffeology containing $\mathcal{D}_0$; that is, any other diffeology containing $\mathcal{D}_0$ also contains $\mathcal{D}$.
\end{definition}

\begin{definition}[Quotient Diffeology]
Let $X$ be a diffeological space, and let $\sim$ be an equivalence relation on $X$.  Let $Y=X/\!\!\sim$, and let $\pi:X\to Y$ be the quotient map.  We define the \emph{quotient diffeology} on $Y$ to be the set of plots $p:U\to Y$ satisfying: for all $u\in U$, there exist an open neighbourhood $V\subseteq U$ of $u$ and a plot $q:V\to X$ so that $p|_V=\pi\circ q$.
\end{definition}

Let $M$ be a manifold, and let $G$ be a Lie group acting smoothly on $M$.  Throughout this thesis we use the following notation: for $g\in G$ and $x\in M$, we denote the action of $g$ on $x$ by $g\cdot x$, and the $G$-orbit of $x$ by $G\cdot x$.  Since $g$ can be viewed as a diffeomorphism of $M$ onto itself, we shall often write $g_*$ and $g^*$ for the pushforward and pullback of this diffeomorphism.\\

\begin{example}[Geometric Quotient]
By definition of the quotient diffeology, we know that for any plot $p:U\to M/G$, and for every $u\in U$, there exist an open neighbourhood $V\subseteq U$ of $u$ and a plot $q$ of $M$ satisfying $p|_V=\pi\circ q$, where $\pi:M\to M/G$ is the quotient map.  However, we also know that any plot $q$ of $M$ is just a smooth map from its domain into $M$.  Thus, the diffeology on $M/G$ is generated by the set $\{\pi\circ q~|~q:V\to M\text{ is smooth, $V$ open in some $\RR^n$}\}$.
\end{example}

\begin{definition}[Subset Diffeology]
Let $(X,\mathcal{D})$ be a diffeological space, and let $Y$ be a subset of $X$.  Then the \emph{subset diffeology} on $Y$ is defined to be all plots $p:U\to X$ whose images are contained in $Y$.
\end{definition}

Let $X$ be a diffeological space, and let $\sim$ be an equivalence relation on $X$.  Let $Y$ be a subset of $X$ such that if $x\in Y$, then for any $y\in X$ in the same equivalence class as $x$, we have $y\in Y$ as well.  That is, $Y$ is the union of a set of equivalences classes.  Then $\sim$ induces an equivalence relation on $Y$, also denoted by $\sim$.

\begin{proposition}[Subquotient Diffeology]\labell{p:commsquare}
The subset diffeology on $Y/\!\!\sim$ induced by the quotient diffeology on $X/\!\!\sim$ is equal to the quotient diffeology on $Y/\!\!\sim$ induced by the subset diffeology on $Y$.
\end{proposition}

\begin{proof}
Let $\pi:X\to X/\!\!\sim$ and $\pi_Y:Y\to Y/\!\!\sim$ be the quotient maps.  Let $p:U\to Y/\!\!\sim$ be a plot in the subset diffeology induced by $X/\!\!\sim$.  Then $p$ is a plot of $X/\!\!\sim$ with image in $Y/\!\!\sim$.  Thus, for every $u\in U$ there is an open neighbourhood $V\subseteq U$ of $u$ and a plot $q:V\to X$ such that $$p|_V=\pi\circ q.$$  Note that the image of $q$ is in $Y$, since for any $x\in Y$, all points $y\in X$ such that $y\sim x$ are contained in $Y$ as well.  Hence $$p|_V=\pi_Y\circ q.$$  So $p|_V$ is in the quotient diffeology on $Y/\!\!\sim$ induced by $Y$.  By the definition of the quotient diffeology, $p$ is a plot on $Y/\!\!\sim$.\\

Next, let $p:U\to Y/\!\!\sim$ be a plot in the quotient diffeology induced by $Y$.  Then, for every $u\in U$ there is an open neighbourhood $V\subseteq U$ of $u$ and a plot $q:V\to Y$ such that $$p|_V=\pi_Y\circ q.$$  But $q$ is in the subset diffeology on $Y$, and hence is a plot of $X$ with image in $Y$.  So, $$p|_V=\pi\circ q.$$  But then $p|_V$ is a plot in the quotient diffeology on $X/\!\!\sim$ with image in $Y/\!\!\sim$.  Thus, $p|_V$ is in the subset diffeology on $Y/\!\!\sim$ induced by $X/\!\!\sim$.  By the definition of a diffeology, $p$ itself is a plot in the subset diffeology on $Y/\!\!\sim$.
\end{proof}

We have the following commutative diagram of diffeological spaces, where the lemma above resolves any potential ambiguity in the diffeological structure on $Y/\!\!\sim$.

$$\xymatrix{
Y \ar[r]^i \ar[d]_{\pi_Y} & X \ar[d]^{\pi} \\
Y/\!\!\sim \ar[r]_j & X/\!\!\sim \\
}$$

\begin{example}[Symplectic Quotient]
Let $G$ be a Lie group acting on a symplectic manifold $(M,\omega)$.  The action is called \emph{symplectic} if it preserves the symplectic form; that is, for any $g\in G$, we have $g^*\omega=\omega$.  A symplectic action is \emph{Hamiltonian} if there is a $G$-coadjoint equivariant map $\Phi:M\to\g^*$ satisfying the following.  For any $\xi\in\g=\operatorname{Lie}(G)$, let $\Phi^\xi:M\to\RR$ be defined by $$\Phi^\xi(x):=\langle\xi,\Phi(x)\rangle$$ where $\langle,\rangle$ is the usual pairing between $\g$ and $\g^*$.  Then, $$\xi_M\hook\omega=-d\Phi^\xi.$$  Here, $\xi_M$ is the vector field on $M$ induced by $\xi$: for any $x\in M$, $$\xi_M|_x:=\frac{d}{dt}\Big|_{t=0}(\exp(t\xi)\cdot x).$$
$\Phi$ is called the \emph{momentum map} (in the literature it is also known as the \emph{moment map}).  Denote by $Z$ the zero level set of $\Phi$.  This is a $G$-invariant subset of $M$.  We call the orbit space $Z/G$ the \emph{symplectic quotient}.  The induced diffeology on the symplectic quotient $Z/G$ is both the subset diffeology from the geometric quotient $M/G$ and the quotient diffeology from $Z$ by \hypref{p:commsquare}{Proposition}.
\end{example}

\begin{definition}[Product Diffeology]
Let $X$ and $Y$ be two diffeological spaces.  Then the set $X\times Y$ acquires the \emph{product diffeology}, defined as follows.  Let $\operatorname{pr}_X:X\times Y\to X$ and $\operatorname{pr}_Y:X\times Y\to Y$ be the natural projections.  A parametrisation $p:U\to X\times Y$ is a plot if $\operatorname{pr}_X\circ p$ and $\operatorname{pr}_Y\circ p$ are plots of $X$ and $Y$, respectively.
\end{definition}

\begin{definition}[Smooth Functions]
Let $X$ be a diffeological space.  A (real-valued) \emph{diffeologically smooth function} on $X$ is a diffeologically smooth map $f:X\to\RR$ where $\RR$ is equipped with the standard diffeology.  Hence, for any plot $p:U\to X$, $f\circ p$ is a smooth function on $U$ in the usual sense.  The set of diffeologically smooth functions form a commutative $\RR$-algebra under pointwise addition and multiplication.  Note that this set is precisely $\mathcal{D}(X,\RR)$.  Note that if $X$ and $Y$ are diffeological spaces and $F:X\to Y$ is a diffeologically smooth map between them, then for any $f\in\mathcal{D}(Y,\RR)$, $f\circ F$ is a diffeologically smooth function on $X$.  Thus we have $$F^*\mathcal{D}(Y,\RR)\subseteq\mathcal{D}(X,\RR).$$
\end{definition}

\begin{example}
Let $M$ be a manifold.  Then $\mathcal{D}(M,\RR)=\CIN(M)$.
\end{example}

\begin{proposition}[Diffeologically Smooth Functions on a Geometric Quotient]\labell{p:functions}
Let $G$ be a Lie group acting on a manifold $M$, and let $\pi:M\to M/G$ be the orbit map.  Then $\pi^*:\mathcal{D}(M/G,\RR)\to\CIN(M)$ is an isomorphism of $\RR$-algebras onto the invariant smooth functions $\CIN(M)^G$.
\end{proposition}

\begin{proof}
Let $f\in\mathcal{D}(M/G,\RR)$.  Then the pullback $\pi^*f$ is smooth on $M$ since $\pi$ is smooth.  Also, for any $g\in G$, we have $g^*\pi^*f=\pi^*f$, and so the image of $\pi^*$ is in the set of invariant functions on $M$.\\

Next, assume that $\pi^*f=0$.  Then since $\pi$ is surjective, we have that $f=0$.  So $\pi^*$ is injective.  We need to show that it is surjective onto invariant functions.\\

Let $\tilde{f}$ be an invariant smooth function on $M$.  Define $f:M/G\to\RR$ to be $f(x):=\tilde{f}(y)$ where $y$ is any point in $\pi^{-1}(x)$.  This is well-defined since any two points in $\pi^{-1}(x)$ differ by an element of $G$, and $\tilde{f}$ is $G$-invariant.  Let $p:U\to M/G$ be a plot of $M/G$.  Then for any $u\in U$, there is an open neighbourhood $V\subseteq U$ of $u$ and a plot $q:V\to M$ such that $$p|_V=\pi\circ q.$$  Since $$f\circ\pi\circ q=\tilde{f}\circ q\in\CIN(V)$$ and smoothness of a function on $U$ is a local condition, we have that $f\circ p\in\CIN(U)$.  Since $p$ is an arbitrary plot, we have that $f\in\mathcal{D}(M/G,\RR)$.
\end{proof}

\begin{definition}[Standard Functional Diffeology]
Let $Y$ and $Z$ be diffeological spaces, and set $X=\mathcal{D}(Y,Z)$.  The \emph{standard functional diffeology} on $X$ is defined as follows.  A parametrisation $p:U\to X$ is a plot if the map $$\Psi_p:U\times Y\to Z:(\mu,y)\to p(\mu)(y)$$ is diffeologically smooth.  In this case we often will write $f_\mu:Y\to Z$ instead of $p(\mu):Y\to Z$ to emphasise that the image of $p$ is a family of functions.
\end{definition}

\section{Differential Spaces}\labell{s:differentialspaces}

Differential spaces were introduced by Sikorski in the late 1960's (see \cite{sikorski67}, \cite{sikorski71}), and appear in the literature often, even if they are not referred to as such.  For example, the ring of ``smooth'' functions on a geometric quotient from a compact Lie group, studied by Schwarz \cite{schwarz75} and Cushman-\'Sniatycki \cite{cushman-sniatycki01}, or the ring of functions on a symplectic quotient introduced by Arms-Cushman-Gotay \cite{acg91}, are differential structures on the respective spaces.\\

At the same time, in the late 1960's, Aronszajn expressed the need for a theory of smooth structures on arbitrary subsets of $\RR^n$ (see \cite{aronszajn67}).  He and Szeptycki then developed the theory of subcartesian spaces, and applied it to Bessel potentials (see \cite{aronszajn75}, \cite{aronszajn80}).  The classical definition for a subcartesian space involves generalising the notion of an atlas, but it can be shown that the resulting category is isomorphic to a full subcategory of differential spaces, whose objects have a much simpler definition.  We prove this isomorphism in the appendix (\hypref{t:subcart}{Theorem}).\\

\begin{definition}[Differential Space]
Let $X$ be a nonempty set.  A \emph{differential structure}, sometimes called a \emph{Sikorski structure}, on $X$ is a nonempty family $\mathcal{F}$ of functions into $\RR$, along with the weakest topology on $X$ for which every element of $\mathcal{F}$ is continuous, satisfying:
\begin{enumerate}
\item (Smooth Compatibility) For any positive integer $k$, functions $f_1,...,f_k\in\mathcal{F}$, and $F\in\CIN(\RR^k)$, the composition $F(f_1,...,f_k)$ is contained in $\mathcal{F}$.
\item (Locality) Let $f:X\to\RR$ be a function such that for any $x\in X$ there exist an open neighbourhood $U\subseteq X$ of $x$ and a function $g\in\mathcal{F}$ satisfying $f|_U=g|_U$.  Then $f\in\mathcal{F}$.
\end{enumerate}
A set $X$ equipped with a differential structure $\mathcal{F}$ is called a \emph{differential space}, or a \emph{Sikorski space}, and is denoted $(X,\mathcal{F})$.
\end{definition}

\begin{remark}\labell{r:diffstr}
\noindent
\begin{enumerate}
\item Let $X$ be a set and $\mathcal{F}$ a family of real-valued functions on it.  We will call the weakest topology on $X$ such that $\mathcal{F}$ is a set of continuous functions the \emph{topology induced} or \emph{generated} by $\mathcal{F}$, and denote it by $\mathcal{T}_\mathcal{F}$.  A subbasis for this topology is given by $$\{f^{-1}(I)~|~f\in\mathcal{F},~I\text{ is an open interval in $\RR$}\}.$$  In the case that $\mathcal{F}$ is a differential structure, by smooth compatibility and the facts that translation and rescaling are smooth, the subbasis is equal to $$\{f^{-1}((0,1))~|~f\in\mathcal{F}\}.$$  We will often refer to this as the \emph{subbasis induced} or \emph{generated} by $\mathcal{F}$.  Also, the basis comprised of finite intersections of elements of this subbasis we will refer to as the \emph{basis induced} or \emph{generated} by $\mathcal{F}$.
\item The smooth compatibility condition of a differential structure guarantees that $\mathcal{F}$ is a commutative $\RR$-algebra under pointwise addition and multiplication.
\item The locality condition indicates that a differential structure $\mathcal{F}$ on $X$ induces a sheaf of functions: for any open subset $U$ of $X$, define $\mathcal{F}(U)$ to be all functions $f:U\to\RR$ such that if $x\in U$, then there exist an open neighbourhood $V\subseteq U$ of $x$ and a function $g\in\mathcal{F}$ such that $$g|_V=f|_V.$$
\end{enumerate}
\end{remark}

\begin{example}[Manifolds]
A manifold $M$ comes equipped with the differential structure given by its smooth functions $\CIN(M)$.
\end{example}

\begin{definition}[Functionally Smooth Maps]
Let $(X,\mathcal{F}_X)$ and $(Y,\mathcal{F}_Y)$ be two differential spaces.  A map $F:X\to Y$ is \emph{functionally smooth} if $F^*\mathcal{F}_Y\subseteq\mathcal{F}_X$.  $F$ is called a \emph{functional diffeomorphism} if it is functionally smooth and has a functionally smooth inverse.  Denote the set of functionally smooth maps between $X$ and $Y$ by $\mathcal{F}(X,Y)$.
\end{definition}

\begin{remark}
Note that in the literature a functionally smooth map is sometimes called a \emph{Sikorski smooth} map; for example, in \cite{stacey10}.
\end{remark}

\begin{remark}
A functionally smooth map is continuous with respect to the topologies induced by the differential structures.
\end{remark}

\begin{example}[Smooth Maps Between Manifolds]
Given two manifolds $M$ and $N$, the functionally smooth maps between $M$ and $N$ are exactly the usual smooth maps $\CIN(M,N)$.
\end{example}

\begin{remark}
Differential spaces along with functionally smooth maps form a category.
\end{remark}

Let $X$ be a set, and let $\mathcal{Q}$ be a family of real-valued functions on $X$.  Equip $X$ with the topology induced by $\mathcal{Q}$.  Define a family $\mathcal{F}$ of real-valued functions on $X$ as follows. $f\in\mathcal{F}$ if for any $x\in X$ there exists an open neighbourhood $U\subseteq X$ of $x$, functions $q_1,...,q_k\in\mathcal{Q}$, and a function $F\in\CIN(\RR^k)$ satisfying $$f|_U=F(q_1,...,q_k)|_U.$$

\begin{lemma}\labell{l:topologies}
The two topologies $\mathcal{T}_\mathcal{Q}$ and $\mathcal{T}_\mathcal{F}$ are equal.
\end{lemma}

\begin{proof}
Since $\mathcal{Q}\subseteq\mathcal{F}$, we have that the subbasis induced by $\mathcal{Q}$ is contained in the subbasis induced by $\mathcal{F}$, and so $\mathcal{T}_\mathcal{Q}\subseteq\mathcal{T}_\mathcal{F}$.  We now wish to show the opposite containment.\\

Fix $f\in\mathcal{F}$ and $x\in X$.  Let $I\subset\RR$ be an open interval containing $f(x)$.  We wish to find a set $W\in\mathcal{T}_\mathcal{Q}$ containing $x$ and contained in $f^{-1}(I)$.  By definition of $\mathcal{F}$, there is some set $U\in\mathcal{T}_\mathcal{Q}$ containing $x$, functions $q_1,...,q_k\in\mathcal{Q}$, and $F\in\CIN(\RR^k)$ such that $f|_U=F(q_1,...,q_k)|_U$.  Let $y=(q_1,...,q_k)(x)$, and let $B=\prod_{i=1}^k(a_i,b_i)$ be an open box containing $y$ and contained in $F^{-1}(I)$.  Then, $(q_1,...,q_k)^{-1}(B)\cap U$ is a set contained in $f^{-1}(I)\cap U\subseteq f^{-1}(I)$.  But, $$(q_1,...,q_k)^{-1}(B)=q_1^{-1}(\pr_1(B))\cap...\cap q_k^{-1}(\pr_k(B)),$$ where $\pr_i$ is the $i$th projection. This intersection is a finite intersection of open sets in $\mathcal{T}_\mathcal{Q}$.  Hence, $(q_1,...,q_k)^{-1}(B)\cap U$ is an open set in $\mathcal{T}_\mathcal{Q}$ containing $x$ and contained in $f^{-1}(I)$.  So let $W:=(q_1,...,q_k)^{-1}(B)\cap U$.
\end{proof}

\begin{proposition}
$(X,\mathcal{F})$ is a differential space.
\end{proposition}

\begin{proof}
First, we show smooth compatibility.  Let $f_1,...,f_k\in\mathcal{F}$ and $F\in\CIN(\RR^k)$.  Then, we want to show $F(f_1,...,f_k)\in\mathcal{F}$.  Fix $x\in X$.  Then for each $i=1,...,k$ there exist an open neighbourhood $U_i$ of $x$, $q_i^1,...,q_i^{m_i}\in\mathcal{Q}$ and $F_i\in\CIN(\RR^{m_i})$ such that $f_i|_{U_i}=F_i(q_i^1,...,q_i^{m_i})|_{U_i}$.  Let $U$ be the intersection of the neighbourhoods $U_i$, which itself is an open neighbourhood of $x$.  Then, $$F(f_1,...,f_k)|_U=F(F_1(q_1^1,...,q_1^{m_1}),...,F_k(q_k^1,...,q_k^{m_k}))|_U.$$ Let $N:=m_1+...+m_k$.  Define $\tilde{F}\in\CIN(\RR^N)$ by $$\tilde{F}(x^1,...,x^N)=F(F_1(x^1,...,x^{m_1}),F_2(x^{m_1+1},...,x^{m_1+m_2}),...,F_k(x^{m_1+...+m_{k-1}+1},...,x^N)).$$ Then $$F(f_1,...,f_k)|_U=\tilde{F}(q_1^1,...q_1^{m_1},q_2^1,...q_2^{m_2},...,q_k^1,...,q_k^{m_k})|_U.$$ By definition of $\mathcal{F}$, we have $F(f_1,...,f_k)\in\mathcal{F}$.\\

Next, we show locality.  Let $f:X\to\RR$ be a function with the property that for every $x\in X$ there is an open neighbourhood $U$ of $x$ and a function $g\in\mathcal{F}$ such that $f|_U=g|_U$.  Fix $x$, and let $U$ and $g$ satisfy this property.  Shrinking $U$ if necessary, there exist $q_1,...,q_k\in\mathcal{Q}$ and $F\in\CIN(\RR^k)$ such that $g|_U=F(q_1,...,q_k)|_U$.  Hence, $f|_U=F(q_1,...,q_k)|_U$.  Since this is true at each $x\in X$, by definition, $f\in\mathcal{F}$.  This completes the proof.
\end{proof}

\begin{definition}[Generated Differential Structures]
We say that the differential structure $\mathcal{F}$ above is \emph{generated} by $\mathcal{Q}$.
\end{definition}

\begin{lemma}
Let $(X,\mathcal{F})$ be a differential space.  Then for any subset $Y\subseteq X$, the subspace topology on $Y$ is the weakest topology for which the restrictions of $\mathcal{F}$ to $Y$ are continuous.
\end{lemma}

\begin{proof}
We first set some notation.  Let $\mathcal{T}_Y$ be the subspace topology on $Y$, and let $\mathcal{G}$ be all restrictions of functions in $\mathcal{F}$ to $Y$.\\

Fix $U\in\mathcal{T}_Y$ and $x\in U$.  We will show that there exists a basic open set $W$ in $\mathcal{T}_\mathcal{G}$ such that $x\in W\subseteq U$.  By definition of the subspace topology on $Y$, there exists an open set $V\in\mathcal{T}_\mathcal{F}$ such that $$U=V\cap Y.$$ There exist $f_1,...,f_k\in\mathcal{F}$ such that $$\tilde{W}:=\bigcap_{i=1}^kf_i^{-1}((0,1))$$ is a basic open set of $X$ containing $x$ and contained in $V$.  Define $W:=\tilde{W}\cap Y$.  Then,
\begin{align*}
W=&\bigcap_{i=1}^kf_i^{-1}((0,1))\cap Y\\
=&\bigcap_{i=1}^k(f_i|_Y)^{-1}((0,1)).
\end{align*}
But $f_i|_Y\in\mathcal{G}$, and so $W$ is a basic open set in $\mathcal{T}_{\mathcal{G}}$ that contains $x$ and is contained in $U$.\\

Next, we show that for any $U\in\mathcal{T}_\mathcal{G}$, $U$ is in fact in the subspace topology.  It is sufficient to show this for any basic open set $U$, in the basis generated by $\mathcal{G}$.  To this end, fix a basic open set $U\in\mathcal{T}_\mathcal{G}$ and $x\in U$.  There exist $g_1,...,g_k\in\mathcal{G}$ such that $$U=\bigcap_{i=1}^kg_i^{-1}((0,1)).$$  But then there exist $f_1,...,f_k\in\mathcal{F}$ such that for each $i=1,...,k$ we have $g_i=f_i|_Y$.  Then, $$U=\bigcap_{i=1}^kf_i^{-1}((0,1))\cap Y.$$  Since $\bigcap_{i=1}^kf_i^{-1}((0,1))$ is open on $X$, we have that $U$ is open in the subspace topology on $Y$.  We have shown that the subspace topology on $Y$ and the topology generated by restrictions of functions in $\mathcal{F}$ to $Y$ are one and the same.
\end{proof}

The above lemma allows us to make the following definition.

\begin{definition}[Differential Subspace]
Let $(X,\mathcal{F})$ be a differential space, and let $Y\subseteq X$ be any subset.  Then $Y$, with the subspace topology, acquires a differential structure $\mathcal{F}_Y$ generated by restrictions of functions in $\mathcal{F}$ to $Y$.  That is, $f\in\mathcal{F}_Y$ if and only if for every $x\in Y$ there is an open neighbourhood $U\subseteq X$ of $x$ and a function $\tilde{f}\in\mathcal{F}$ such that $$f|_{U\cap Y}=\tilde{f}|_{U\cap Y}.$$  We call $(Y,\mathcal{F}_Y)$ a \emph{differential subspace} of $X$.
\end{definition}

\begin{definition}[Product Differential Structure]
Let $(X,\mathcal{F})$ and $(Y,\mathcal{G})$ be two differential spaces.  The \emph{product differential space} $(X\times Y,\mathcal{F}\times\mathcal{G})$ is given by the set $X\times Y$ equipped with the differential structure $\mathcal{F}\times\mathcal{G}$, generated by functions of the form $f\circ\pr_X$ for $f\in\mathcal{F}$, and $g\circ\pr_Y$ for $g\in\mathcal{G}$.  Here, $\pr_X$ and $\pr_Y$ are the projections onto $X$ and $Y$, respectively.  In particular, the projection maps are functionally smooth.
\end{definition}

\begin{definition}[Quotient Differential Structure]
Let $(X,\mathcal{F})$ be a differential space, and let $\sim$ be an equivalence relation on $X$.  Then $X/\!\sim$ obtains a differential structure, called the \emph{quotient differential structure}, $\mathcal{G}=\{f:X/\!\sim\;\to\RR~|~\pi^*f\in\mathcal{F}\}$ where $\pi:X\to X/\!\sim$ is the quotient map.
\end{definition}

\begin{remark}
The quotient map $\pi$ above is smooth by definition.  Also,  we do not endow the set $X/\!\sim$ above with the quotient topology.  In general, the topology on $X/\!\sim$ induced by  $\mathcal{G}$ and the quotient topology do not match (see the following example).  In fact, the induced topology is contained in the quotient topology.
\end{remark}

\begin{example}[Quotient Topology Does Not Work]
We give an example to illustrate the issue with topologies mentioned in the above remark.  Consider the quotient space $\RR/I$, where $I$ is the open interval $(0,1)$.  By definition of the quotient topology, letting $\pi$ be the quotient map, we have that $\pi((0,1))$ is a one-point set that is open. $f$ is in the quotient differential structure if its pullback by $\pi$ is in $\CIN(\RR)$.  In this case, $\pi^*f$ is constant on $(0,1)$. But since level sets are closed, we have that $\pi^*f$ is constant on $[0,1]$.  Thus, $f$ is constant on the three-point set $\{\pi(0),\pi(1)\}\cup\pi((0,1))$.  Thus, the pre-image of any open interval of $\RR$ by any function in the quotient differential structure will never be included in the one-point set $\pi((0,1))$.  Thus, the quotient topology is strictly stronger than the topology induced by the quotient differential structure.
\end{example}

\begin{definition}[Subcartesian Space]\labell{d:subcart}
A \emph{subcartesian space} is a paracompact, second-countable, Hausdorff differential space $(S,\CIN(S))$ where for each $x\in S$ there is an open neighbourhood $U\subseteq S$ of $x$, $n\in\NN$, and a diffeomorphism $\varphi:U\to\tilde{U}\subseteq\RR^n$, called a \emph{chart}, onto a differential subspace $\tilde{U}$ of $\RR^n$.  Unless otherwise it is unclear, we shall henceforth call functionally smooth maps between subcartesian spaces simply \emph{smooth}.
\end{definition}

\begin{remark} \noindent
\begin{enumerate}
\item Subcartesian spaces, along with smooth maps between them, form a full subcategory of the category of differential spaces.
\item A subcartesian space admits smooth partitions of unity (see \cite{marshall75}).
\item For any subset $A\subseteq\RR^n$, define $\mathfrak{n}(A)$ to be the ideal of all smooth functions on $\RR^n$ whose restrictions to $A$ are identically zero.  Let $S$ be a subcartesian space. Then, for each chart $\varphi:U\to\tilde{U}\subseteq\RR^n$, the set of restrictions of functions in $\CIN(\RR^n)$ to $\tilde{U}$ is isomorphic as an $\RR$-algebra to $\CIN(\RR^n)/\mathfrak{n}(\tilde{U})$.  We thus have $\varphi^*\CIN(\RR^n)\cong\CIN(\RR^n)/\mathfrak{n}(\tilde{U})$ as $\RR$-algebras.
\end{enumerate}
\end{remark}

\begin{proposition}[Closed Differential Subspaces of Subcartesian Spaces]\labell{p:closedsubset}
If $R$ is a closed differential subspace of a subcartesian space $S$, then $\CIN(R)=\CIN(S)|_R$, the restrictions of functions in $\CIN(S)$ to $R$.
\end{proposition}

\begin{proof}
It is clear that $\CIN(S)|_R\subseteq\CIN(R)$.  To show the opposite inclusion, fix $f\in\CIN(R)$.  By definition of $\CIN(R)$, we can find an open covering $\{U_\alpha\}_{\alpha\in A}$ of $R$ such that for each $\alpha$, there is a function $g_\alpha\in\CIN(S)$ satisfying $$g_\alpha|_{U_\alpha}=f|_{U_\alpha}.$$  Let $B=\{0\}\cup A$ (assume here that $A$ does not include $0$).  For each $\alpha\in A$, let $V_\alpha$ be an open subset of $S$ such that $U_\alpha=R\cap V_\alpha$.  Let $V_0$ be the complement of $R$ in $S$ and define $g_0:=0$.  Define $\{\zeta_\beta\}_{\beta\in B}$ to be a partition of unity subordinate to $\{V_\beta\}_{\beta\in B}$.  Let $\tilde{g}:=\sum_{\beta\in B}\zeta_\beta g_\beta$.  Then
\begin{align*}
\tilde{g}|_R=&\sum_{\beta\in B}\zeta_\beta|_R g_\beta\\
=&\sum_{\beta\in B}\zeta_\beta f\\
=&f.
\end{align*}
\end{proof}

We now look at some examples in the theory of Lie group actions and symplectic geometry where such spaces arise.  Let $G$ be a compact Lie group acting smoothly on a manifold $M$, and let $\pi:M\to M/G$ be the quotient map.  Equip the geometric quotient $M/G$ with the quotient differential structure, which we denote by $\CIN(M/G)$.  Note that $\pi^*:\CIN(M/G)\to\CIN(M)^G$ is an isomorphism of $\RR$-algebras, where $\CIN(M)^G$ is the algebra of $G$-invariant smooth functions.

\begin{theorem}[Geometric Quotients are Subcartesian]\labell{t:quotsubc}
If $G$ is a compact Lie group acting on a manifold $M$, then $M/G$ is a subcartesian space whose topology matches the quotient topology induced by $\pi$.
\end{theorem}

\begin{proof}
The fact that $M/G$ equipped with the quotient differential structure is a subcartesian space is proven by Schwarz in \cite{schwarz75}.  That the quotient topology and the induced topology from $\CIN(M/G)$ are the same is shown by Cushman-\'Sniatycki in \cite{cushman-sniatycki01}.
\end{proof}

Now let $(M,\omega)$ be a connected symplectic manifold, and let $G$ be a compact Lie group acting on $M$ in a Hamiltonian fashion.  Let $\Phi:M\to\g^*$ be a ($G$-coadjoint equivariant) momentum map, define $Z:=\Phi^{-1}(0)$, and let $i:Z\to M$ be the inclusion.  Note that $Z$ is a $G$-invariant subset of $M$, and comes equipped with a differential structure $\CIN(Z)$ induced by $M$.  In particular, since $Z$ is closed, by \hypref{p:closedsubset}{Proposition} we have that $i^*\CIN(M)=\CIN(Z)$.  Consequently, $Z/G$ is a closed differential subspace of $M/G$.  We have the following commutative diagram, where $\pi_Z:=\pi|_Z$.

\begin{equation}\labell{d:main}
\xymatrix{
Z \ar[r]^{i} \ar[d]_{\pi_Z} & M \ar[d]^{\pi} \\
Z/G \ar[r]_{j} & M/G \\
}
\end{equation}

\begin{theorem}[Symplectic Quotients are Subcartesian]\labell{t:CINZ}
$Z/G$ as a subspace of $M/G$ is a subcartesian space.  Moreover, its subspace differential structure is equal to the quotient differential structure obtained from $Z$.
\end{theorem}

\begin{proof}
Note that $Z/G$ is a closed subset of $M/G$ (and hence is subcartesian), and so $\CIN(Z/G)=\CIN(M/G)|_{Z/G}$ by \hypref{p:closedsubset}{Proposition}.  We now show that $\pi_Z$ is smooth.  Let $f\in\CIN(Z/G)$.  Then there exists $g\in\CIN(M/G)$ such that $f=j^*g$.  Let $\tilde{g}=\pi^*g\in\CIN(M)^G$.  Let $\tilde{f}=i^*\tilde{g}\in\CIN(Z)^G$.  Then, $\tilde{f}=\pi_Z^*f$.\\

Next, since $\pi_Z$ is surjective, $\pi_Z^*$ is injective.  To show that $\pi_Z^*$ is surjective onto $\CIN(Z)^G$, fix $\tilde{f}\in\CIN(Z)^G$. Since $Z$ is closed, applying \hypref{p:closedsubset}{Proposition} once again, there exists $\tilde{g}\in\CIN(M)$ such that $\tilde{f}=i^*\tilde{g}$.  Averaging over $G$, we may assume that $\tilde{g}$ is $G$-invariant.  Thus, there exists $g\in\CIN(M/G)$ such that $\pi^*g=\tilde{g}$.  Thus, $f=j^*g\in\CIN(Z/G)$, and $\pi_Z^*f=\tilde{f}$.  We get that $\pi_Z^*:\CIN(Z/G)\to\CIN(Z)^G$ is an isomorphism of $\RR$-algebras.
\end{proof}

\begin{remark}
The smooth structure $\CIN(Z/G)$ is equal to a smooth structure on $Z/G$ introduced by Arms, Cushman and Gotay in \cite{acg91}.  The isomorphism $\pi_Z^*:\CIN(Z/G)\to\CIN(Z)^G$ is in fact the definition of the latter.
\end{remark}

\section{Fr\"olicher Spaces and Reflexivity}\labell{s:reflexivity}

This section is part of a joint project with Augustin Batubenge, Patrick Iglesias-Zemmour, and Yael Karshon.  Given a diffeological space, there is a natural way to construct a differential space out of it.  Conversely, given a differential space, there is a natural way to construct a diffeological space out of it.  We ask, when starting with a diffeology, and constructing a differential space out of this, and then constructing a diffeology out of this differential structure, do we obtain the same diffeology that we started with?  In general, the answer is no, but the diffeologies for which this is true form a full subcategory of diffeological spaces.  This subcategory turns out to be isomorphic to the full subcategory resulting from the same procedure applied to differential spaces.  Moreover, these subcategories turn out to be isomorphic to the category of Fr\"olicher spaces, another generalisation of smooth structures, which we define below.\\

Fix a set $X$.  Let $\mathcal{C}_0$ be a family of maps from $\RR$ into $X$, let $\mathcal{D}_0$ be a set of parametrisations into $X$ (recall that a parametrisation is a map from an open subset of a Euclidean space), and let $\mathcal{F}_0$ be a family of functions from $X$ to $\RR$.  We denote $$\Phi\mathcal{C}_0:=\{f:X\to\RR~|~\forall c\in\mathcal{C}_0,~f\circ c\in\CIN(\RR)\},$$
$$\Phi\mathcal{D}_0:=\{f:X\to\RR~|~\forall (p:U\to X)\in\mathcal{D}_0,~f\circ p\in\CIN(U)\},$$
$$\Gamma\mathcal{F}_0:=\{c:\RR\to X~|~\forall f\in\mathcal{F}_0,~f\circ c\in\CIN(\RR)\},$$ and $$\Pi\mathcal{F}_0:=\{\text{parametrisations }p:U\to X~|~\forall f\in\mathcal{F}_0,~f\circ p\in\CIN(U)\}.$$

\begin{lemma}\labell{l:inclusions}
We have the following four inclusions:
$$\mathcal{D}_0\subseteq\Pi\Phi\mathcal{D}_0,$$
$$\mathcal{F}_0\subseteq\Phi\Pi\mathcal{F}_0,$$
$$\mathcal{C}_0\subseteq\Gamma\Phi\mathcal{C}_0,$$
and
$$\mathcal{F}_0\subseteq\Phi\Gamma\mathcal{F}_0.$$
\end{lemma}

\begin{proof}
We prove the first statement.  Fix $p\in\mathcal{D}_0$.  To show that $p\in\Pi\Phi\mathcal{D}_0$, we need to show for any $f\in\Phi\mathcal{D}_0$ that $f\circ p$ is smooth.  But by definition of $\Phi\mathcal{D}_0$, we have $f\circ q$ is smooth for all $q\in\mathcal{D}_0$; in particular, $f\circ p$ is smooth.\\

We prove the second statement.  Fix $f\in\mathcal{F}_0$.  To show that $f\in\Phi\Pi\mathcal{F}_0$, we need to show for all $p\in\Pi\mathcal{F}_0$ that $f\circ p$ is smooth.  But $p\in\Pi\mathcal{F}_0$ only if $g\circ p$ is smooth for all $g\in\mathcal{F}_0$.  Thus $f\circ p$ is smooth for all such $p$.\\

We prove the third statement.  Fix $c\in\mathcal{C}_0$.  Then for all $f\in\Phi\mathcal{C}_0$, we have $f\circ c$ smooth.  But then by definition of $\Gamma$ we have that $c\in\Gamma\Phi\mathcal{C}_0$.\\

We prove the last statement.  Fix $f\in\mathcal{F}_0$.  Then for all $c\in\Gamma\mathcal{F}_0$, we have that $f\circ c$ is smooth.  But then by definition of $\Phi$ we have that $f\in\Phi\Gamma\mathcal{F}_0$.
\end{proof}

\begin{lemma}\labell{l:dftopologies}
Let $\mathcal{T}_{\mathcal{D}_0}$ be the strongest topology on $X$ such that all parametrisations in $\mathcal{D}_0$ are continuous, and let $\mathcal{T}_{\Phi\mathcal{D}_0}$ be the weakest topology on $X$ such that all functions in $\Phi\mathcal{D}_0$ are continuous.  Then $$\mathcal{T}_{\Phi\mathcal{D}_0}\subseteq\mathcal{T}_{\mathcal{D}_0}.$$ In particular, all parametrisations in $\mathcal{D}_0$ are continuous with respect to $\mathcal{T}_{\Phi\mathcal{D}_0}$.
\end{lemma}

\begin{proof}
Let $V\in\mathcal{T}_{\Phi\mathcal{D}_0}$.  Then fixing $p\in\mathcal{D}_0$, we want to show that $p^{-1}(V)$ is open in $U:=\dom(p)$.  To this end, let $u\in p^{-1}(V)$.  Then there exists an open set $W$ containing $p(u)$ and contained in $V$ of the form $$W:=\bigcap_{i=1}^kf_i^{-1}((a_i,b_i))$$ for some open intervals $(a_i,b_i)\subseteq\RR$ and functions $f_i\in\Phi\mathcal{D}_0$.  But then $$u\in p^{-1}(W)\subseteq p^{-1}(V).$$  But $$p^{-1}(W)=\bigcap_{i=1}^k(f_i\circ p)^{-1}((a_i,b_i)).$$  Since $f_i\circ p$ is a smooth function on $U$ for each $i$, we have that $p^{-1}(W)$ is a finite intersection of open subsets in $U$, and hence itself is open.  Thus $p^{-1}(V)$ is open in $U$, and we are done.
\end{proof}

\begin{lemma}\labell{l:phidiffstr}
$\Phi\mathcal{D}_0$ is a differential structure on $X$.
\end{lemma}

\begin{proof}
Equip $X$ with the weakest topology such that $\Phi\mathcal{D}_0$ is a set of continuous functions.  Then, we show the two conditions of a differential structure.  Fix $f_1,...,f_k\in\Phi\mathcal{D}_0$, and $F\in\CIN(\RR^k)$.  We want to show that $F(f_1,...,f_k)\in\Phi\mathcal{D}_0$; that is, for any parametrisation $p\in\mathcal{D}_0$, we want $F(f_1,...,f_k)\circ p$ to be smooth.  But this is the same asking that $F(f_1\circ p,...,f_k\circ p)$ be smooth, and this is true since each $f_i\circ p$ is smooth by definition of $\Phi\mathcal{D}_0$, and the composition with $F$ maintains smoothness. This shows smooth compatibility.\\

We show locality.  Let $f:X\to\RR$ be a function satisfying: for all $x\in X$ there is an open neighbourhood $V\subseteq X$ of $x$ and a function $g\in\Phi\mathcal{D}_0$ such that $f|_V=g|_V$.  We want to show that $f\in\Phi\mathcal{D}_0$.  Fix $(p:U\to X)\in\mathcal{D}_0$, and let $u\in U$.  Then there is an open neighbourhood $V\subseteq X$ of $p(u)$ and a function $g\in\Phi\mathcal{D}_0$ such that $f|_V=g|_V$.  Hence, $f\circ p|_{p^{-1}(V)}=g\circ p|_{p^{-1}(V)}$.  Now $g\circ p$ is smooth, and by \hypref{l:dftopologies}{Lemma} $p^{-1}(V)$ is open, so $g\circ p|_{p^{-1}(V)}$ is smooth, and so we have that $f\circ p$ restricted to the open set $p^{-1}(V)$ is smooth.  Since smoothness on $U$ is a local condition, and $u\in U$ is arbitrary, we have that $f\circ p$ is globally smooth, and so $f\in\Phi\mathcal{D}_0$.
\end{proof}

\begin{corollary}\labell{c:phidiffstr}
$\Phi\mathcal{C}_0$ is a differential structure on $X$.
\end{corollary}

\begin{proof}
Since $\mathcal{C}_0$ is a set of parametrisations of $X$, we simply apply \hypref{l:phidiffstr}{Lemma}.
\end{proof}

\begin{lemma}\labell{l:pidiffeol}
$\Pi\mathcal{F}_0$ is a diffeology on $X$.
\end{lemma}

\begin{proof}
We want to show that $\Pi\mathcal{F}_0$ is a diffeology.  We first check that it contains all the constants maps into $X$.  Fix $x\in X$, and let $p:U\to X$ be a constant parametrisation with image $\{x\}$.  Then for any $f\in\mathcal{F}_0$, $f\circ p$ is smooth, and so $p\in\Pi\mathcal{F}_0$.\\

Next, we want to show locality. If $p:U\to X$ is a parametrisation satisfying for every $u\in U$, there is an open neighbourhood $V\subseteq U$ of $u$ such that $p|_V\in\Pi\mathcal{F}_0$, then we want to show that $p\in\Pi\mathcal{F}_0$. But for any such $p$, we have that for any $f\in\mathcal{F}_0$ and any $u\in U$, there is an open neighbourhood $V\subseteq U$ of $u$ such that $f\circ p|_V$ is smooth.  But since smoothness on $U$ is a local condition, we have that $f\circ p$ is smooth globally, and so $p\in\Pi\mathcal{F}_0$.\\

Finally, we show smooth compatibility.  Let $U$ and $V$ be open subsets of Cartesian spaces, and let $F:V\to U$ be a smooth map.  Let $(p:U\to X)\in\Pi\mathcal{F}_0$. Then for any $f\in\mathcal{F}_0$, we have that $f\circ p\circ F$ is smooth since $f\circ p$ is.  Thus, $p\circ F\in\Pi\mathcal{F}_0$.  We have shown that $\Pi\mathcal{F}_0$ is a diffeology.
\end{proof}

\begin{definition}[Reflexive Diffeologies and Differential Spaces]
Let $\mathcal{D}$ be a diffeology on a set $X$.  We say that $\mathcal{D}$ is \emph{reflexive} if $\Pi\Phi\mathcal{D}=\mathcal{D}$.  Similarly, let $\mathcal{F}$ be a differential structure on $X$.  We say that $\mathcal{F}$ is \emph{reflexive} if $\Phi\Pi\mathcal{F}=\mathcal{F}$.
\end{definition}

\begin{proposition}[Reflexive Stability]\labell{p:reflexive}
Let $X$ be a set, and let $\mathcal{F}_0$ be a family of real-valued functions on $X$, and $\mathcal{D}_0$ a family of parametrisations on $X$.  Then,
\begin{enumerate}
\item $\mathcal{F}:=\Phi\mathcal{D}_0$ is a reflexive differential structure on $X$,
\item $\mathcal{D}:=\Pi\mathcal{F}_0$ is a reflexive diffeology on $X$.
\end{enumerate}
\end{proposition}

\begin{proof}
\begin{enumerate}
\item By \hypref{l:phidiffstr}{Lemma}, we know that $\Phi\mathcal{D}_0$ is a differential structure on $X$. We want to show that $$\Phi\Pi\Phi\mathcal{D}_0=\Phi\mathcal{D}_0.$$  By \hypref{l:inclusions}{Lemma} applied to $\Phi\mathcal{D}_0$, we have $\Phi\mathcal{D}_0\subseteq\Phi\Pi\Phi\mathcal{D}_0$.\\

    For the opposite inclusion, since $\mathcal{D}_0\subseteq\Pi\Phi\mathcal{D}_0$, we have that $\Phi\mathcal{D}_0\supseteq\Phi\Pi\Phi\mathcal{D}_0$\\

\item By \hypref{l:pidiffeol}{Lemma}, we know that $\Pi\mathcal{F}_0$ is a diffeology on $X$.  We want to show that $$\Pi\Phi\Pi\mathcal{F}_0=\Pi\mathcal{F}_0.$$ Again, by \hypref{l:inclusions}{Lemma} we have $\Pi\mathcal{F}_0\subseteq\Pi\Phi\Pi\mathcal{F}_0,$ and since $\mathcal{F}_0\subseteq\Phi\Pi\mathcal{F}_0$, we have $\Pi\mathcal{F}_0\supseteq\Pi\Phi\Pi\mathcal{F}_0$.  This completes the proof.
\end{enumerate}
\end{proof}

\begin{remark}
Reflexive diffeological spaces form a full subcategory of diffeological spaces, and reflexive differential spaces form a full subcategory of differential spaces.
\end{remark}

\begin{theorem}[Reflexive Theorem]\labell{t:reflexiveisom}
There is a natural isomorphism of categories from reflexive diffeological spaces to reflexive differential spaces.
\end{theorem}

\begin{proof}
We first define a functor $\mathbf{\Phi}$ from reflexive diffeological spaces to reflexive differential spaces as follows.  Let $(X,\mathcal{D})$ be a reflexive diffeological space.  Then define $\mathbf{\Phi}(X,\mathcal{D})$ to be $(X,\Phi\mathcal{D})$. By \hypref{p:reflexive}{Proposition}, this is a reflexive differential space.\\

Let $F:(X,\mathcal{D}_X)\to(Y,\mathcal{D}_Y)$ be a diffeologically smooth map between reflexive diffeological spaces.  Then we claim $\mathbf{\Phi}(F):=F$ is a functionally smooth map between $(X,\Phi\mathcal{D}_X)$ and $(Y,\Phi\mathcal{D}_Y)$.  Indeed, let $f\in\Phi\mathcal{D}_Y$.  We want to show that $f\circ F\in\Phi\mathcal{D}_X$; that is, for any $p\in\mathcal{D}_X$, we want $(f\circ F)\circ p$ to be smooth.  But since $f\in\Phi\mathcal{D}_Y$, and for any $p\in\mathcal{D}_X$ we have $F\circ p\in\mathcal{D}_Y$, by definition of $\Phi$ we know that $f\circ(F\circ p)$ is smooth.  It follows that $\mathbf{\Phi}$ is a functor between reflexive diffeological spaces and reflexive differential spaces.\\

Next, we define a functor $\mathbf{\Pi}$ from reflexive differential spaces to reflexive diffeological spaces as follows.  Let $(X,\mathcal{F})$ be a reflexive differential space.  Define $\mathbf{\Pi}(X,\mathcal{F})$ to be $(X,\Pi\mathcal{F})$.  By \hypref{p:reflexive}{Proposition} we know that $(X,\Pi\mathcal{F})$ is a reflexive diffeological space.\\

Let $F:(X,\mathcal{F}_X)\to(Y,\mathcal{F}_Y)$ be a functionally smooth map between reflexive differential spaces.  We claim that $\mathbf{\Pi}(F):=F$ is a diffeologically smooth map between $(X,\Pi\mathcal{F}_X)$ and $(Y,\Pi\mathcal{F}_Y)$.  Indeed, let $p\in\Pi\mathcal{F}_X$.  We want to show that $F\circ p$ is in $\Pi\mathcal{F}_Y$.  That is, for any $f\in\mathcal{F}_Y$, we want $f\circ(F\circ p)$ to be smooth.  But for any $f\in\mathcal{F}_Y$, since $F$ is functionally smooth, we have that $f\circ F\in\mathcal{F}_X$.  Since $p\in\Pi\mathcal{F}_X$, by definition of $\Pi$ we know that $(f\circ F)\circ p$ is smooth.  It follows that $\mathbf{\Pi}$ is a functor between reflexive differential spaces and reflexive diffeological spaces.\\

Finally, we need to show that $\mathbf{\Phi}$ and $\mathbf{\Pi}$ are inverses of one another.  This is clear for maps.  Let $(X,\mathcal{F})$ be a reflexive differential space.  Then
\begin{align*}
\mathbf{\Phi}\circ\mathbf{\Pi}(X,\mathcal{F})=&\mathbf{\Phi}(X,\Pi\mathcal{F})\\
=&(X,\Phi\Pi\mathcal{F})\\
=&(X,\mathcal{F}),
\end{align*}
where the last line follows from the fact that $(X,\mathcal{F})$ is reflexive.  Similarly, if $(X,\mathcal{D})$ is a reflexive diffeological space, then we get that $$\mathbf{\Pi}\circ\mathbf{\Phi}(X,\mathcal{D})=(X,\mathcal{D}).$$ Hence the two functors are inverses of one another, and we have that the two categories are isomorphic.
\end{proof}

\begin{lemma}\labell{l:phigamma}
We have the equalities $$\Phi\mathcal{C}_0=\Phi\Gamma\Phi\mathcal{C}_0$$ and $$\Gamma\mathcal{F}_0=\Gamma\Phi\Gamma\mathcal{F}_0.$$
\end{lemma}

\begin{proof}
The proof is similar to that of \hypref{p:reflexive}.

\mute{

We begin with the first equality.  Let $f\in\Phi\mathcal{C}_0$. Then, by definition of $\Gamma$, for all $c\in\Gamma\Phi\mathcal{C}_0$, we have $f\circ c$ smooth.  But then by definition of $\Phi$ we have $f\in\Phi\Gamma\Phi\mathcal{C}_0$.\\

For the opposite inclusion, let $f\in\Phi\Gamma\Phi\mathcal{C}_0$.  Then we have for all $c\in\Gamma\Phi\mathcal{C}_0$ that $f\circ c$ is smooth.  In particular, by \hypref{l:inclusions}{Lemma} we have that $f\circ c$ is smooth for all $c\in\mathcal{C}_0$.  Hence, $f\in\Phi\mathcal{C}_0$.  This establishes the first equality.\\

For the second equality, let $c\in\Gamma\mathcal{F}_0$.  Then, by definition of $\Phi$, for all $f\in\Phi\Gamma\mathcal{F}_0$, we have $f\circ c$ smooth.  But then by definition of $\Gamma$ we have that $c\in\Gamma\Phi\Gamma\mathcal{F}_0$.\\

For the opposite inclusion, let $c\in\Gamma\Phi\Gamma\mathcal{F}_0$.  Then, for all $f\in\Phi\Gamma\mathcal{F}_0$, we have that $f\circ c$ is smooth.  In particular, by \hypref{l:inclusions}{Lemma} we know that $f\circ c$ is smooth for all $f\in\mathcal{F}_0$. Hence, $c\in\Gamma\mathcal{F}_0$, which completes the proof.

%end mute
}

\end{proof}

\begin{definition}[Fr\"olicher Spaces]\labell{d:frolicher}
Fix a set $X$.  A \emph{Fr\"olicher structure} on $X$ is a family $\mathcal{F}$ of real-valued functions and a family $\mathcal{C}$ of curves $\RR\to X$ such that $\Phi\mathcal{C}=\mathcal{F}$ and $\Gamma\mathcal{F}=\mathcal{C}$.  We call the triplet $(X,\mathcal{C},\mathcal{F})$ a Fr\"olicher space.
\end{definition}

Fr\"olicher spaces were first introduced by Fr\"olicher in \cite{frolicher82}.

\begin{proposition}[Fr\"olicher Stability]
Let $X$ be a set, and let $\mathcal{F}_0$ be a family of functions on $X$, and $\mathcal{C}_0$ a family of curves into $X$.
\begin{enumerate}
\item Let $\mathcal{C}:=\Gamma\mathcal{F}_0$ and $\mathcal{F}:=\Phi\Gamma\mathcal{F}_0$.  Then $X$ equipped with $\mathcal{C}$ and $\mathcal{F}$ is a Fr\"olicher space.
\item Let $\mathcal{F}:=\Phi\mathcal{C}_0$ and $\mathcal{C}=\Gamma\Phi\mathcal{C}_0$.  Then $X$ equipped with $\mathcal{C}$ and $\mathcal{F}$ is a Fr\"olicher space.
\end{enumerate}
\end{proposition}

\begin{proof}
\begin{enumerate}
\item We need to show that $\Phi\mathcal{C}=\mathcal{F}$ and $\Gamma\mathcal{F}=\mathcal{C}$.  For the first equality $$\Phi\mathcal{C}=\Phi\Gamma\mathcal{F}_0=\mathcal{F}$$ by definition of $\mathcal{F}$.  For the second,
    \begin{align*}
    \Gamma\mathcal{F}=&\Gamma\Phi\Gamma\mathcal{F}_0\\
    =&\Gamma\mathcal{F}_0\\
    =&\mathcal{C}
    \end{align*}
    by \hypref{l:phigamma}{Lemma}.

\item We need to show that $\Gamma\mathcal{F}=\mathcal{C}$ and $\Phi\mathcal{C}=\mathcal{F}$.  For the first equality, $$\Gamma\mathcal{F}=\Gamma\Phi\mathcal{C}_0=\mathcal{C}$$ by definition of $\mathcal{C}$.  For the second equality,
    \begin{align*}
    \Phi\mathcal{C}=&\Phi\Gamma\Phi\mathcal{C}_0\\
    =&\Phi\mathcal{C}_0\\
    =&\mathcal{F}
    \end{align*}
    by \hypref{l:phigamma}{Lemma}.  This completes the proof.
\end{enumerate}
\end{proof}

\begin{definition}[Fr\"olicher Smooth Maps]
Let $(X,\mathcal{C}_X,\mathcal{F}_X)$ and $(Y,\mathcal{C}_Y,\mathcal{F}_Y)$ be Fr\"olicher spaces.  Let $F:X\to Y$ be a map.  Then $F$ is \emph{Fr\"olicher smooth} if for every $f\in\mathcal{F}_Y$, $f\circ F\in\mathcal{F}_X$.
\end{definition}

\begin{remark}
Using the same notation as in the definition above, note that for any $c\in\mathcal{C}_X$, we have $F\circ c\in\mathcal{C}_Y$.  Indeed, for every $f\in\mathcal{F}_Y$, we have $f\circ F\circ c$ is smooth since $f\circ F\in\mathcal{F}_X$.  Hence $F\circ c\in\Gamma\mathcal{F}_Y=\mathcal{C}_Y.$  Moreover, $F:X\to Y$ is Fr\"olicher smooth if and only if for any $c\in\mathcal{C}_X$, we have $F\circ c\in\mathcal{C}_Y$.
\end{remark}

Let $(X,\mathcal{D})$ be a diffeological space.  Define $$\Lambda\mathcal{D}:=\{c:\RR\to X~|~c\in\mathcal{D}\}.$$

\begin{theorem}[Boman's Theorem]\labell{t:boman}
Let $f:\RR^n\to\RR$ be a function such that for any $c\in\CIN(\RR,\RR^n)$, we have $f\circ c\in\CIN(\RR)$.  Then, $f\in\CIN(\RR^n)$.
\end{theorem}

\begin{proof}
See \cite{boman67}.
\end{proof}

\begin{lemma}\labell{l:boman}
Let $(X,\mathcal{D})$ be a reflexive diffeological space.  Then $$\Phi\Lambda\mathcal{D}=\Phi\mathcal{D}.$$
\end{lemma}

\begin{proof}
Let $f\in\Phi\mathcal{D}$.  Then for every $p\in\mathcal{D}$, we have that $f\circ p$ is smooth.  In particular, since $\Lambda\mathcal{D}\subseteq\mathcal{D}$, we have that $f\circ c$ is smooth for all $c\in\Lambda\mathcal{D}$.  Hence, $f\in\Phi\Lambda\mathcal{D}$.\\

For the opposite inclusion, fix $f\in\Phi\Lambda\mathcal{D}$, and fix $(p:U\to X)\in\mathcal{D}$.  We want to show that $f\circ p$ is smooth.  It is enough to show this locally on $U$.  Fix $u\in U$, and let $V\subseteq U$ be an open neighbourhood of $u$ admitting a diffeomorphism $\psi:\RR^n\to V$ (where $n=\dim(U)$).  If we can show that $f\circ p\circ\psi$ is smooth, then we have that $f\circ p|_V=(f\circ p\circ\psi)\circ\psi^{-1}$ is smooth, and we are done.\\

Now, $f\circ p\circ\psi$ is a function from $\RR^n$ to $\RR$.  Let $c\in\CIN(\RR,\RR^n)$.  By definition of a diffeology, we have that $p\circ\psi\circ c\in\mathcal{D}$.  In particular, $p\circ\psi\circ c\in\Lambda\mathcal{D}$.  By the definition of $\Phi$ we have that $(f\circ p\circ\psi)\circ c$ is smooth.  Since $c$ is arbitrary, by \hypref{t:boman}{Theorem}, we know that $f\circ p\circ\psi$ is smooth.  This is what we needed to show.
\end{proof}

Let $(X,\mathcal{D})$ be a reflexive diffeological space. Define $\mathbf{\Lambda}(X,\mathcal{D}):=(X,\Lambda\mathcal{D},\Phi\Lambda\mathcal{D})$ and $\mathbf{\Lambda}(F)=F$ for every diffeologically smooth map $F$ between reflexive diffeological spaces.

\begin{proposition}\labell{p:frolicher}
$\mathbf{\Lambda}$ is a functor from reflexive diffeological spaces to Fr\"olicher spaces.
\end{proposition}

\begin{proof}
We start with objects.  Let $(X,\mathcal{D})$ be a reflexive diffeological space.  We want to show that $(X,\Lambda\mathcal{D},\Phi\Lambda\mathcal{D})$ is a Fr\"olicher space.  In particular, that $\Gamma\Phi\Lambda\mathcal{D}=\Lambda\mathcal{D}$.  By \hypref{l:inclusions}{Lemma} we already know that $\Lambda\mathcal{D}\subseteq\Gamma\Phi\Lambda\mathcal{D}$.  To show the opposite inclusion, note that by \hypref{l:boman}{Lemma}, we have $\Phi\Lambda\mathcal{D}=\Phi\mathcal{D}$, and so we only need to show that $\Gamma\Phi\mathcal{D}\subseteq\Lambda\mathcal{D}$.\\

Let $c\in\Gamma\Phi\mathcal{D}$.  Then, for every $f\in\Phi\mathcal{D}$, we have that $f\circ c$ is smooth.  Hence, $c\in\Pi\Phi\mathcal{D}$.  But $\mathcal{D}$ is a reflexive diffeology, and so $\Pi\Phi\mathcal{D}=\mathcal{D}$, and $c\in\mathcal{D}$.  In particular, we have that $c\in\Lambda\mathcal{D}$.  We have shown that $\mathbf{\Lambda}$ takes reflexive diffeological spaces to Fr\"olicher spaces.\\

Next, let $(X,\mathcal{D}_X)$ and $(Y,\mathcal{D}_Y)$ be reflexive diffeological spaces and let $F:X\to Y$ be a diffeologically smooth map.  We want to show that it is a Fr\"olicher smooth map between $(X,\Lambda\mathcal{D}_X,\Phi\Lambda\mathcal{D}_X)$ and $(Y,\Lambda\mathcal{D}_Y,\Phi\Lambda\mathcal{D}_Y)$.  That is, we want to show that for any $f\in\Phi\Lambda\mathcal{D}_Y$, we have $f\circ F\in\Phi\Lambda\mathcal{D}_X$.  But by \hypref{l:boman}{Lemma} this is the same as asking that $f\circ F\in\Phi\mathcal{D}_X$.  But this we have already shown in the proof of \hypref{t:reflexiveisom}{Theorem}: diffeologically smooth maps between reflexive diffeological spaces are functionally smooth between the corresponding reflexive differential spaces.  And so we are done.
\end{proof}

Let $\mathbf{\Xi}$ be the forgetful functor from Fr\"olicher spaces to differential spaces: $\mathbf{\Xi}(X,\mathcal{C},\mathcal{F})=(X,\mathcal{F})$, and $\mathbf{\Xi}$ takes maps to themselves.  By \hypref{p:reflexive}{Proposition} $\mathbf{\Xi}$ takes Fr\"olicher spaces to reflexive differential spaces.

\begin{theorem}[Isomorphism of Categories I]\labell{t:frolicher}
$\mathbf{\Lambda}$ is an isomorphism of categories between reflexive diffeological spaces and Fr\"olicher spaces, with inverse functor given by $\mathbf{\Pi}\circ\mathbf{\Xi}$.
\end{theorem}

\begin{proof}
This is clear for maps, so we only show this for objects.  Let $(X,\mathcal{D})$ be a reflexive diffeological space.  Then, $$\mathbf{\Lambda}(X,\mathcal{D})=(X,\Lambda\mathcal{D},\Phi\Lambda\mathcal{D})=(X,\Lambda\mathcal{D},\Phi\mathcal{D})$$ is a Fr\"olicher space by \hypref{p:frolicher}{Proposition}.  $$\mathbf{\Xi}(X,\Lambda\mathcal{D},\Phi\mathcal{D})=(X,\Phi\mathcal{D})$$ is a reflexive differential space, and $$\mathbf{\Pi}(X,\Phi\mathcal{D})=(X,\Pi\Phi\mathcal{D})$$ is a reflexive diffeological space.  But since $\mathcal{D}$ is a reflexive diffeology, we have that $\Pi\Phi\mathcal{D}=\mathcal{D}$, and so we have shown that $\mathbf{\Pi}\circ\mathbf{\Xi}\circ\mathbf{\Lambda}$ is the identity functor on reflexive diffeological spaces.\\

Let $(X,\mathcal{C},\mathcal{F})$ be a Fr\"olicher space.  Then, $\mathbf{\Xi}(X,\mathcal{C},\mathcal{F})=(X,\mathcal{F})$ is a reflexive differential space.  Applying $\mathbf{\Pi}$ we get the reflexive diffeological space $(X,\Pi\mathcal{F})$.  Now,
\begin{align*}
\mathbf{\Lambda}(X,\Pi\mathcal{F})=&(X,\Lambda\Pi\mathcal{F},\Phi\Lambda\Pi\mathcal{F})\\
=&(X,\Lambda\Pi\mathcal{F},\Phi\Pi\mathcal{F}) && \text{by \hypref{l:boman}{Lemma},}\\
=&(X,\Lambda\Pi\mathcal{F},\mathcal{F}) && \text{since $\mathcal{F}$ is reflexive.}
\end{align*}
But $(X,\Lambda\Pi\mathcal{F},\mathcal{F})$ is a Fr\"olicher space, meaning that $\Gamma\mathcal{F}=\Lambda\Pi\mathcal{F}$. But we already know that $\Gamma\mathcal{F}=\mathcal{C}$, and so we end up with the Fr\"olicher space that we started with. Thus, $\mathbf{\Lambda}\circ\mathbf{\Pi}\circ\mathbf{\Xi}$ is the identity functor on Fr\"olicher spaces.
\end{proof}

\begin{corollary}[Isomorphism of Categories II]\labell{c:frolicher}
There is a natural isomorphism of categories from Fr\"olicher spaces to reflexive differential spaces, given by $\mathbf{\Xi}$.
\end{corollary}

\begin{proof}
We already know that $\mathbf{\Pi}\circ\mathbf{\Xi}$ is an isomorphism between Fr\"olicher spaces and reflexive diffeological spaces, by the above theorem.  Also, we know that $\mathbf{\Phi}$ is the inverse functor to $\mathbf{\Pi}$, by \hypref{t:reflexiveisom}{Theorem}.  Thus, $$\mathbf{\Phi}\circ\mathbf{\Pi}\circ\mathbf{\Xi}=\mathbf{\Xi}$$ is an isomorphism of categories from Fr\"olicher spaces to reflexive differential spaces.
\end{proof}

\comment{Check Stacey's paper to see exactly what he did.}

\section{Examples \& Counterexamples}\labell{s:examples}

The purpose of this section is to give examples of diffeological and differential spaces that are reflexive (and hence Fr\"olicher in principle), and also examples of spaces that are not reflexive.  Before we begin with these examples, we prove a few lemmas that will be useful when working with the examples.

\begin{lemma}[$\Pi$ Respects Subsets]\labell{l:subcartdiffeol}
Let $(X,\mathcal{F})$ be a differential space, and let $Y\subseteq X$.  Let $\mathcal{F}_Y$ be the subspace differential structure on $Y$.  Then $\Pi\mathcal{F}_Y$ is the subset diffeology on $Y$ with respect to the diffeology $\Pi\mathcal{F}$ on $X$.
\end{lemma}

\begin{proof}
Fix $p\in\Pi\mathcal{F}_Y$.  Then for any $f\in\mathcal{F}_Y$, we have that $f\circ p$ is smooth.  In particular, for any $g\in\mathcal{F}$, we have that $g|_Y\in\mathcal{F}_Y$, and so $g\circ p=g|_Y\circ p$ is smooth, and so $p\in\Pi\mathcal{F}$, and hence is a plot in the subset diffeology.\\

For the opposite inclusion, fix a plot $p$ in the subset diffeology on $Y$.  Then $p$ is a plot in $\Pi\mathcal{F}$ with image in $Y$.  Fix $f\in\mathcal{F}_Y$, let $u$ be in the domain of $p$, and let $y=p(u)$.  There is an open neighbourhood $U\subseteq X$ of $y$ (in the subspace topology) and a function $g\in\mathcal{F}$ such that $g|_{Y\cap U}=f|_{Y\cap U}$.  Without loss of generality, assume $$U=\bigcap_{i=1}^kh_i^{-1}((0,1))$$ for some $h_1,...,h_k\in\mathcal{F}$.  Note that $h_i\circ p$ smooth for each $i$.  Hence, $$V:=p^{-1}(U)=\bigcap_{i=1}^kp^{-1}(h_i^{-1}((0,1)))$$ is open in the domain of $p$, and contains $u$.  We have that $f\circ p|_V=g\circ p|_V$, which is smooth since $p|_V\in\Pi\mathcal{F}$.  Thus, $f\circ p$ is smooth, and $p\in\Pi\mathcal{F}_Y$.
\end{proof}

\begin{remark}
By the above lemma, we know that the subset diffeology on any subset of $\RR^n$ is the diffeology induced by the subspace differential structure.  Moreover, it is reflexive by \hypref{p:reflexive}{Proposition}.
\end{remark}

\begin{lemma}[$\Phi$ Respects Quotients]\labell{l:quotdiffstr}
Let $(X,\mathcal{D})$ be a diffeological space, and let $\sim$ be an equivalence relation on $X$.  Set $Y=X/\sim$ and $\pi:X\to Y$ the quotient map.  Equip $Y$ with the quotient diffeology, denoted $\mathcal{D}_Y$.  Then $\Phi\mathcal{D}_Y$ is the quotient differential structure on $Y$ induced by $\Phi\mathcal{D}$ via $\pi$, and $\pi^*:\Phi\mathcal{D}_Y\to\Phi\mathcal{D}$ is an injection.
\end{lemma}

\begin{proof}
Let $f\in\Phi\mathcal{D}_Y$.  Then, for any $p\in\mathcal{D}_Y$, we have that $f\circ p$ is smooth.  In particular, for any $q\in\mathcal{D}$, we have that $\pi\circ q\in\mathcal{D}_Y$, and so $f\circ\pi\circ q$ is smooth.  But this shows that $\pi^*f\in\Phi\mathcal{D}$, and hence $f$ is in the quotient differential structure.\\

Next, for any $f$ in the quotient differential structure, we want to show that $f\circ p$ is smooth for all $p\in\mathcal{D}_Y$.  Fix such an $f$ and $p:U\to Y$.  For any $u\in U$, there is an open neighbourhood $V\subseteq U$ of $u$ and a plot $q\in\mathcal{D}$ such that $p|_V=\pi\circ q$.  Thus, $f\circ p|_V=f\circ\pi\circ q$.  But $f\circ\pi\circ q$ is smooth by definition of the quotient differential structure, and so $f\circ p|_V$ is smooth.  Thus $f\circ p$ is smooth, and $f\in\Phi\mathcal{D}_Y$.\\

Finally, if $f\in\Phi\mathcal{D}_Y$ such that $\pi^*f=0$, then since $\pi$ is a surjection, we have that $f=0$.
\end{proof}

Let $(X,\mathcal{D})$ be a diffeological space, and let $Y\subseteq X$ be a subset equipped with the subset diffeology, denoted $\mathcal{D}_Y$.  Denote by $X^n$ the product $n$-fold product $X\times...\times X$, and similarly by $Y^n$ the $n$-fold product of $Y$.

\begin{lemma}
The product diffeology on $Y^n$ is the same as the subset diffeology on $Y^n$ as a subset of $X^n$.
\end{lemma}

\begin{proof}
Let $p:U\to Y^n$ be a plot in the product diffeology.  If $\pi_i:Y^n\to Y$ is the $i^\text{th}$ coordinate projection, then by definition of the product diffeology, $\pi_i\circ p$ is a plot on $Y$ for each $i=1,...,n$.  Hence these are plots in $\mathcal{D}_Y$; that is, they are plots in $\mathcal{D}$ with image in $Y$.  Thus, since each coordinate of $p$ is a plot of $X$, $p$ itself is a plot in the product diffeology on $X^n$ with image in $Y^n$.  Thus $p$ is contained in the subspace diffeology on $Y^n\subseteq X^n$.\\

Now, let $p:U\to Y^n\subseteq X^n$ be a plot in the subset diffeology.  Then, it is a plot on $X^n$ in the product diffeology, and so $\pi_i\circ p$ is a plot on $X$ for each $i=1,...,n$, each with image in $Y$.  Hence, $\pi_i\circ p$ is a plot in $\mathcal{D}_Y$, and so $p$ is a plot in the product diffeology on $Y^n$.
\end{proof}

\begin{example}[Manifolds with Boundary \& Corners]
Let $\KK^n$ be the \emph{positive orthant} of $\RR^n$, defined as the subset $[0,\infty)^n\subset\RR^n$.  $\KK^n$ inherits the subset diffeology from $\RR^n$, consisting of all smooth parametrisations $p:U\to\RR^n$ such that $p_i(u)\geq 0$ for all $i=1,...,n$, where $p_i(u)$ is the $i^\text{th}$ coordinate of $p(u)$.  Let $\mathcal{D}$ be the subset diffeology on $\KK^n$.

\begin{lemma}
The differential structure $\Phi\mathcal{D}$ on $\KK^n$ is exactly the subspace differential structure $\CIN(\KK^n)$ on $\KK^n\subset\RR^n$.
\end{lemma}

\begin{proof}
Fix $f\in\CIN(\KK^n)$.  Then since plots in $\mathcal{D}$ are smooth maps into $\RR^n$ with image in $\KK^n$, and $f\in\CIN(\KK^n)$ extends to a smooth function $\tilde{f}\in\CIN(\RR^n)$ by definition, we have that $f\circ p$ is smooth for all plots $p\in\mathcal{D}$.\\

Conversely, let $f\in\Phi\mathcal{D}$.  Then for any $p\in\mathcal{D}$, we have that $f\circ p$ is smooth.  We want to show that $f\in\CIN(\KK^n)$.  Now, $\KK^n$ can be identified with the orbit space obtained by the Lie group $\ZZ_2^n=\{1,-1\}^n$ acting on $\RR^n$ by $(e_1,...,e_n)\cdot (x_1,...,x_n):=(e_1x_1,...,e_nx_n)$.  Note that the polynomials $x_i^2$ are invariant under this action for each $i=1,...,n$.  Define a plot $(p:\RR^n\to\KK^n)\in\mathcal{D}$ by $$p(x_1,...,x_n):=(x_1^2,...,x_n^2).$$  We know that $f\circ p$ is smooth, and so by Schwarz (see \cite{schwarz75}), there exists a smooth function $\tilde{f}:\RR^n\to\RR$ such that $(f\circ p)(x_1,...,x_n)=\tilde{f}(x_1^2,...,x_n^2)$.  Thus, $\tilde{f}|_{\KK^n}=f$.  Hence, $f$ is the restriction of a smooth function on $\RR^n$ to $\KK^n$, and so is contained in $\CIN(\KK^n)$.
\end{proof}

\begin{remark}
The above lemma, and its proof, is a generalisation of the same statement for half-spaces $\RR^{n-1}\times[0,\infty)$ by Iglesias-Zemmour in \cite{iglesias}.  This is used to show that the (traditional) differential structure on a manifold with boundary is equal to the family of real-valued diffeologically smooth functions on the space equipped with its natural diffeology.
\end{remark}

By \hypref{l:subcartdiffeol}{Lemma} we have that $\Pi\CIN(\KK^n)$ is the subspace diffeology on $\KK^n$.  We have shown that the subspace diffeology on $\KK^n$, as well as the subspace differential structure on $\KK^n$, are reflexive.
\end{example}

\begin{example}[Geometric Quotient]
Let $G$ be a compact Lie group acting on a manifold $M$.  Equip $M/G$ with the quotient diffeology, denoted $\mathcal{D}$.  Then by \hypref{p:functions}{Proposition} and \hypref{t:quotsubc}{Theorem}, we have that $\Phi\mathcal{D}=\CIN(M/G)$.  Hence, $\CIN(M/G)$ is a reflexive differential structure on the geometric quotient.  However, it is known that the quotient diffeology on the geometric quotient $\RR^n/O(n)$ is different for each $n$ (see \cite{iglesias}, Exercise 50, page 81 with solution at the back of the book).  This gives a family of distinct diffeologies $\mathcal{D}_n$ for which $\Phi\mathcal{D}_m=\Phi\mathcal{D}_n$ for all $m,n$, and hence a family of non-reflexive diffeologies whose underlying sets are all naturally identified with the set $[0,\infty)$.
\end{example}

\begin{example}[Wedge Sum of Euclidean Spaces]\labell{x:wedgesum}
For convenience, in the following, we denote by $0_k$ the origin in $\RR^k$. Fix $k\in\NN$, and let $n_1,...,n_k\in\NN\smallsetminus\{0\}$.  Consider the wedge sum $X:=\RR^{n_1}\vee...\vee\RR^{n_k}$, where we identify the origins $0_i\in\RR^{n_i}$ for each $i$.  More precisely, $$X=(\RR^{n_1}\amalg...\amalg\RR^{n_k})/(0_i\sim 0_j, \forall i,j=1,...,k).$$  Equip $X$ with the quotient diffeology, denoted $\mathcal{D}$.  Denote the quotient map by $\pi$.  For convenience, we will denote by $\Xi$ the disjoint union of the Euclidean spaces $\RR^{n_i}$.  Note that $\CIN(\Xi)$ is the space of all $k$-tuples $(f_1,...,f_k)$ where $f_i\in\CIN(\RR^{n_i})$ for each $i$.

\begin{lemma}
$\pi^*$ is an isomorphism of $\RR$-algebras from $\Phi\mathcal{D}$ to the space of all $k$-tuples $(f_1,...,f_k)\in\CIN(\Xi)$ such that $f_i(0_i)=f_j(0_j)$ for all $i$ and $j$.
\end{lemma}

\begin{proof}
Fix $f\in\Phi\mathcal{D}$.  Consider the plot $\id_i:\RR^{n_i}\to\RR^{n_i}\subseteq \Xi$, given by the identity map on $\RR^{n_i}$.  By definition of $\mathcal{D}$, this descends to a plot in $\mathcal{D}$, and so $f\circ\pi\circ\id_i$ is smooth.  Denote this composition by $f_i$, which we identify as a function in $\CIN(\RR^n_i)$.  Then, $f_i(0_i)=f_j(0_j)$ for each $i$ and $j$, and so the $k$-tuple $(f_1,...,f_k)$ satisfies the properties desired.\\

For the opposite inclusion, fix a $k$-tuple $(f_1,...,f_k)$ of functions $f_i\in\CIN(\RR^{n_i})$ that agree at the origins of their respective domains.  Then, define $f:X\to\RR$ by $$f|_{\pi(\RR^{n_i})}=f_i\circ(\pi|_{\RR^{n_i}})^{-1}.$$  Note that $(\pi|_{\RR^{n_i}})^{-1}$ is well-defined, and $f_i(0_i)=f_j(0_j)$ for each $i$ and $j$, and so $f$ is a well-defined function on $X$.  We claim that it is in $\Phi\mathcal{D}$.\\

Let $p:U\to \Xi$ be a plot, and without loss of generality, assume that $U$ is connected.  Then, by continuity, the image of $p$ is contained in some $\RR^{n_i}$ for some $i$.  That is, $p:U\to\RR^{n_i}$ is a smooth map.  Now, $$\pi\circ p=\pi|_{\RR^{n_i}}\circ p,$$ and $f\circ\pi\circ p$ is equal to $f_i\circ\pi|_{\RR^{n_i}}^{-1}\circ\pi\circ p$, which in turn is equal to $f_i\circ p$.  This is smooth.  Since any plot in $\mathcal{D}$ locally factors through $\pi$, we have just shown that $f\in\Phi\mathcal{D}$.\\

So far, we have shown that $\pi^*$ is a well-defined surjection onto the set of $k$-tuples described.  To show that it is an injection, let $f\in\Phi\mathcal{D}$ such that $\pi^*f=0$.  Then, since $\pi^*f$ restricted to each subset $\RR^{n_i}\smallsetminus\{0_i\}$ is equal to the identity map on that subset, we have that $f$ restricted to $X\smallsetminus\{\pi(0_i)\}$ is zero.  By continuity, we have that $f(\pi(0_i))=0$, and so $f$ is zero everywhere.
\end{proof}

\begin{lemma}
$(X,\Phi\mathcal{D})$ is diffeomorphic to the subcartesian space $$S=\{(x_1,...,x_k)\in\RR^{n_1}\times...\times\RR^{n_k}~|~x_i=0\text{ for all but at most one $i=1,...,k$}\}$$ as a differential subspace.
\end{lemma}

\begin{proof}
Let $N:=n_1+...+n_k$.  Let $\varphi:X\to S$ be the natural bijection sending each point of $\pi(\RR^{n_i}\smallsetminus\{0_i\})$ to the corresponding point in $S$, and $\pi(0_i)$ to the origin sitting in $S$.  Note that for any $f\in\CIN(S)$, there exists some $g\in\CIN(\RR^N)$ such that $g|_S=f$.  Then, for any ``component'' $$S_i:=\{0_{n_1+....+n_{i-1}}\}\times\RR^{n_i}\times\{0_{n_{i+1}+...+n_k}\}\subset S$$ we have that $g|_{S_i}$ is smooth from $S_i\cong\RR^{n_i}\to\RR$.  Thus, construct the $k$-tuple $(g|_{S_1},...,g|_{S_k})$.  This satisfies the property described in the above lemma, and so is in the image of $\pi^*$.  Thus, $\varphi^*(g|_S)\in\Phi\mathcal{D}$.\\

For the opposite inclusion, let $f\in\Phi\mathcal{D}$.  We want to show that $(\varphi^{-1})^*f\in\CIN(S)$.  To this end, it is enough to show that this function extends to a smooth function on all of $\RR^N$.  Let $C:=f(\pi(0_i))$, and let $(f_1,...,f_k)=\pi^*f$.  Consider the function $$g(x_1,...,x_k)=f_1(x_1)+...+f_k(x_k)-(k-1)C.$$  $g$ is smooth on $\RR^N$, and its restriction to $S$ is exactly $f\circ\varphi^{-1}$.  This completes the proof.
\end{proof}

Now, identify $(S,\CIN(S))$ and $(X,\Phi\mathcal{D})$.  Then \hypref{l:subcartdiffeol}{Lemma} implies that $\Pi\Phi\mathcal{D}$ is the subspace diffeology on $S$.  Note that any smooth curve $c:\RR\to\RR^{n_1+...+n_k}$ with image in $S$ is a plot in $\Pi\Phi\mathcal{D}$.  But if such a curve intersects more than one component $$S_i:=\{0_{n_1+....+n_{i-1}}\}\times\RR^{n_i}\times\{0_{n_{i+1}+...+n_k}\}$$ at a point other than the origin, that is, there exists $i$ and $j$, $i\neq j$, such that $\varphi^{-1}(c(\RR))$ intersects $\pi(\RR^{n_i}\smallsetminus\{0_i\})$ and $\pi(\RR^{n_j}\smallsetminus\{0_j\})$, then such a plot cannot lift to a plot of $\Xi$, the disjoint union of the Euclidean spaces.  This is because the image of such a lift would be restricted to exactly one $\RR^{n_i}$ due to continuity, which by construction is not the case.  Thus, $\mathcal{D}$ is not reflexive, but $S$ with the subspace differential structure is by \hypref{p:reflexive}{Proposition}.
\end{example}

\begin{example}[Transversely Intersecting Submanifolds]
We next show that given two transversely intersecting submanifolds $N_1$ and $N_2$ of some ambient manifold $M$, their union $N_1\cup N_2$ equipped with the subspace differential structure is reflexive.  Note that by \hypref{l:subcartdiffeol}{Lemma}, we already know that the subspace diffeology on the union is reflexive.\\

To prove reflexivity of the subspace differential structure on the union of the two submanifolds, it is enough to check this locally about points of intersection.  Let $\dim(M)=m$, $\dim(N_1)=n_1$, and $\dim(N_2)=n_2$.  Fixing a point of intersection $x$, there exists a chart $\theta$ of $M$ defined on an open neighbourhood $U$ about $x$ into $\RR^m=\RR^{m-n_2}\times\RR^{n_1+n_2-m}\times\RR^{m-n_1}$ so that $\theta(x)=0_{m}$, $$\theta(N_1\cap U)=\RR^{m-n_2}\times\RR^{n_1+n_2-m}\times\{0_{m-n_1}\}$$ and $$\theta(N_2\cap U)=\{0_{m-n_2}\}\times\RR^{n_1+n_2-m}\times\RR^{m-n_2}.$$  Let $S\subseteq\RR^m$ be the subset $\theta((N_1\cup N_2)\cap U)$ equipped with its differential structure induced by $\RR^m$.  For brevity, we let $a=m-n_2$, $b=n_1+n_2-m$ and $c=m-n_2$.

\begin{lemma}
Let $X$ be given by the quotient $((\RR^a\times\RR^b)\amalg(\RR^c\times\RR^b))/\sim$ where $\sim$ is the equivalence relation given by the following: $(x,y)\in\RR^a\times\RR^b$ is equivalent to $(z,y')\in\RR^c\times\RR^b$ if $x=0_m$, $z=0_n$, and $y=y'$; otherwise, the remaining points of the disjoint union are equivalent to only themselves.  Equip $X$ with the quotient diffeology $\mathcal{D}$.  Then $(X,\Phi\mathcal{D})$ is functionally diffeomorphic to $(S,\CIN(S))$.
\end{lemma}

\begin{proof}
The proof is similar to what we did in the previous example.  Let $\pi:\Xi:=(\RR^a\times\RR^b)\amalg(\RR^c\times\RR^b)\to X$ be the quotient map.  First, we show that $\pi^*:\Phi\mathcal{D}\to\CIN(\Xi)$ is an isomorphism onto the set of all pairs of smooth functions $(f_1,f_2)\in\CIN(\Xi)$ such that $f_1(0_a,y)=f_2(0_c,y)$ for all $y\in\RR^b$.\\

We start with $f\in\Phi\mathcal{D}$.  Let $p:\RR^a\times\RR^b\to\RR^a\times\RR^b\subset \Xi$ be the identity onto the corresponding connected component.  This is a plot of $\Xi$, and $f_1:=\pi^*f\circ p$ is smooth, and we conclude that $\pi^*f$ restricted to the first connected component is smooth.  Likewise, so is the restriction to the second connected component.  By definition of $\pi$, $f_1(0_m,y)=f_2(0_n,y)$ for all $y\in\RR^b$.  Hence, we have shown that $\pi^*f$ satisfies the property desired.  For surjectivity of $\pi^*$, let $(f_1,f_2)$ be a pair in $\CIN(\Xi)$ satisfying $f_1(0_m,y)=f_2(0_n,y)$ for all $y$.  Then,  this descends to a function $f:X\to\RR$.  To check that it is smooth, let $q:U\to X$ be a plot.  Note that $q=\pi\circ p$ for some plot $p$ of $\Xi$.  Thus, $f\circ q=\pi^*f\circ p$.  But $\pi^*f=(f_1,f_2)$, and so $\pi^*f\circ p$ is smooth.   Finally, for injectivity, if $\pi^*f=0$, then by the surjectivity of $\pi$, $f=0$.\\

Next, we prove that $(X,\Phi\mathcal{D})$ is functionally diffeomorphic to $(S,\CIN(S))$.  Let $f\in\CIN(S)$.  Then, by definition, it extends to a smooth function $g\in\CIN(\RR^m)$.  Let $\tilde{\varphi}:\Xi\to S$ be the smooth map given by $\tilde{\varphi}(x,y)=(x,y,0_c)\in\RR^m$ for all $(x,y)\in\RR^a\times\RR^b$, and $\tilde{\varphi}(z,y)=(0_a,y,z)\in\RR^m$ for all $(z,y)\in\RR^c\times\RR^b$.  This descends to a bijection $\varphi:X\to S$.  The restriction of $g$ to $\RR^a\times\RR^b\times\{0_c\}$ is smooth; denote the pullback of this restriction via $\tilde{\varphi}$ by $f_1$.  Likewise, the restriction of $g$ to $\{0_a\}\times\RR^b\times\RR^c$ is smooth; denote its pullback by $f_2$.  Note that $f_1(0_a,y)=f_2(0_c,y)$ for all $y$, and $f_1$ and $f_2$ are smooth.  Hence, $(f_1,f_2)$ is in the image of $\pi^*$, and so $\varphi^*(g|_S)=\varphi^*f\in\Phi\mathcal{D}$.\\

Finally, let $f\in\Phi\mathcal{D}$.  Let $(f_1,f_2)=\pi^*f\in\CIN(\Xi)$.  Then, letting $\lambda(y):=f_1(0_a,y)=f_2(0_c,y)$ for all $y\in\RR^b$, consider the function $g\in\CIN(\RR^m)$ sending $(x,y,z)\in\RR^a\times\RR^b\times\RR^c$ to $g(x,y,z):=f_1(x,y)+f_2(z,y)-\lambda(y)$.  The restriction of $g$ to $S$ is smooth, and $(\varphi^{-1})^*f=g|_S$.  Thus, we have established that $\varphi$ is a functional diffeomorphism between $(X,\Phi\mathcal{D})$ and $(S,\CIN(S))$.
\end{proof}

Note, again, that $\mathcal{D}$ is not a reflexive diffeology on $X$, for similar reasons as to the quotient in the previous example.
\end{example}

\begin{example}[Three Lines in $\RR^2$]
Fix $m>0$.  Let $S$ be the differential subspace of $\RR^2$ given by $$S=\{(x,y)~|~xy(mx-y)=0\}.$$  By \hypref{l:subcartdiffeol}{Lemma}, $\Pi\CIN(S)$ is the subspace diffeology on $S$.  Let $E\subset\RR^3$ be the union of the three coordinate axes. We claim that $\Phi\Pi\CIN(S)$ is isomorphic to $\CIN(E)$.  We note that $E$ is not functionally diffeomorphic to $S$.  This is because the dimension of the Zariski tangent space at the origin in $S$ is 2, whereas that of the origin in $E$ is 3, and this number is a smooth invariant (see \hypref{s:zariski}{Section} of Chapter \ref{ch:vectorfields}).  Hence, assuming the claim, we have that $\CIN(S)$ is not a reflexive differential structure.  We also have that $\Pi\CIN(E)=\Pi\CIN(S)$ since $\Pi\CIN(S)$ is reflexive; hence, $(S,\Pi\CIN(S))$ is diffeologically isomorphic to $(E,\Pi\CIN(E))$.  We now set out to prove the claim.\\

First, recalling from \hypref{x:wedgesum}{Example}, we know that $\CIN(E)$ is isomorphic to all triplets $(f_1,f_2,f_3)$ of smooth functions in $\CIN(\RR)$ that agree at the origin.  Second, we give an explicit smooth bijection between $E$ and $S$: define $\varphi:\RR^3\to\RR^2$ by $$\varphi(x,y,z)=(x+z,y+mz).$$  Then the restriction of $\varphi$ to $E$ is smooth, and is bijective onto $S$.  For simplicity, we refer to the restriction of $\varphi$ to $E$ as $\varphi$ as well.  Let $l_1$ be the $x$-axis contained in $S$, $l_2$ the $y$-axis, and $l_3$ the remaining line given by $y=mx$.\\

Now, fix $f\in\Phi\Pi\CIN(S)$.  Then, for any $p\in\Pi\CIN(S)$, we have that $f\circ p$ is smooth.  We can choose the plot $p$ to be the smooth embedding of $l_j$ (identified with $\RR$) into $S$.  The result is that $f$ restricted to $l_j$ is smooth.  We thus get a triplet $(f_1,f_2,f_3)$, where $f_j=f|_{l_j}$, and $f_i(0)=f_j(0)$ for all $i,j$.  So, we have an inclusion $\varphi^*:\Phi\Pi\CIN(S)\hookrightarrow\CIN(E)$ which sends $f$ to the triplet $(f_1,f_2,f_3)$.\\

Now, for the opposite inclusion, fix $f\in\CIN(E)$.  Let $(f_1,f_2,f_3)$ be the corresponding triple, such that $f|_{\varphi^{-1}(l_j)}=f_j$.  Then, we want to have for any $p\in\Pi\CIN(S)$ that the composition $f\circ\varphi^{-1}\circ p$ is smooth.  By \hypref{t:boman}{Boman's Theorem} we know that $f\circ\varphi^{-1}\circ p$ is smooth if for any smooth curve $c:\RR\to\dom(p)$ we have that $f\circ\varphi^{-1}\circ p\circ c$ is smooth.  Since $p\circ c$ is a plot in $\Pi\CIN(S)$ by definition of a diffeology, and this is the subspace diffeology on $S$, it is enough to show that $f\circ\varphi^{-1}\circ c$ is smooth for all smooth curves $c:\RR\to\RR^2$ with image in $S$.\\

Fix such a curve $c$.  Let $V$ be the set of all points $x\in\RR$ such that there exists an open neighbourhood $U$ of $x$ on which $c$ has image contained entirely in one line $l_j$ for some $j=1,2,3.$  Then, $V$ is open.  Moreover for any $x\in V$, there is a neighbourhood $U$ of $x$ such that $\varphi^{-1}\circ c|_U$ is contained in one coordinate axis of $\RR^3$.  Thus, $f\circ\varphi^{-1}\circ c|_U$ is smooth and so $f\circ\varphi^{-1}\circ c|_V$ is smooth.\\

Let $W$ be the set of all points $x\in\RR$ satisfying: there exists a sequence of points $x_i\in\RR$ satisfying
\begin{enumerate}
\item $\lim x_i=x$,
\item $x_i\neq x$ for all $i$,
\item for some $\mu=1,2,3$, the sequence $c(x_{2j})$ is contained in $l_\mu\smallsetminus\{(0,0)\}$ for all $j$,
\item for some $\nu\neq\mu$, the sequence $c(x_{2j+1})$ is contained in $l_\nu\smallsetminus\{(0,0)\}$ for all $j$.
\end{enumerate}
Let $x\in W$.  Then certainly $x\notin V$, and so $W$ is a subset of the complement of $V$.  Also, if $x$ is in the complement of $W$, then there is no sequence $x_i$ satisfying the properties above.  In particular, any sequence $x_i$ with $\lim x_i=x$ has its tail contained in $l_j$ for some $j=1,2,3$.  Thus, there is an open interval $U$ about $x$ such that $c(U)\subseteq l_j$.  Thus $x\in V$.  We conclude that $W$ is equal to the complement of $V$, and hence is closed.  Moreover, if $x\in W$, then since there exist a sequence $x_{2j}$ converging to $x$ with $c(x_{2j})\in l_{\mu}$, and a sequence $x_{2j+1}$ converging to $x$ with $c(x_{2j+1})\in l_{\nu}$, then it follows that $c(x)$ is in the intersection of the closures of $l_{\mu}$ and $l_{\nu}$, which is precisely the origin.  That is, $W\subseteq c^{-1}((0,0))$.\\

The interior of $W$ is empty.  Indeed, if $x\in W$ is an interior point, then there exists an open ball $B\subset W$ about $x$, and so $c(B)=\{(0,0)\}$.  But then, all points in $B$ are sent by $c$ to the same line, and so $B\subset V$.  But $V$ is the complement of $W$, a contradiction.\\

\begin{lemma}\labell{l:curves}
Fix $x\in W$.  Then $$\frac{d^k}{dt^k}\Big|_{t=x}c(t)=(0,0)$$ for all $k\geq 0$.
\end{lemma}

Before we prove the above lemma, we apply it to the example at hand, assuming that it is true.  $f\circ\varphi^{-1}\circ c$ is continuous everywhere and smooth at points of $V$, but we also need to show that it is smooth at points of $W$.  We do this inductively: we show that $f\circ\varphi^{-1}\circ c$ is differentiable at points of $W$, and hence it is differentiable everywhere.  We then use this to show that the function is twice differentiable, and so on.  Fix $x\in W$.  Let $(x_i)\subset\RR$ be any sequence of points with limit $x$.  Choose three subsequences $(x_i^j)$ such that $c(x_i^j)\in l_j$ for each $j=1,2,3$, and the union of the subsequences is all of $(x_i)$.  (If any of these subsequences are finite, then we can restrict our attention to the tail of $(x_i)$ so as to remove that subsequence.)  Let $\tilde{f}_j:=f_j\circ\varphi^{-1}|_{l_j}$, which is a smooth map when we identify $l_j$ with $\RR$.  Then
\begin{align*}
\frac{f\circ\varphi^{-1}(c(x_i^j))-f\circ\varphi^{-1}(c(x))}{x_i^j-x}=&\frac{\tilde{f}_j(c(x_i^j))-\tilde{f}(0,0)}{x_i^j-x}\\
=&\frac{d}{dt}\Big|_{t=y_i^j}(\tilde{f}_j(c(t)))\\
=&\tilde{f}'_j(c(y_i^j))\cdot\pr_{l_j}\left(\frac{d}{dt}\Big|_{t=y_i^j}c(t)\right)
\end{align*}
for some $y_i^j$ between $x_i^j$ and $x$ by the mean-value theorem.  Here, $\pr_{l_j}$ is the projection onto $l_j$.  The factor $$\frac{d}{dt}\Big|_{t=y_i^j}c(t)$$ goes to $(0,0)$ as $i\to\infty$ by \hypref{l:curves}{Lemma}, and so the last line above vanishes.  So, the first derivative of $f\circ\varphi^{-1}\circ c$ at $x$ exists and is $0$.  Hence we have shown that $f\circ\varphi^{-1}\circ c$ is differentiable everywhere.  Note that we now can apply this argument inductively to all derivatives of $f\circ\varphi^{-1}(c(t))$ along the line $l_j$. Moreover, since this is true for each $j=1,2,3$, we have that all of the derivatives of $f\circ\varphi^{-1}\circ c$ exists at $x$ and are equal to $0$.  Thus, $f\circ\varphi^{-1}\circ c$ is smooth everywhere, and we have that $f\in\Phi\Pi\CIN(S)$.\\

We now set out to prove the lemma.

\begin{proof}
We have two cases: either $x$ is an isolated point of $c^{-1}((0,0))$, or a limit point of it.  We deal with the isolated point case first.  If $x$ is an isolated point of $c^{-1}((0,0))$, then there exists an open interval $(a,b)$ containing $x$ such that $c((a,x))\subset l_\mu\smallsetminus\{(0,0)\}$ and $c((x,b))\subset l_\nu\smallsetminus\{(0,0)\}$, $\mu\neq\nu$.  Assume $\mu=1$; similar arguments will work for the other cases.  Then, $c|_{(a,b)}$ takes on one of the following two forms: $c|_{(a,b)}(t)=(c_1(t),c_2(t))$ where $$c_1(t)=
\begin{cases}
\alpha(t) & \text{if $t\in(a,x)$}\\
0 & \text{if $t\in[x,b)$}
\end{cases}$$
and $$c_2(t)=
\begin{cases}
0 & \text{if $t\in(a,x)$}\\
\beta(t) & \text{if $t\in[x,b)$}
\end{cases}$$
where $\alpha$ and $\beta$ are smooth, or $$c_1(t)=
\begin{cases}
\alpha(t) & \text{if $t\in(a,x)$}\\
\beta(t) & \text{if $t\in[x,b)$}
\end{cases}$$
and $$c_2(t)=
\begin{cases}
0 & \text{if $t\in(a,x)$}\\
m\beta(t) & \text{if $t\in[x,b)$}
\end{cases}$$
where $\alpha$ and $\beta$ are smooth.  We need to show that in all cases, all derivatives of $\alpha$ and $\beta$ vanish at $x$.  We deal with the first form now.  $c_1$ is a smooth map, and so by continuity of each derivative, $$\limit{t\to x^-}\alpha^{(k)}(t)=\limit{t\to x}c_1^{(k)}(t)=0.$$ A similar argument holds for $\beta$.\\

For the second form, the same argument in the preceding paragraph holds for $m\beta$, and hence for $\beta$.  Since $c_1$ is smooth, we have $$\limit{t\to x^-}\alpha^{(k)}(t)=\limit{t\to x^+}\beta^{(k)}(t)$$ for all $k$.  Since the right-hand side vanishes for all $k$, we are done.  This completes this case.\\

Now assume that $x$ is a limit point of $c^{-1}((0,0))$.  Then, there exists a sequence $(x_i)\subset c^{-1}((0,0))$ such that $x_i\to x$.  But then for each $i$, we have $c(x_i)=(0,0)$.  Let $c(t)=(c_1(t),c_2(t))$.  Then, by the mean-value theorem, we have that there exists a sequence $(y^1_i)$ such that for each $i$, $y^1_i\in(x_i,x_{i+1})$ and $c'_1(y^1_i)=0.$  Since $y^1_i\to x$, we have $c'_1(x)=0$.  Moreover, we can apply a similar inductive argument to achieve that $c_1^{(k)}(x)=0$ for all $k$.  We apply this to $c_2$ as well, and this finishes the proof.
\end{proof}

\end{example}

\begin{example}[$\RR$ Modulo a Closed Interval]
Let $X=\RR/[0,1]$, equipped with the quotient diffeology, denoted $\mathcal{D}_X$.  Note that for any plot $p:U\to X$, there is a smooth map $q:U\to\RR$ such that $p=\pi_X\circ q$, where $\pi_X$ is the quotient map.  By \hypref{l:quotdiffstr}{Lemma}, $\Phi\mathcal{D}_X$ is the quotient differential structure on $X$.  Thus $f\in\Phi\mathcal{D}_X$ if and only if $\pi_X^*f$ is smooth.  But for any $f\in\Phi\mathcal{D}_X$, $\pi_X^*f$ is constant on $[0,1]$.  Also, for any $g\in\CIN(\RR)$ that is constant on $[0,1]$, $g$ descends to a well-defined function $\tilde{g}$ on $X$.  For any $p\in\mathcal{D}_X$, there is a smooth map $q:U\to\RR$ satisfying $p=\pi_X\circ q$, and $$\tilde{g}\circ p=\tilde{g}\circ\pi_X\circ q=g\circ q,$$ which is smooth, and so $\tilde{g}\in\Phi\mathcal{D}_X$.  Thus, $\pi_X^*(\Phi\mathcal{D}_X)$ is exactly the smooth functions on $\RR$ that are constant on $[0,1]$.\\

Now, consider the differential space $(\RR,\mathcal{F})$ where $\mathcal{F}$ is all smooth functions whose derivatives vanish at $0$.  Define $\varphi:X\to\RR$ by $$\varphi(\pi_X(x))=
\begin{cases}
x & \text{if $x<0$,}\\
0 & \text{if $x\in[0,1]$,}\\
x-1 & \text{if $x>1$.}
\end{cases}$$

\begin{lemma}
$\varphi$ is a functional diffeomorphism between $(X,\Phi\mathcal{D}_X)$ and $(\RR,\mathcal{F})$.
\end{lemma}

\begin{proof}
First, $\varphi$ is well-defined and a homeomorphism ($X$ being equipped with the quotient topology).  We need to show that it is functionally smooth with functionally smooth inverse.  Fix $f\in\mathcal{F}$.  Then $\pi_X^*\varphi^*f$ is smooth on $\RR\smallsetminus[0,1]$, constant on $[0,1]$, and continuous everywhere.  If we show that it is also smooth at $0$ and $1$, then it will be smooth everywhere, and hence be in the set $\pi_X^*(\Phi\mathcal{D}_X)$.  Since $\pi_X^*$ is an injection, this would imply that $\varphi^*f\in\Phi\mathcal{D}_X$.\\

We show now that $\pi_X^*\varphi^*f$ is smooth at $x=0$.  We start with the left limit.
\begin{align*}
&\limit{x\to0^-}\frac{f\circ\varphi\circ\pi_X(x)-f\circ\varphi\circ\pi_X(0)}{x-0}\\
=&\limit{x\to0^-}\frac{f(x)-f(0)}{x-0}\\
=&f'(0)=0.
\end{align*}
Now for the right limit.
\begin{align*}
&\limit{x\to0^+}\frac{f\circ\varphi\circ\pi_X(x)-f\circ\varphi\circ\pi_X(0)}{x-0}\\
=&\limit{x\to0^+}\frac{f(0)-f(0)}{x-0}\\
=&0.
\end{align*}
Since these two limits agree with $f'(0)$, this proves differentiability at $0$.  Now, assume that the $k^\text{th}$ derivative of $f\circ\varphi\circ\pi_X$ at $0$ exists and is equal to $0$.  Then,
\begin{align*}
&\limit{x\to0^-}\frac{\frac{d^k}{dx^k}\Big|_{t=x}(f\circ\varphi\circ\pi_X(t))-\frac{d^k}{dx^k}\Big|_{t=0}(f\circ\varphi\circ\pi_X(t))}{x-0}\\
=&\limit{x\to0^-}\frac{\frac{d^{k-1}}{dx^{k-1}}\Big|_{t=x}(f'(\varphi(\pi_X(t)))\cdot1)-0}{x-0}\\
=&\limit{x\to0^-}\frac{f^{(k)}(x)-f^{(k)}(0)}{x-0}\\
=&f^{(k+1)}(0)=0.
\end{align*}
Also,
\begin{align*}
&\limit{x\to0^+}\frac{\frac{d^k}{dx^k}\Big|_{t=x}(f\circ\varphi\circ\pi_X(t))-\frac{d^k}{dx^k}\Big|_{t=0}(f\circ\varphi\circ\pi_X(t))}{x-0}\\
=&\limit{x\to0^+}\frac{\frac{d^{k-1}}{dx^{k-1}}\Big|_{t=x}(f'(\varphi(\pi_X(t)))\cdot0)-0}{x-0}\\
=&0.
\end{align*}
Thus, the $k^\text{th}$ derivative of $f\circ\varphi\circ\pi_X$ exists and is equal to $0$.  By mathematical induction, we have smoothness at $0$ with all derivatives vanishing there.  A similar argument holds for $x=1$, and we have that $\varphi$ is functionally smooth.\\

Next, let $g\in\Phi\mathcal{D}_X$.  We want to show that $g\circ\varphi^{-1}$ is smooth on $\RR$ with all derivatives vanishing at $0$.  We know that $g\circ\varphi^{-1}|_{\RR\smallsetminus\{0\}}$ is smooth, so we only need to check $x=0$.  We again check left and right limits.
\begin{align*}
&\limit{x\to0^-}\frac{g\circ\varphi^{-1}(x)-g\circ\varphi^{-1}(0)}{x-0}\\
=&\limit{x\to0^-}\frac{g\circ\pi_X(x)-g\circ\pi_X(0)}{x-0}\\
=&\frac{d}{dx}\Big|_{x=0}(\pi_X^*g)(x)=0.
\end{align*}

\begin{align*}
&\limit{x\to0^+}\frac{g\circ\varphi^{-1}(x)-g\circ\varphi^{-1}(0)}{x-0}\\
=&\limit{x\to0^+}\frac{g\circ\pi_X(x+1)-g\circ\pi_X(1)}{(x+1)-1}\\
=&\frac{d}{dx}\Big|_{x=1}(\pi_X^*g)(x)=0.
\end{align*}
Thus, $g\circ\varphi^{-1}$ is differentiable at $0$.  We apply a similar induction argument as the one above, and obtain smoothness at $0$ with all derivatives vanishing there.  Thus $g\circ\varphi^{-1}\in\mathcal{F}$.
\end{proof}

Finally, consider the map $p:\RR\to\RR$ defined as $$p(x)=
\begin{cases}
x & \text{if $x<0$},\\
2x & \text{if $2x\geq0$}.
\end{cases}$$
For any $f\in\mathcal{F}$, $f\circ p$ is continuous, and smooth on $\RR\smallsetminus\{0\}$.  We claim that it is also smooth at $0$.  Indeed,
\begin{align*}
&\limit{x\to0^-}\frac{f\circ p(x)-f\circ p(0)}{x-0}\\
=&\limit{x\to0^-}\frac{f(x)-f(0)}{x-0}\\
=&f'(0)=0,
\end{align*}
and
\begin{align*}
&\limit{x\to0^+}\frac{f\circ p(x)-f\circ p(0)}{x-0}\\
=&\limit{x\to0^+}\frac{f(2x)-f(2\cdot0)}{x-0}\\
=&2f'(0)=0.
\end{align*}
So, we have differentiability.  Now, assume that $f\circ p$ is $k^\text{th}$ differentiable at $0$, with the $k^\text{th}$ derivative equal to $0$.
\begin{align*}
&\limit{x\to0^-}\frac{\frac{d^k}{dx^k}\Big|_{t=x}(f\circ p(t))-\frac{d^k}{dx^k}\Big|_{t=0}f\circ p(t)}{x-0}\\
=&\limit{x\to0^-}\frac{\frac{d^{k-1}}{dx^{k-1}}\Big|_{t=x}(f'(p(t))\cdot1)-0}{x-0}\\
=&\limit{x\to0^-}\frac{f^{(k)}(x)-f^{(k)}(0)}{x-0}\\
=&f^{(k+1)}(0)=0.
\end{align*}

\begin{align*}
&\limit{x\to0^+}\frac{\frac{d^k}{dx^k}\Big|_{t=x}(f\circ p(t))-\frac{d^k}{dx^k}\Big|_{t=0}f\circ p(t)}{x-0}\\
=&\limit{x\to0^+}\frac{\frac{d^{k-1}}{dx^{k-1}}\Big|_{t=x}(f'(p(t))\cdot2)-0}{x-0}\\
=&\limit{x\to0^+}\frac{2^kf^{(k)}(2x)-2^kf^{(k)}(0)}{x-0}\\
=&2^{k+1}f^{(k+1)}(0)=0.
\end{align*}
By induction, the above two limits show that all derivatives of $f\circ p$ exist at $0$ and vanish there.  Thus $f\circ p$ is smooth, and so $p\in\Pi\mathcal{F}$.  However, $\varphi^{-1}\circ p$ does not lift to a smooth map $q:\RR\to\RR$, so $\varphi^{-1}$ is not diffeologically smooth with respect to $\Pi\mathcal{F}$ and $\mathcal{D}_X$.  We conclude that $\mathcal{D}_X$ is not reflexive.
\end{example}

\begin{example}[$\RR$ Modulo an Open Interval]
Let $Y:=\RR/(0,1)$.  Equip $Y$ with the quotient diffeology, denoted $\mathcal{D}_Y$, and let $\pi_Y:\RR\to Y$ be the quotient map.  Note that the quotient topology on $Y$ is non-Hausdorff: the point $\pi_Y((0,1))$ is open, and its closure is $\pi_Y([0,1])$.  A plot $p:U\to Y$ lifts to a plot $q:U\to\RR$; that is, $p=\pi_Y\circ q$ always for some smooth map $q$ into $\RR$.  Let $X$ be the space $\RR/[0,1]$ as in the previous example.  Define the map $H:Y\to X$ as $H(x)=x$ for all $x\notin\pi_Y([0,1])$, $H(x)=\pi_X([0,1])$ for all $x\in\pi_Y([0,1])$.  We have the following commutative diagram.

$$\xymatrix{
& \RR \ar[dl]_{\pi_X} \ar[dr]^{\pi_Y} & \\
X & & Y \ar[ll]^{H} \\
}$$

By \hypref{l:quotdiffstr}{Lemma}, $\Phi\mathcal{D}_Y$ is the quotient differential structure on $Y$, and so we know that $\pi_Y^*f$ is smooth for all $f\in\Phi\mathcal{D}_Y$.  Moreover, $\pi_Y^*f$ is constant on $[0,1]$.  Also, any smooth function constant on $[0,1]$ descends to a smooth function in the quotient differential structure, and hence in $\Phi\mathcal{D}_Y$.  Fix $g\in\Phi\mathcal{D}_X$.  Then, $\pi_X^*g$ is constant on $[0,1]$ and is smooth, and so descends to a (unique) function $f\in\Phi\mathcal{D}_Y$.  Hence, $H^*$ is well-defined onto $\Phi\mathcal{D}_Y$, and $H$ is smooth.  Also, fix $f\in\Phi\mathcal{D}_Y$.  Then, $\pi_Y^*f$ is in the image of $\pi_X^*$, and so there exists a function $g\in\Phi\mathcal{D}_X$ such that $H^*g=f$.  Hence, $H^*$ is surjective.  If $0=H^*g$, then $g=0$ since $H$ is surjective, and we have that $H^*$ is injective.  In fact, it is an isomorphism between the differential structures $\Phi\mathcal{D}_X$ and $\Phi\mathcal{D}_Y$, even though $H$ is not a functional diffeomorphism, as it is not bijective.\\

Let $p\in\mathcal{D}_Y$.  Then, there is some plot $q$ on $\RR$ such that $p=\pi_Y\circ q$.  But then $$H\circ p=H\circ\pi_Y\circ q=\pi_X\circ q,$$ and so $H$ is also diffeologically smooth.  However, $H$ does not provide a bijection between $\mathcal{D}_Y$ and $\mathcal{D}_X$; indeed, the constant plots into $\pi_Y(0)$, $\pi_Y((0,1))$ and $\pi_Y(1)$ are all mapped to the same plot by $H$.\\

Now, let $p\in\Pi\Phi\mathcal{D}_Y$.  Fix $g\in\Phi\mathcal{D}_X$.  Then since $g\circ H$ is in $\Phi\mathcal{D}_Y$, we have that $g\circ H\circ p$ is smooth. Since $g$ is arbitrary, we have that $H\circ p$ is in $\Pi\Phi\mathcal{D}_X$.  Thus $H$ is also diffeologically smooth between $\Pi\Phi\mathcal{D}_Y$ and $\Pi\Phi\mathcal{D}_X$ as well.\\

Consider the discontinuous map $q:\RR\to\RR$ defined by $$q(x)=
\begin{cases}
x-\frac{1}{2} & \text{if $x\leq\frac{1}{2}$},\\
x+\frac{1}{2} & \text{if $x>\frac{1}{2}$}.
\end{cases}$$
We claim that $\pi_Y\circ q$ is a plot in $\Pi\Phi\mathcal{D}_Y$. We need only to check that $f\circ\pi_Y\circ q$ is smooth at $\frac{1}{2}$ for any $f\in\Phi\mathcal{D}_Y$.  Checking left and right limits will accomplish this.

\begin{align*}
&\limit{x\to\frac{1}{2}^-}\frac{f\circ\pi_Y\circ q(t)-f\circ\pi_Y\circ q(1/2)}{x-1/2}\\
=&\limit{x\to\frac{1}{2}^-}\frac{(\pi_Y^*f)(x-1/2)-(\pi_Y^*f)(0)}{x-1/2}\\
=&\frac{d}{dx}\Big|_{x=0}(\pi_Y^*f)(x)=0,
\end{align*}
using the fact that $\pi_Y^*f$ is constant on $[0,1]$.  As well,
\begin{align*}
&\limit{x\to\frac{1}{2}^+}\frac{f\circ\pi_Y\circ q(t)-f\circ\pi_Y\circ q(1/2)}{x-1/2}\\
=&\limit{x\to\frac{1}{2}^+}\frac{(\pi_Y^*f)(x+1/2)-(\pi_Y^*f)(0)}{x-1/2}\\
=&\limit{x\to\frac{1}{2}^+}\frac{(\pi_Y^*f)(x+1/2)-(\pi_Y^*f)(1)}{x-1/2}\\
=&\frac{d}{dx}\Big|_{x=1}(\pi_Y^*f)(x)=0.
\end{align*}
Thus, $f\circ\pi_Y\circ q$ is differentiable at $x=1/2$, and the derivative vanishes there.  Now, assuming that the $k^\text{th}$ derivative exists and vanishes as well at $1/2$,
\begin{align*}
&\limit{x\to\frac{1}{2}^-}\frac{\frac{d^k}{dx^k}\Big|_{t=x}\left(f\circ\pi_Y\circ q(t)\right)-\frac{d^k}{dx^k}\Big|_{t=1/2}\left(f\circ\pi_Y\circ q(t)\right)}{x-1/2}\\
=&\limit{x\to\frac{1}{2}^-}\frac{\frac{d^{k-1}}{dx^{k-1}}\Big|_{t=x}\left(\frac{d}{dx}\Big|_{u=t-1/2}(\pi_Y^*f(u))\cdot1\right)-0}{x-1/2}\\
=&\limit{x\to\frac{1}{2}^-}\frac{\frac{d^k}{dx^k}\Big|_{t=x}(\pi_Y^*f(t-1/2))-\frac{d^k}{dx^k}\Big|_{t=0}(\pi_Y^*f)(t)}{x-1/2}\\
=&\frac{d^{k+1}}{dx^{k+1}}\Big|_{x=0}(\pi_Y^*f)(x)=0,
\end{align*}
and a similar argument works for the right limit.  Thus, $f\circ\pi_Y\circ q$ is smooth, and so $\pi_Y\circ q\in\Pi\Phi\mathcal{D}_Y$.  However, $\pi_Y\circ q$ has no lift to a continuous map into $\RR$, let alone smooth.  Thus, $\mathcal{D}_Y$ is not a reflexive diffeology.
\end{example}

\begin{example}[The Rational Numbers]
Consider $\QQ\subset\RR$ with its subspace differential structure.  By definition, any function in this differential structure locally extends to a smooth function on $\RR$.  Note that this includes the restriction of functions such as $$f(x)=\frac{1}{x-\sqrt{2}}.$$  Now, $\Pi\CIN(\QQ)$ consists solely of constant maps into $\QQ$.  And so $\Phi\Pi\CIN(\QQ)$ is the set of \emph{all} functions $f:\QQ\to\RR$, as $f\circ p$ is constant for all $p\in\Pi\CIN(\QQ)$.  Thus $\CIN(\QQ)$ is not reflexive.
\end{example}

\begin{example}[The Irrational Torus]
Fix an irrational number $\alpha$.  Let $X$ be the quotient $\RR/(\ZZ+\alpha\ZZ)$, equipped with its quotient diffeology $\mathcal{D}$.  This is the set of equivalence classes where $x\sim y$ if there exist integers $m$ and $n$ such that $x=y+m+\alpha n$.  Then $\Phi\mathcal{D}$ is equal to all constant functions: $\Phi\mathcal{D}\cong\RR$.  Hence, $\Pi\Phi\mathcal{D}$ is the set of \emph{all} maps into $X$, as the composition of such a map with any of the functions in $\Phi\mathcal{D}$ is constant, which is smooth.  Hence $\mathcal{D}$ is not reflexive.  (See \cite{iglesias}, Exercise 4 with solution at the back of the book.)
\end{example}

\begin{example}[$(\RR,C^k(\RR))$]
Consider the real line $\RR$ equipped with the differential structure consisting of all $k$-differentiable functions $C^k(\RR)$ with respect to the Euclidean topology on $\RR$ ($k$ finite).  We claim that this space is not reflexive.  Let $c\in\Gamma C^k(\RR)$.  Then, $c$ is a smooth map into $\RR$ in the usual sense since the identity map is in $C^k(\RR)$.\\

Assume that $c\in\Gamma C^k(\RR)$ is a non-constant curve with domain an open interval $I\subseteq\RR$.  Then there is some $t\in I$ such that $c'(t)\neq0$.  But then there is an open interval $J\subseteq I$ about $t$ such that $c|_J$ is a diffeomorphism.  Take any $f\in C^k(\RR)$ that is not smooth on $c(J)$.  Then $f\circ c$ is not smooth on $J$, which contradicts our assumption that $c\in\Gamma C^k(\RR)$.  Thus, all curves in $\Gamma C^k(\RR)$ are constant.  But then $\Phi\Gamma C^k(\RR)$ is the family of all real-valued functions on $\RR$, and not just continuous ones.\\

Since $(\RR,\Gamma C^k(\RR),C^k(\RR))$ is not Fr\"olicher, we conclude that $(\RR, C^k(\RR))$ is not reflexive.
\end{example} 
\chapter[Forms on Geometric and Symplectic Quotients]{Diffeological Forms on Geometric and Symplectic Quotients}\labell{ch:forms}

Let $G$ be a compact Lie group, and let $P\to B$ be a principal $G$-bundle.  Then it is known (see \cite{koszul53}) that the de Rham complex of differential forms on $B$ is isomorphic to the complex of basic differential forms on $P$, defined below.  In fact, $G$ in general does not need to be compact; we only require that the action on $P$ be proper and free (see \cite{palais61}).  In the case that the Lie group is compact but the action is not necessarily free, then the quotient is not a manifold in general.  However, we show that diffeological differential forms on the quotient are isomorphic to basic differential forms upstairs.\\

In the case of a Hamiltonian group action, if $0$ is a regular value of the momentum map, and the group action is free when restricted to the zero level of the momentum map, then the quotient of the zero level, called the symplectic quotient, is known to be a symplectic manifold (see \cite{marsden-weinstein76}, \cite{meyer73}).  When $0$ is not a regular value, and/or when the action of the Lie group on the zero level is not free, then this symplectic quotient need not be a manifold.  It is shown in \cite{lerman-sjamaar91} that it is a symplectic stratified space.  In \cite{sjamaar05}, Sjamaar defines a de Rham complex on this space.  However this definition is not intrinsic to the symplectic quotient, as it depends on the Hamiltonian action on the original manifold.  In the case that $0$ is a regular value, we apply the theory described above to the symplectic quotient, and achieve an isomorphism between Sjamaar's de Rham complex, and that of diffeological differential forms.\\

We also look at reduction in stages, and show that the diffeological structures on the resulting symplectic quotients given by reduction in stages are diffeomorphic.  This allows one to use Sjamaar's de Rham complex in conjunction with reduction in stages.\\

In the last section, we compare diffeological differential forms on orbifolds to the classical definition (see \cite{satake56}, \cite{satake57}, \cite{haefliger84}, \cite{alr07}).\\

\section{The Geometric Quotient}

\begin{definition}[Differential Forms]
Let $(X,\mathcal{D})$ be a diffeological space.  A \emph{ (diffeological) differential $k$-form} $\alpha$ on $X$ is an assignment to each plot $(p:U\to X)\in\mathcal{D}$ a differential $k$-form $\alpha(p)\in\Omega^k(U)$ satisfying for every open set of Euclidean space $V$ and every smooth map $F:V\to U$ $$\alpha(p\circ F)=F^*(\alpha(p)).$$  This latter condition is called \emph{smooth compatibility}.  Denote the set of differential forms by $\Omega^k(X)$.
\end{definition}

\begin{definition}[Wedge Product]
Let $(X,\mathcal{D})$ be a diffeological space, and let $\alpha\in\Omega^k(X)$ and $\beta\in\Omega^l(X)$.  Then define the \emph{wedge product} of $\alpha$ and $\beta$, denoted $\alpha\wedge\beta$, to be the $(k+l)$-form defined by $$(\alpha\wedge\beta)(p)=\alpha(p)\wedge\beta(p)$$ for all plots $p\in\mathcal{D}$.  Then $\Omega^*(X)=\bigoplus_{k=0}^\infty\Omega^k(X)$ is an exterior algebra.
\end{definition}

\begin{definition}[Pullback Map]
Let $X$ and $Y$ be diffeological spaces, and let $F:X\to Y$ be a diffeologically smooth map between them.  Let $\alpha$ be a differential $k$-form on $Y$.  Then define the \emph{pullback} $F^*\alpha$ to be the $k$-form on $X$ satisfying: for every plot $p:U\to X$, $$(F^*\alpha)(p)=\alpha(F\circ p).$$
\end{definition}

\begin{remark}
On open subsets of Euclidean space, diffeological differential forms can be identified with the usual notion of differential forms.  Indeed, the smooth compatibility condition in the definition of a diffeological differential form is the usual transformation rule for transition functions in the traditional definition of a differential form.  Due to this identification, if $(X,\mathcal{D})$ is a diffeological space and $(p:U\to X)\in\mathcal{D}$, then for any differential form $\alpha$ on $X$, we have $\alpha(p)=p^*\alpha$.  In fact, we will use this notation henceforth instead of $\alpha(p)$.
\end{remark}

\begin{definition}[Exterior Derivative of Diffeological Forms]
Let $(X,\mathcal{D})$ be a diffeological space, and let $\alpha$ be a $k$-form on it.  Define the \emph{exterior derivative} of $\alpha$, denoted $d\alpha$, by $$p^*(d\alpha)=d(p^*\alpha)$$ for any plot $p\in\mathcal{D}$.  The exterior derivative commutes with pullback, and all of the usual formulae involving pullbacks, the exterior derivative, and the wedge product hold.  We thus have the de Rham complex $(\Omega^*(X),d)$.
\end{definition}

\begin{proposition}[Injectivity of Pullbacks via Quotient Maps]\labell{p:injectivity}
Let $X$ be a diffeological space and let $\sim$ be an equivalence relation on it.  Then the quotient map $\pi:X\to X/\!\!\sim$ induces an injection $\pi^*:\Omega^k(X/\!\!\sim)\to\Omega^k(X)$.
\end{proposition}

\begin{proof}
Assume that $\alpha=\pi^*\beta$ for some $\beta\in\Omega^k(X/\!\!\sim)$, and that $\alpha=0$.  We want to show that $\beta=0$.  Let $p:U\to X$ be a plot.  Then $0=p^*\alpha=p^*\pi^*\beta$.  But since all plots on $X/\!\!\sim$ are locally of the form $\pi\circ p$ with $p$ a plot of $X$, we see that $\beta=0$.
\end{proof}

Now, let $G$ be a Lie group acting on a manifold $M$.  A form $\alpha$ is \emph{horizontal} if for any $x\in M$ and $v\in T_x(G\cdot x)$, we have $$v\hook\alpha=0.$$  $\alpha$ is \emph{basic} if it is both $G$-invariant and horizontal.  Basic forms form a subcomplex of the de Rham complex on $M$ (see \cite{koszul53}).  By \hypref{p:injectivity}{Proposition}, we have an injective map $\pi^*:\Omega^k(M/G)\to\Omega^k(M)$.  We now strive to show that the image of this map is contained in $\Omega^k_{basic}(M)$, the subcomplex of basic forms.  To this end, we first prove some lemmas.\\

Fix $x\in M$.  Let $A_x:G\to M$ be the map given by $A_x(g)=g\cdot x$.

\begin{lemma}
$A_x$ is smooth for all $x\in M$.
\end{lemma}

\begin{proof}
We already know that the map $G\times M\to M:(g,y)\mapsto g\cdot y$ is smooth, by definition of a smooth action.  But $A_x$ is just a restriction of this map to the submanifold $G\times\{x\}$, and so it itself is smooth.
\end{proof}

Fix $x\in M$. Let $p:U\to M$ be a plot that factors as $p=A_x\circ q$ where $q:U\to G$ is a plot.  Let $j:\{*\}\to M/G$ be the map sending $*$ to $\pi(x)$.  Let $c:G\to\{*\}$ be the constant map.  Then, $j$ and $c$ are smooth, and we have the following commutative diagram of diffeologically smooth maps.
$$\xymatrix{
U \ar[r]_q \ar@/^1pc/[rr]^p & G \ar[r]_{A_x} \ar[d]^{c} & M \ar[d]_{\pi} \\
 & \{*\} \ar[r]_{j} & M/G \\
}$$

\begin{lemma}\labell{l:actionmap}
Let $\alpha=\pi^*\beta$ for some $\beta\in\Omega^k(M/G)$, where $k>0$.  Then, if $p$ is as in the lemma above, we have $p^*\alpha=0$.
\end{lemma}

\begin{proof}
This is immediate from the commutative diagram above.
\end{proof}

\begin{proposition}[Pullbacks from the Geometric Quotient are Basic]\labell{p:pullbacksarebasic}
Let $\alpha=\pi^*\beta$ for some $\beta\in\Omega^k(M/G)$.  Then $\alpha$ is basic.
\end{proposition}

\begin{proof}
By \hypref{p:functions}{Proposition}, the case where $k=0$ is already done.  Assume $k>0$.  We first show that $\alpha$ is $G$-invariant.  Let $g\in G$ and let $p$ be any plot.  Then,
\begin{align*}
p^*g^*\alpha=&p^*g^*\pi^*\beta\\
=&p^*\pi^*\beta\\
=&p^*\alpha.
\end{align*}

Next, we show that $\alpha$ is horizontal.  It is enough to prove that for any $x\in M$, the pullback of $\alpha$ to the submanifold $G\cdot x$ vanishes.  Let $H$ be the stabiliser of $x$.  Then we identify the orbit $G\cdot x$ with $G/H$.

$$\xymatrix{
G \ar[r]^{A_x} \ar[d] & M \\
G/H \ar[r]_{\cong} & G\cdot x \ar[u] \\
}$$

By \hypref{p:injectivity}{Proposition} it is sufficient to show that the pullback of $\alpha$ to $G/H$ pulled back further via the quotient map $G\to G/H$ vanishes.  But note that this is equivalent to showing that the pullback of $\alpha$ by $A_x$ vanishes.  Let $q:U\to G$ be a plot of $G$.  Then $A_x\circ q$ is a plot of $M$ since $A_x$ is smooth.  By \hypref{l:actionmap}{Lemma} $(A_x\circ q)^*\alpha$ vanishes.  This completes the proof.
\end{proof}

\section{Basic Forms and the Geometric Quotient}

We begin with some tools that we will need later on.

\begin{proposition}[Characterisation of the Image of the Pullback Map]\labell{p:pullbackimgchar}
Let $X$ be a diffeological space, $\sim$ an equivalence relation on $X$, and $\pi:X\to X/\!\!\sim$ the quotient map.  Then a differential form $\alpha$ on $M$ is in the image of $\pi^*$ if and only if, for every two plots $p_1:U\to M$ and $p_2:U\to M$ such that $\pi\circ p_1=\pi\circ p_2$, we have  $$p_1^*\alpha=p_2^*\alpha.$$
\end{proposition}

\begin{proof}
See section 6.38 of \cite{iglesias}.
\end{proof}

\begin{lemma}\labell{l:finiteunion}
Let $U\subseteq\RR^n$ be an open set, and let $C_1,...,C_k\subseteq U$ be closed subsets of $U$ such that $$U=\bigcup_{i=1}^kC_i.$$  Then the set $W:=\bigcup_{i=1}^k\inter(C_i)$ is open and dense in $U$.
\end{lemma}

\begin{proof}
It is clear that $W$ is open.  To check density, we want to show that for any $u\in U$ and any open neighbourhood $V\subseteq U$ of $u$, the intersection $W\cap V$ is nonempty. To this end, fix $u$ and $V$.  Then, $$V=\bigcup_{i=1}^k(C_i\cap V).$$  $V$ is open, and so it cannot be the finite union of nowhere dense subsets.  Hence, there is some $i=1,...,k$, such that $C_i\cap V$ is nowhere dense.  That is, the interior of the closure of $C_i$ in $V$ is nonempty.  But $C_i$ is closed, and so there is some $i=1,...,k$ such that $\inter(C_i\cap V)$ is nonempty.  Thus, $V\cap W\neq\emptyset$.
\end{proof}

Let $G$ be a compact Lie group and let $M$ be a manifold on which $G$ acts, with $\pi:M\to M/G$ the orbit map. The purpose of this section is to show that $\pi^*$ is an isomorphism of complexes from differential forms on $M/G$ to basic differential forms on $M$.  We begin by proving the result for the case of finite groups.  Note that for finite groups, basic $k$-forms are simply $G$-invariant $k$-forms, as the tangent space to a $G$-orbit at any point is trivial.

\begin{proposition}[Finite Group Case]\labell{p:finite}
Let $G$ be a finite group acting on a manifold $M$.  If $\alpha$ is an invariant $k$-form, then there is some $k$-form $\beta$ on $M/G$ such that $\pi^*\beta=\alpha$.
\end{proposition}

\begin{proof}
If $k=0$, then by \hypref{p:functions}{Proposition} we are done.  Assume $k>0$.  Let $\alpha$ be an invariant $k$-form on $M$.  By \hypref{p:pullbackimgchar}{Proposition}, it is enough to show the following: if $p_1:U\to M$ and $p_2:U\to M$ are plots such that $\pi\circ p_1=\pi\circ p_2$, then $p_1^*\alpha=p_2^*\alpha$.  Fix two such plots $p_1:U\to M$ and $p_2:U\to M$.  For each $g\in G$ let $$C_g:=\{u\in U~|~g\cdot p_1(u)=p_2(u)\}.$$  By continuity, $C_g$ is closed for each $g$.  By our assumption on $p_1$ and $p_2$, $$U=\bigcup_{g\in G}C_g.$$  By \hypref{l:finiteunion}{Lemma}, the set $\bigcup_{i=1}^k\inter(C_g)$ is open and dense in $U$.  Thus, it is enough to show $p_1^*\alpha=p_2^*\alpha$ on $\inter(C_g)$ for each $g\in G$ such that the interior of $C_g$ is nonempty.\\

Fix such a $g\in G$.  Since $g\circ p_1|_{\inter(C_g)}=p_2|_{\inter(C_g)}$, we have that $$(p_1^*g^*\alpha)|_{\inter(C_g)}=(p_2^*\alpha)|_{\inter(C_g)}.$$
But since $\alpha$ is invariant, we have $g^*\alpha=\alpha$.  Thus, $p_1^*\alpha=p_2^*\alpha$ on the open subset $\inter(C_g)$.  This completes the proof.
\end{proof}

Next, we prove a special case of when we have a group acting linearly and orthogonally on a vector space $V$.  This requires some technical lemmas, which we prove first.  \hypref{l:techn1}{Lemma}, \hypref{l:techn2}{Lemma}, and \hypref{c:techn2}{Corollary} are due to Yael Karshon (private communication).

\begin{lemma}\labell{l:techn1}
Let $G$ be a compact connected Lie group acting linearly and orthogonally on an inner product space $V$.  Let $g\in G$, $\eta\in\g$ such that $\exp(\eta)=g$, and let $v\in V$.  Then there exists $v'\in V$ such that $|v'|\leq|v|$ and $g\cdot v-v=\eta\cdot v'.$
\end{lemma}

\begin{proof}
We identify tangent spaces at points of $V$ with $V$ itself, as it is a vector space.
\begin{align*}
g\cdot v-v=& \exp(t\eta)\cdot v\big|_0^1\\
=&\int_0^1\left(\frac{d}{dt}\exp(t\eta)\cdot v\right)dt\\
=&\int_0^1\left(\eta\cdot\exp(t\eta)\cdot v\right)dt\\
=&\eta\cdot\int_0^1\left(\exp(t\eta)\cdot v\right)dt.
\end{align*}
So define $v':=\int_0^1\left(\exp(t\eta)\cdot v\right)dt.$ Finally,
\begin{align*}
|v'|=&\left|\int_0^1\left(\exp(t\eta)\cdot v\right)dt\right|\\
\leq& \int_0^1\left|\exp(t\eta)\cdot v\right|dt\\
=&\int_0^1|v|dt\\
=&|v|.
\end{align*}
The second-last line comes from the fact that the action is orthogonal.  This completes the proof.
\end{proof}

In the setting of the lemma above, let $\gamma_1$ and $\gamma_2$ be curves from $\RR$ into $V$ such that $0=\gamma_1(0)=\gamma_2(0)$, and for every $t\in\RR$ there exists $g_t\in G$ satisfying $\gamma_2(t)=g_t\cdot\gamma_1(t)$. Note here that $t\mapsto g_t$ is not necessarily continuous.  Let $\xi_1=\dot{\gamma}_1(0)$ and $\xi_2=\dot{\gamma}_2(0)$.

\begin{lemma}\labell{l:techn2}
There exists a sequence of real numbers $t_n$ converging to $0$, a sequence of vectors $v_n$ converging to $0$ in $V$, and a sequence $\mu_n$ in $\g$ such that $$\xi_2-\xi_1=\underset{n\to\infty}{\lim}\mu_n\cdot v_n.$$
\end{lemma}

\begin{proof}
Let $(t_n)$ be any sequence in $\RR$ that converges to $0$, where for each $n$, we have $t_n\neq0$.  For each $n$ let $\eta_n\in\g$ be an element of $\g$ satisfying $\exp(\eta_n)=g_{t_n}$. Since we are working on a vector space, it makes sense to subtract the curves: $\gamma_2(t)-\gamma_1(t)$.  We have
\begin{align*}
\xi_2-\xi_1=&\frac{d}{dt}\Big|_{t=0}(\gamma_2(t)-\gamma_1(t))\\
=&\underset{t\to0}{\lim}\left(\frac{\gamma_2(t)-\gamma_1(t)}{t}\right)\\
=&\underset{n\to\infty}{\lim}\left(\frac{\gamma_2(t_n)-\gamma_1(t_n)}{t_n}\right)\\
=&\underset{n\to\infty}{\lim}\left(\frac{g_{t_n}\cdot\gamma_1(t_n)-\gamma_1(t_n)}{t_n}\right)\\
=&\underset{n\to\infty}{\lim}\left(\frac{\eta_n\cdot v'_n}{t_n}\right)\\
\end{align*}
where $v_n'\in V$ satisfies $|\gamma_1(t_n)|\geq|v'_n|$ for each $n$.  The last line is a result of \hypref{l:techn1}{Lemma}.  Set $v_n=v'_n$.  From the inequality $|\gamma_1(t_n)|\geq|v_n|$ we have that $v_n\to0$ as $n\to\infty$.  Finally, set $\mu_n:=\eta_n/t_n$.  This completes the proof.
\end{proof}

\begin{corollary}\labell{c:techn2}
Let $\alpha$ be a basic form on $V$.  Then, $(\xi_2-\xi_1)\hook\alpha=0$.
\end{corollary}

\begin{proof}
Again, we identify tangent spaces at points of $V$ with $V$ itself.
\begin{align*}
(\xi_2-\xi_1)\hook\alpha|_0=&\underset{n\to\infty}{\lim}\left((\mu_n\cdot v_n)\hook(\alpha|_{v_n})\right)\\
=&\underset{n\to\infty}{\lim}\left((\mu_n)_V\hook\alpha|_{v_n}\right)
\end{align*}
where $(\mu_n)_V$ is the vector field on $V$ induced by $\mu_n\in\g$.  But recall that $\alpha$ is basic, and hence the last line above vanishes.
\end{proof}

Next, we reduce the case of a compact Lie group to a compact connected Lie group.

\begin{lemma}\labell{l:conntodisc}
Fix a positive integer $n$.  Assume that for every compact connected Lie group $K$ of dimension $n$ acting on $M$ we have that any $K$-basic form on $M$ is a pullback of a form on $M/K$.  Then for any compact (possibly disconnected) Lie group $G$ of dimension $n$, every $G$-basic form on $M$ is the pullback of a form on $M/G$.
\end{lemma}

\begin{proof}
Let $G_0$ be the identity component of $G$.  $G_0$ is a subgroup of $G$, and so has an induced action on $M$.  Let $\pi_0:M\to M/G_0$ be the orbit map.  We want to show that for any $G$-basic form $\alpha$ on $M$, if $p_1:U\to M$ and $p_2:U\to M$ are two plots satisfying $\pi\circ p_1=\pi\circ p_2$ (where here $\pi:M\to M/G$ is the orbit map), then $p_1^*\alpha=p_2^*\alpha$.  Fix such $\alpha$, $p_1$, and $p_2$.  For any $\gamma\in G/G_0$, let $$C_\gamma:=\{u\in U~|~\exists g\in\gamma\text{ such that }p_2(u)=g\cdot p_1(u)\}.$$  Since $\gamma$ is a finite set, by continuity we have that $C_\gamma$ is the finite union of closed sets, and hence closed.  By \hypref{l:finiteunion}{Lemma}, we know that the set $\bigcup_{\gamma\in G/G_0}\inter(C_\gamma)$ is open and dense in $U$.  It is thus enough to show that $p_1^*\alpha=p_2^*\alpha$ on each nonempty open set $\inter(C_\gamma)$.\\

Let $\gamma\in G/G_0$ such that $\gamma\neq G_0$ and $\inter(C_\gamma)$ is not empty.  Fix $g\in\gamma$ and define $\tilde{p}_1:U\to M$ as the composition $g\circ p_1$.  This is a plot of $M$, and for any $u\in\inter(C_\gamma)$, $p_2(u)=h\cdot\tilde{p}_1(u)$ for some $h\in G_0$.  Hence, $\pi_0\circ\tilde{p}_1|_{\inter(C_\gamma)}=\pi_0\circ p_2|_{\inter(C_\gamma)}$.  Note that the restrictions $\tilde{p}_1|_{\inter(C_\gamma)}$ and $p_2|_{\inter(C_\gamma)}$ are plots of $M$.  By hypothesis, since $G_0$ is connected and $\alpha$ is $G_0$-basic (because it is $G$-basic), we have that $\tilde{p}_1^*\alpha=p_2^*\alpha$ on $\inter(C_\gamma)$ by \hypref{p:pullbackimgchar}{Proposition}.  But $\tilde{p}_1^*\alpha=p_1^*g^*\alpha=p_1^*\alpha$ since $\alpha$ is $G$-invariant, and so $p_1^*\alpha=p_2^*\alpha$ on $\inter(C_\gamma)$.\\

Now we only need to check the case when $\gamma=G_0$.  But in this case, for each $u\in C_\gamma$, there is some $h\in G_0$ such that $h\cdot p_1(u)=p_2(u)$.  We apply the same argument as above, and we get that $p_1^*\alpha=p_2^*\alpha$ on $\inter(C_\gamma)$.  Hence, $p_1^*\alpha=p_2^*\alpha$ on the open dense subset $\bigcup_{\gamma\in G/G_0}\inter(C_\gamma)$, and hence everywhere by continuity.
\end{proof}

Next, we reduce the case of a compact connected Lie group of dimension $n$ to a lower dimensional Lie group, or to the case of a linear and orthogonal action on a vector space.  First we need some lemmas.\\

Let $G$ be a compact connected Lie group acting on a manifold $M$.  Fix $x\in M$ and let $H$ be the stabiliser of $x$.  The slice theorem (see \cite{koszul53}) states that there is a $G$-invariant open neighbourhood $U$ of $x$ and an equivariant diffeomorphism $F:U\to G\times_H V$ where $V$ is the normal space to $G\cdot x$ at $x$, that is, $V=T_xM/T_x(G\cdot x)$.  Let $i:V\to G\times_H V$ be the smooth injection $v\mapsto[e,v]$ where $e\in G$ is the identity.  Then this map is equivariant with respect to $H$, and descends to a bijection $\psi:V/H\to(G\times_H V)/G$. We have the following commutative diagram, where $\pr_2$ is the projection onto $V$, and $\pi_H$, $\pi_V$, $\pi$ are quotient maps.

$$\xymatrix{
G\times V \ar[d]_{\pr_2} \ar[dr]^{\pi_H} & \\
V \ar[r]_i \ar[d]_{\pi_V} & G\times_H V \ar[d]^{\pi} \\
V/H \ar[r]_{\psi~~~} & (G\times_H V)/G \\
}$$

\begin{lemma}\labell{l:samequot}
$\psi:V/H\to(G\times_H V)/G$ is a diffeological diffeomorphism.
\end{lemma}

\begin{proof}
Let $p:U\to V/H$ be a plot in the quotient diffeology on $V/H$.  Then, for every $u\in U$, there is an open neighbourhood $W\subseteq U$ of $u$ and a plot $q:W\to V$ such that $$p|_W=\pi_V\circ q.$$  Then, $i\circ q$ is a plot of $G\times_H V$ (with image in $i(V)$), and so $$\pi\circ i\circ q=\psi\circ\pi_V\circ q=\psi\circ p|_W$$ is a plot of $(G\times_H V)/G$.  Thus, $\psi\circ p$ is a plot of $(G\times_H V)/G$.  Hence $\psi$ is diffeologically smooth.  Since $\psi$ is a bijection, it induces an injection of plots in the quotient diffeology of $V/H$ into the quotient diffeology on $(G\times_H V)/G$.  We next show that $\psi^{-1}$ is diffeologically smooth.\\

Let $\tilde{p}:U\to (G\times_H V)/G$ be a plot in the quotient diffeology on $(G\times_H V)/G$.  Then, for every $u\in U$ there is an open neighbourhood $W\subseteq U$ of $u$ and a plot $q:W\to G\times_H V$ such that $$\tilde{p}|_W=\pi\circ q.$$ Shrinking $W$ if necessary, we may assume that there is a plot $r:W\to G\times V$ in the product diffeology on $G\times V$ such that $$q=\pi_H\circ r.$$ Consider the plot $\pr_2\circ r$ on $V$.  We claim that $\pi\circ i\circ\pr_2=\pi\circ\pi_H$.  If this is true, then $$\pi\circ i\circ\pr_2\circ r=\pi\circ\pi_H\circ r=\tilde{p}|_W.$$  But since $\pi\circ i=\psi\circ\pi_V$, we would have $$\tilde{p}|_W=\psi\circ\pi_V\circ\pr_2\circ r.$$ Since $\pi_V\circ\pr_2\circ r$ is a plot on $V/H$, we have that $\tilde{p}$ locally is the pushforward of a plot on $V/H$ by $\psi$; that is, $\psi^{-1}\circ p$ locally is a plot of $V/H$.  But by the axioms of diffeology, these glue together into one plot, $\psi^{-1}\circ\tilde{p}$, and hence $\psi^{-1}$ is diffeologically smooth.\\

We now prove our claim: $\pi\circ i\circ\pr_2=\pi\circ\pi_H$.  Fix a point $(g,v)\in G\times V$.  Then $$\pi\circ i\circ\pr_2(g,v)=\pi\circ i(v)=\pi([e,v])$$ and $\pi([e,v])$ is the whole $G$-orbit containing $[e,v]$.  On the other hand, $\pi\circ\pi_H(g,v)=\pi[g,v]$ which is the $G$-orbit containing $[g,v]$.  But this is the same orbit as that containing $[e,v]$.  And so we are done.
\end{proof}

The above lemma in conjunction with the slice theorem shows that $U/G$ is diffeologically diffeomorphic to $V/H$.

\begin{lemma}\labell{l:basicformslice}
$i^*$ is an injection from $G$-basic forms on $G\times_H V$ to $H$-basic forms on $V$.
\end{lemma}

\begin{proof}
Fix a $G$-basic form $\alpha$ on $G\times_H V$.  Then $i^*\alpha$ is well-defined.  Since $\alpha$ is $G$-basic, it is also $H$-basic.  Since $i$ is $H$-equivariant, $i^*\alpha$ is also $H$-basic.  We now only need to show injectivity.\\

Assume $i^*\alpha=0$.  Fix $y\in G\times_H V$ and $w\in T_y(G\times_H V)$.  There is some $g\in G$ such that $g\cdot y\in i(V)$ and so $g_*w\in T_{g\cdot y}(G\times_H V)$.  Since $i(V)$ is transverse to the orbit $G\cdot y$ at $g\cdot y$, we have $$T_{g\cdot y}(G\times_H V)=T_{g\cdot y}(G\cdot y)+T_{g\cdot y}(i(V)),$$ and so $g_*w=w_{orbit}+w_V$ where $w_{orbit}\in T_{g\cdot y}(G\cdot y)$ and $w_V\in T_{g\cdot y}(i(V))$.  Thus,

\begin{align*}
w\hook\alpha=&~g_*w\hook\alpha\\
=&~w_{orbit}\hook\alpha+w_V\hook\alpha\\
=&~w_V\hook\alpha
\end{align*}
since $\alpha$ is $G$-basic.  There is some $w'\in T_{i^{-1}(g\cdot y)}V$ such that $i_*w'=w_V$.  So,
\begin{align*}
w\hook\alpha=&w_V\hook\alpha\\
=&~i_*w'\hook\alpha\\
=&~w'\hook i^*\alpha\\
=&~0,
\end{align*}
and therefore we conclude that $\alpha=0$, completing the proof.
\end{proof}

\begin{lemma}\labell{l:indstep}
Let $G$ be a compact connected Lie group of dimension $n$ acting on a manifold $M$.  Assume that for every compact Lie group $H$ of dimension less than $n$, and for every Lie group action of $H$ on every connected manifold $N$, that the pullback of forms on $N/H$ via the orbit map is an isomorphism onto the $H$-basic forms on $N$.  Then the pullback of forms on $M/G$ via the orbit map is an isomorphism onto the $G$-basic forms on $M$.
\end{lemma}

\begin{proof}
Fix a $G$-basic form $\alpha$ on $M$.  We would like to show that for any two plots $p_1:W\to M$ and $p_2:W\to M$ such that $\pi\circ p_1=\pi\circ p_2$, we have $p_1^*\alpha=p_2^*\alpha$.  Let $p_1$ and $p_2$ be two such plots.  Fix $u\in W$, and let $x=p_2(u)$.  Note that there exists $g\in G$ such that $g\cdot p_1(u)=p_2(u)$.  Let $H$ be the stabiliser of $x$.  By the slice theorem there exists a $G$-invariant open neighbourhood $U$ of $x$ and an equivariant diffeomorphism $F:U\to G\times_H V$ where $V=T_xM/T_x(G\cdot x)$.  Let $\pi_H:V\to V/H$ be the orbit map, and let $i:V\to G\times_H V$ be the map $v\mapsto[e,v]$.  By \hypref{l:basicformslice}{Lemma}, $(F^{-1}\circ i)^*$ is an injection from $G$-basic forms on $U$ to $H$-basic forms on $V$.\\

Assume that $H$ is not equal to $G$; in particular, that it is a subgroup of $G$ of dimension less than $n$.  By our hypothesis, there is a form $\beta$ on $V/H$ such that $\pi_H^*\beta=(F^{-1}\circ i)^*(\alpha|_U)$.  By \hypref{l:samequot}{Lemma}, there is a diffeological diffeomorphism $\psi:V/H\to U/G$.  We have that $$i^*(F^{-1})^*(\pi|_U)^*(\psi^{-1})^*\beta=\pi_H^*\beta=i^*(F^{-1})^*(\alpha|_U).$$  Since $(F^{-1}\circ i)^*$ is injective, we conclude that $$(\pi|_U)^*(\psi^{-1})^*\beta=\alpha|_U.$$  Thus
\begin{align*}
(p_1|_{p_1^{-1}(U)})^*g^*(\alpha|_U)=&(p_1|_{p_1^{-1}(U)})^*(\alpha|_U)\\
=&(p_1|_{p_1^{-1}(U)})^*(\pi|_U)^*(\psi^{-1})^*\beta\\
=&(p_2|_{p_1^{-1}(U)})^*(\pi|_U)^*(\psi^{-1})^*\beta\\
=&(p_2|_{p_1^{-1}(U)})^*(\alpha|_U).
\end{align*}

We had assumed that $H$ is not equal to $G$ in the argument above; that is, $x$ was assumed not to be a fixed point.  In the case that it is a fixed point, $F$ identifies $U$ with $V=T_xM$, sending $x$ to $0\in V$.  Fixing a $G$-invariant Riemannian metric, we have that $G$ acts linearly and orthogonally on $V$.  Let $v\in T_uW$.  It will suffice to show that $v\hook p_1^*\alpha=v\hook p_2^*\alpha$.  Let $\xi_1=g_*(p_1)_*v$ and $\xi_2=(p_2)_*v$.  Then it suffices to show that $(\xi_2-\xi_1)\hook\alpha=0$.  But using the identification $F:U\to V$, it is enough to solve the problem on $V$, where $\xi_1$ and $\xi_2$ are tangent vectors at $0\in V$.  But $(\xi_2-\xi_1)\hook\alpha=0$ has been shown already in \hypref{c:techn2}{Corollary}.  This completes the proof.
\end{proof}

We now prove the main result of this section.

\begin{theorem}[Basic Forms Equal Pullbacks]\labell{t:geomquot}
Let $G$ be a compact Lie group acting on a manifold $M$.  Let $\pi:M\to M/G$ be the orbit map.  Then $\pi^*$ is an isomorphism from differential forms on $M/G$ onto basic differential forms on $M$.
\end{theorem}

\begin{proof}
By \hypref{p:pullbacksarebasic}{Proposition} and \hypref{p:injectivity}{Proposition}, we have that $\pi^*$ is an injective map into basic forms on $M$.  So it is enough to show that it is surjective onto the basic forms.\\

We shall prove this by induction on the dimension of $G$.  If $n=0$, then by \hypref{p:finite}{Proposition} we are done.  This is the base case.\\

Next, assume that the theorem is proved for all Lie group actions of compact Lie groups of dimensions $0$ up to $n-1$.  Assume that $G$ is of dimension $n$.  By \hypref{l:conntodisc}{Lemma} it is enough if we assume that $G$ is connected.  Then by \hypref{l:indstep}{Lemma} we are done.
\end{proof}

\section{Short Introduction to Stratified Spaces}\labell{s:stratified}

For the purposes of the following section, as well as the following chapter, we need to introduce stratified spaces.  Unfortunately, in the literature, there are many definitions, not all of which are equivalent (see \cite{sniatycki03}, \cite{pflaum01}, \cite{goresky-macpherson88}).  For our purposes, we start with a topological definition, following closely the terminology used in \cite{lerman-sjamaar91}.  We then transport these concepts into the differential space category, following closely concepts introduced in \cite{sniatycki03} and \cite{lusala-sniatycki11}.\\

Let $X$ be a Hausdorff, paracompact topological space, and $(A,\leq)$ a partially ordered set.

\begin{definition}[Decomposed Space]
A \emph{decomposition} of $X$ with respect to $(A,\leq)$ is a locally finite partition of $X$, denoted by $\mathcal{P}$, into disjoint, connected, locally closed (topological) manifolds $S_i$, called \emph{pieces} such that the set $\mathcal{P}$ is indexed by $A$, and $i\leq j$ if and only if $S_i\subseteq\overline{S}_j$ if and only if $S_i\cap\overline{S}_j\neq\emptyset$.  The \emph{dimension} of $X$, denoted $\dim(X)$, is the supremum over $A$ of the dimensions of the pieces.  $X$ equipped with a decomposition $\mathcal{P}$ will be referred to as a \emph{decomposed space}, denoted $(X,\mathcal{P})$.  Often we will drop the notation $\mathcal{P}$ when the decomposition has been made clear.
\end{definition}

\begin{remark}
We will only consider decomposed spaces in which the pieces are finite-dimensional.
\end{remark}

\begin{definition}[Depth]
Let $(X,\mathcal{P})$ be a decomposed space, and fix a piece $S\in\mathcal{P}$.  The \emph{depth} of $S$, denoted $\operatorname{depth}_X(S)$ is defined as $$\operatorname{depth}_X(S):=\sup\{n\in\NN~|~S=S_{a_0}\subsetneq S_{a_1}\subsetneq...\subsetneq S_{a_n}\}$$ where each $S_{a_i}\in\mathcal{P}$.  Note the \emph{strict} inclusions in the definition. The \emph{depth} of $X$ is given by $$\operatorname{depth}(X):=\sup\{\operatorname{depth}_X(S_a)~|~a\in A\}.$$
\end{definition}

\mute{

\begin{example}[Cone of a Decomposed Space]
Let $X$ be a decomposed space with pieces $S_a$ ($a\in A$).  The \emph{cone} of $X$, denoted $\operatorname{cone}(X)$, is defined by the topological quotient $$\operatorname{cone}(X):=X\times[0,\infty)/X\times\{0\}.$$  This is also a decomposed space, with pieces given by $S_a\times(0,1)$ for each $a\in A$, and the vertex given by the equivalence class $X\times\{0\}$.  Note that the depth of $\operatorname{cone}(X)$ is equal to $\operatorname{depth}(X)+1$.
\end{example}

%end mute
}

A stratified space is a decomposed space in which the pieces  fit together in a specific way.  Note that the following definition is recursive (in particular, $F$ will have a smaller depth than $S$).

\begin{definition}[Stratified Space]
A decomposed space $X$ is a \emph{stratified space} if the pieces of $X$, called \emph{strata}, satisfy the following condition.

\begin{itemize}\labell{cond:localtriv}
\item (Local Triviality) For every $x\in X$, there is an open neighbourhood $U\subseteq X$ of $x$, a stratified space $F$ with a distinguished point $o\in F$ such that $\{o\}$ is a stratum of $F$, and a homeomorphism $\varphi:U\to(S\cap U)\times F$ where $S$ is the stratum of $X$ containing $x$.  $\varphi$ is required to satisfy $$\varphi(s)=(s,o)$$ for each $s\in S\cap U$, and to map strata into strata.
\end{itemize}
\end{definition}

\begin{remark}
The above local triviality condition is often written in the literature using cones over stratified spaces instead of $F$.  However, it will be easier to transport the definition we use above into the smooth category.
\end{remark}

\begin{example}\labell{x:square}
Consider the square $[0,1]\times[0,1]$.  The partition given by $$\mathcal{P}=\Big\{\{(0,0)\},\;\{(0,1)\},\;\{(1,0)\},\;\{(1,1)\},\;(0,1)\times\{0\},$$$$
(0,1)\times\{1\},\;\{0\}\times(0,1),\;\{1\}\times(0,1),\;(0,1)^2\Big\}$$ makes the square into a stratified space.
\end{example}

\begin{definition}[Smooth Decomposed Space]
A \emph{smooth decomposed space} is a triple $(X,\mathcal{F},\mathcal{P})$ where $(X,\mathcal{F})$ is a differential space, and $(X,\mathcal{P})$ is a decomposed space with respect to the topology induced by $\mathcal{F}$.  We require that for each piece $S\in\mathcal{P}$, the inclusion map $i_S:S\to X$ induces a smooth manifold structure on $S$.
\end{definition}

\begin{definition}[Smooth Stratified Space]
A smooth decomposed space $X$ is a \emph{smooth stratified space} if the pieces of the decomposition satisfy the following condition.
\begin{itemize}
\item (Smooth Local Triviality) For every $x\in X$, there is an open neighbourhood $U\subseteq X$ of $x$, a smooth stratified space $F$ with a distinguished point $o\in F$ such that $\{o\}$ is a stratum of $F$, and a diffeomorphism $\varphi:U\to(S\cap U)\times F$ where $S$ is the stratum of $X$ containing $x$.  $\varphi$ is required to satisfy $$\varphi(s)=(s,o)$$ for each $s\in S\cap U$, and to map strata into strata.
\end{itemize}
Again, the pieces in the decomposition in this case are called \emph{strata}.
\end{definition}

For our purposes, we will always assume that the differential structure on a smooth stratified space is subcartesian.

\begin{definition}[Smooth Stratified Map]
Let $X$ and $Y$ be smooth stratified spaces, and let $F:X\to Y$ be a smooth map.  Then $F$ is \emph{stratified} if for each stratum $S$ of $X$, $F(S)\subset T$ for some stratum $T$ of $Y$.
\end{definition}

\begin{remark}
Smooth stratified spaces, along with stratified maps, form a category.
\end{remark}

Let $G$ be a compact Lie group acting on a (smooth) manifold $M$.  Let $H$ be a closed subgroup of $G$, and let $M_{(H)}$ be the set of all points in $M$ whose stabiliser is a conjugate of $H$.  Then, $M$ is the disjoint union of the sets $M_{(H)}$ as $H$ runs over closed subgroups of $G$.  The quotient map $\pi:M\to M/G$ partitions $M/G$ into sets $(M/G)_{(H)}:=\pi(M_{(H)})$ as $H$ runs over closed subgroups of $G$.

\begin{theorem}[Orbit-Type Stratifications]\labell{t:pot}
\noindent
\begin{enumerate}
\item The partitions on $M$ and $M/G$ defined above yield smooth stratifications with respect to the smooth structures $\CIN(M)$ and $\CIN(M/G)$, respectively, whose strata are given by connected components of the sets $M_{(H)}$ and $(M/G)_{(H)}$.
\item Each subset $M_{(H)}$ is a $G$-invariant submanifold of $M$.  If $M$ and $G$ are connected, then each stratum in the stratification on $M$ is $G$-invariant.
\item If $M$ is connected, then there exists a closed subgroup $K$ of $G$ such that the strata contained in $M_{(K)}$ form an open dense subset of $M$, and hence $(M/G)_{(K)}$ is an open dense subset of $M/G$.
\item The orbit map $\pi:M\to M/G$ is stratified with respect to the stratifications described above.
\end{enumerate}
\end{theorem}
\begin{proof}
The last statement above is clear by definition of the stratifications. See \cite{duistermaat-kolk04} and \cite{cushman-sniatycki01} for the first three statements.
\end{proof}

\begin{definition}\labell{d:symplquotstrat}
We call the above stratifications \emph{orbit-type stratifications} of each respective space.
\end{definition}

\section{The Symplectic Quotient}

We now turn our attention to the symplectic quotient, and the induced stratification on it.  Using the stratification, we show how to define a differential form on the symplectic quotient.\\

Assume that $G$ is a compact Lie group acting in a Hamiltonian fashion on a symplectic manifold $(M,\omega)$ with momentum map $\Phi$ and $Z:=\Phi^{-1}(0)$.  Let $i:Z\to M$ be the inclusion.  For each closed subgroup $H$ of $G$, let $Z_{(H)}:=M_{(H)}\cap Z$.  Note that this is a $G$-invariant subset of $Z$ since both $Z$ and $M_{(H)}$ are invariant.  Let $(Z/G)_{(H)}:=\pi(Z_{(H)})$.  For each nonempty such subset, let $\pi_{(H)}:=\pi|_{Z_{(H)}}$ and $i_{(H)}:=i|_{Z_{(H)}}$.  Finally, let $\pi_Z:=\pi|_Z$ and let $j:Z/G\to M/G$ be the inclusion so that the following diagram commutes.

$$\xymatrix{
Z \ar[d]_{\pi_Z} \ar[r]^i & M \ar[d]^{\pi} \\
Z/G \ar[r]_j & M/G \\
}$$

\begin{theorem}[Orbit-Type Stratification (Hamiltonian Version)]\labell{t:poth}
\noindent
\begin{enumerate}
\item The partitions on $Z$ and $Z/G$ defined above yield smooth stratifications with respect to $\CIN(Z)$ and $\CIN(Z/G)$, respectively, whose strata are given by connected components of the sets $Z_{(H)}$ and $(Z/G)_{(H)}$.
\item Each subset $Z_{(H)}$ is a $G$-invariant submanifold of $M$.
\item If $M$ is connected and $\Phi$ is a proper map, then there exists a closed subgroup $K$ of $G$ such that the strata contained in $Z_{(K)}$ form an open dense subset of $Z$, and hence $(Z/G)_{(K)}$ is an open dense subset of $Z/G$.
\item The orbit map $\pi_Z:Z\to Z/G$ along with the inclusions $i$ and $j$ are stratified with respect to the stratifications described above.
\end{enumerate}
\end{theorem}

\begin{proof}
See \cite{lerman-sjamaar91}.
\end{proof}

\begin{definition}[Orbit-Type Stratifications]
The stratifications defined on $Z$ and $Z/G$ above are called \emph{orbit-type stratifications} of each space.
\end{definition}

\begin{remark}
Note that $$(Z/G)_{(H)}=Z_{(H)}/G=(Z/G)\cap(M/G)_{(H)}.$$
\end{remark}

Here and afterward, we let $K$ be a closed subgroup of $G$ such that $Z_{(K)}$ and $(Z/G)_{(K)}$ are open and dense strata.  We have the following commutative diagram of diffeological spaces, where $I$ and $J$ are inclusions.

$$\xymatrix{
Z_{(K)} \ar[r]^I \ar[d]_{\pi_{(K)}} \ar@/^2pc/[rr]^{i_{(K)}} & Z \ar[r]^i \ar[d]_{\pi_Z} & M \ar[d]^{\pi} \\
(Z/G)_{(K)} \ar[r]_J & Z/G \ar[r]_j & M/G
}$$

\begin{definition}
A \emph{Sjamaar differential $l$-form} $\sigma$ on $Z/G$ is a differential $l$-form on $(Z/G)_{(K)}$ (in the ordinary sense) such that there exists $\tilde{\sigma}\in\Omega^l(M)$ satisfying $i_{(K)}^*\tilde{\sigma}=\pi_{(K)}^*\sigma$.  Denote the set of Sjamaar $l$-forms by $\Omega_{Sjamaar}^l(Z/G)$.
\end{definition}

We can define an exterior derivative $d$ of a Sjamaar form $\alpha$ as simply the usual exterior derivative of $\alpha$ as a differential form on $(Z/G)_{(K)}$.  In his paper \cite{sjamaar05} Sjamaar shows that $\Omega^*_{Sjamaar}(Z/G)$ equipped with the exterior derivative yields a subcomplex of $(\Omega^*((Z/G)_{(K)}),d)$, which satisfies a Poincar\'e Lemma, Stokes' Theorem, and a de Rham theorem.

\begin{example}[$\ZZ_2\circlearrowright\RR^2$]
Consider $\ZZ_2=\{\pm1\}$ acting on $(\RR^2,dx\wedge dy)$ by scalar multiplication.  This action preserves the symplectic form, and is Hamiltonian with constant momentum map.  The zero level is thus all of $\RR^2$, and the symplectic quotient is equal to the geometric quotient.\\

We will use the induced differential structure to describe the quotient as a subset of $\RR^3$.  Three polynomials that generate all invariant polynomials are $x^2-y^2$, $2xy$, and $x^2+y^2$.  The quotient map $\pi$ can be expressed as $$\pi(x,y)=(x^2-y^2,2xy,x^2+y^2).$$  The image is the cone $$C:=\{(s,t,u)\in\RR^3~|~s^2+t^2=u^2,~u\geq0\}.$$  The open dense stratum of $C$ is the cone minus its apex.  Call this $C'$.  Consider the differential form $\sigma=\frac{1}{4u}ds\wedge dt$ on $\RR^3$ minus the plane $u=0$.  This pulls back to a form on $C'$.  Moreover,
\begin{align*}
(\pi|_{\pi^{-1}(C')})^*\sigma=&\frac{1}{4(x^2+y^2)}d(x^2-y^2)\wedge d(2xy)\\
=&\frac{1}{4(x^2+y^2)}(2xdx-2ydy)\wedge2(ydx+xdy)\\
=&\frac{1}{x^2+y^2}(x^2+y^2)(dx\wedge dy)\\
=&dx\wedge dy.
\end{align*}
This extends to the symplectic form $dx\wedge dy$ on all of $\RR^2$.  Hence, $\sigma$ is a Sjamaar form.
\end{example}

\begin{lemma}\labell{l:extendform}
Let $i:N\to M$ be a closed submanifold, and let $\alpha\in\Omega^l(N)$.  Then there exists $\tilde{\alpha}\in\Omega^l(M)$ such that $i^*\tilde{\alpha}=\alpha$.
\end{lemma}

\begin{proof}
Let $U$ be a tubular neighbourhood of $N$, and let $r:U\to N$ be a retraction of $U$ onto $N$.  Let $b:M\to\RR$ be a bump function such that $b|_N=1$ and $\supp(b)\subset U$.  Then, define $\tilde{\alpha}$ to be an $l$-form on $U$ defined by $b(r^*\alpha)$.  Then, $i^*\tilde{\alpha}=\alpha$.  Since $\tilde{\alpha}$ has support in $U$, we can extend it as $0$ to the rest of $M$.
\end{proof}

\begin{theorem}\labell{t:sjamaar}
Assume that $0\in\g^*$ is a regular value for the momentum map $\Phi$.  Then the pullback map $J^*:\Omega^l(Z/G)\to\Omega^l((Z/G)_{(K)})$ is a bijection onto the Sjamaar $l$-forms.  That is, $J$ induces an isomorphism of complexes from diffeological forms on $Z/G$ onto the Sjamaar forms.
\end{theorem}

\begin{remark}
Recall that $0\in\g^*$ being a regular value for $\Phi$ is equivalent to $\dim(\Stab(z))=0$ for all $z\in Z$.
\end{remark}

\begin{proof}
We first prove that the image of $J^*$ is in Sjamaar forms; that is, for any $\beta\in\Omega^l(Z/G)$, there exists $\tilde{\alpha}\in\Omega^l(M)$ such that $i_{(K)}^*\tilde{\alpha}=\pi_{(K)}^*(J^*\beta)$.

From \hypref{t:geomquot}{Theorem}, we know that basic forms on a manifold are in bijection with the diffeological forms on the geometric quotient.  But this is exactly the set up we have here: $Z/G$ is the geometric quotient of the manifold $Z$.  Thus, there is a (basic) form $\alpha\in\Omega^l(Z)$ such that $\pi_Z^*\beta=\alpha$.  But since $Z$ is a closed submanifold of $M$, by \hypref{l:extendform}{Lemma} there is a form $\tilde{\alpha}\in\Omega^l(M)$ such that $i^*\tilde{\alpha}=\alpha$.  So $$\pi_{(K)}^*J^*\beta=I^*\pi_Z^*\beta=i_{(K)}^*\tilde{\alpha},$$ and so $J^*\beta$ is a Sjamaar form.\\

We next prove that $J^*$ is surjective.  Let $\sigma$ be a Sjamaar $l$-form.  Then there exists $\tilde{\alpha}\in\Omega^l(M)$ such that $i_{(K)}^*\tilde{\alpha}=\pi_{(K)}^*\sigma$.  We claim that $i^*\tilde{\alpha}$ is basic on $Z$.  If this claim is true, then by \hypref{t:geomquot}{Theorem}, there exists a diffeological form $\beta\in\Omega^l(Z/G)$ such that $\pi_Z^*\beta=i^*\tilde{\alpha}$.  Then, $I^*\pi_Z^*\beta=i_{(K)}^*\tilde{\alpha}$, and so $$\pi_{(K)}^*J^*\beta=\pi_{(K)}^*\sigma.$$  But by \hypref{p:injectivity}{Proposition}, $\pi_{(K)}^*$ is one-to-one, and so we could conclude that $J^*\beta^*=\sigma$.  So we strive to prove the claim.

Let $x\in Z$ and let $v\in T_xZ$ be tangent to $G\cdot x$.  If $x\in Z_{(K)}$, then
\begin{align*}
v\hook i^*\tilde{\alpha}=&v\hook  i_{(K)}^*\tilde{\alpha}\\
=&v\hook\pi_{(K)}^*\sigma\\
=&0.
\end{align*}
If $x\notin Z_{(K)}$, then since $Z_{(K)}$ is open and dense in $Z$, and $T(G\cdot x)$ is in the closure of $\bigcup_{z\in Z_{(K)}}T_z(G\cdot z)$ in $TZ$, we have that there exist a sequence $(z_\nu)$ of points in $Z_{(K)}$ and vectors $(v_\nu)$ such that for each $\nu$, $v_\nu\in T_{z_\nu}Z_{(K)}$ is tangent to $G\cdot z_\nu$, $z_\nu\to x$, and $v_\nu\to v$.  Then,
\begin{align*}
v\hook i^*\tilde{\alpha}=&\underset{\nu\to\infty}{\lim}\left(v_\nu\hook i_{(K)}^*\tilde{\alpha}\right)\\
=&0.
\end{align*}
Thus, $i^*\tilde{\alpha}$ is basic on $Z$.  We now have that $J^*$ is surjective.\\

We finally wish to show that $J^*$ is one-to-one.  To this end, let $\beta\in\Omega^l(Z/G)$ such that $J^*\beta=0$.  We want to show that $\beta$ itself is equal to $0$.  That is, for any plot $p:U\to Z/G$, we want $p^*\beta=0$.  For any plot $p:U\to Z/G$ and for any $u\in U$ there exist an open neighbourhood $V\subseteq U$ of $u$, a chart $q:W\to Z$ of $Z$, and a smooth map $F:V\to W$ such that $p|_V=\pi_Z\circ q\circ F$.  So it is sufficient to show that $(\pi_Z\circ q)^*\beta=0$ for any chart $q:W\to Z$.  Moreover, in this case, it is sufficient to show that $(\pi_Z\circ q)^*\beta$ is $0$ on an open dense subset of $W$.

To this end, fix a chart $q:W\to Z$.  Note that $B:=q^{-1}(\pi_Z^{-1}((Z/G)_{(K)}))$ is an open dense subset of $W$.  Also, $q|_B:B\to(Z/G)_{(K)}$ is a plot of $(Z/G)_{(K)}$. Thus, by assumption, $$(\pi\circ q|_B)^*\beta=(\pi\circ q|_B)^*J^*\beta=0.$$  That is to say, $(\pi\circ q|_B)^*\beta$ is equal to $0$ on the open dense subset $B$ of $W$.  So we have shown that $\beta=0$, and we conclude that $J^*$ is one-to-one.  This completes the proof.
\end{proof}

The next theorem is a partial result when we do not assume that $0\in\g^*$ is a regular value of $\Phi$.

\begin{theorem}
Let $\sigma$ be a Sjamaar $l$-form on $Z/G$.  Then there exists $\beta\in\Omega^l(Z/G)$ such that $J^*\beta=\sigma$.
\end{theorem}

\begin{proof}
Since $\sigma$ is a Sjamaar $l$-form, there exists $\tilde{\alpha}\in\Omega^k(M)$ such that $i_{(K)}^*\tilde{\alpha}=\pi_{(K)}^*\sigma$.  If there exists $\beta\in\Omega^l(Z/G)$ such that $\pi_Z^*\beta=i^*\tilde{\alpha}$, then, $$\pi_{(K)}^*J^*\beta=i_{(K)}^*\tilde{\alpha}=\pi_{(K)}^*\sigma.$$  By \hypref{p:injectivity}{Proposition}, $\pi_{(K)}^*$ is injective, and so $J^*\beta=\sigma$.  We now wish to show the existence of $\beta$.\\

Let $\alpha=i^*\tilde{\alpha}$.  By \hypref{p:pullbackimgchar}{Proposition}, it is enough to show that for any two plots $p_1:U\to Z$ and $p_2:U\to Z$ such that $\pi_Z\circ p_1=\pi_Z\circ p_2$, we have $p_1^*\alpha=p_2^*\alpha$.  Fix two such plots.  Define $$C_{(H)}:=\{u\in U~|~p_1(u)\in Z_{(H)}\}.$$  Note that due to the invariance of each stratum, $p_1(u)$ and $p_2(u)$ will always lie within the same stratum, for each $u\in U$.  Without loss of generality, assume that the images of $p_1$ and $p_2$ intersect only a finite number of orbit-type strata (which is possible since $G$ is compact).  Then there are only finitely many nonempty $C_{(H)}$.  By \hypref{l:finiteunion}{Lemma} the set $\bigcup_{(H)}\inter(\overline{C_{(H)}})$ is an open dense subset of $U$. Thus we only need to show that $$(p_1|_{\inter(\overline{C_{(H)}})})^*\alpha=(p_2|_{\inter(\overline{C_{(H)}})})^*\alpha$$ for each $(H)$ such that the interior of $\overline{C_{(H)}}$ is nonempty. Fix such a conjugacy class.  Now, $p_1|_{\inter(\overline{C_{(H)}})}$ and $p_2|_{\inter(\overline{C_{(H)}})}$ are plots on $\overline{Z_{(H)}}$. $C_{(H)}=(p_1)^{-1}(Z_{(H)})$, and since $Z_{(H)}$ is open in $\overline{Z_{(H)}}$, we have that $C_{(H)}$ is open (and dense) in $\overline{C_{(H)}}$.  Thus, $B_{(H)}:=C_{(H)}\cap\inter(\overline{C_{(H)}})$ is open and dense in $\inter(\overline{C_{(H)}})$.  Denote $p'_1:=p_1|_{B_{(H)}}$ and $p'_2:=p_2|_{B_{(H)}}$.  These are two plots on $Z_{(H)}$ with the same domain, satisfying $\pi_{(H)}\circ p'_1=\pi_{(H)}\circ p'_2$.  But $Z_{(H)}$ is a $G$-manifold, and the pullback of $\alpha$ to $Z_{(H)}$ is equal to $i_{(H)}^*\tilde{\alpha}$, which we know is basic (see Lemma 3.3 of \cite{sjamaar05}).  Thus by a combination of \hypref{t:geomquot}{Theorem} and \hypref{p:pullbackimgchar}{Proposition}, we have that $(p'_1)^*\alpha=(p'_2)^*\alpha$.  Hence, $p_1^*\alpha=p_2^*\alpha$ on an open dense subset of $\inter(\overline{C_{(H)}})$.  This completes the proof.
\end{proof}

\section{Reduction in Stages}

One application of \hypref{t:sjamaar}{Theorem} is to reduction in stages.  More details are given in \cite{lerman-sjamaar91}, but we review the important details here.\\

Let $G_1$ and $G_2$ be compact Lie groups acting on a smooth connected symplectic manifold $(M,\omega)$ such that their actions are Hamiltonian, and the actions commute.  Let $\Phi_1$ and $\Phi_2$ be the corresponding momentum maps.  Average $\Phi_1$ over $G_2$ and $\Phi_2$ over $G_1$.  Then $G=G_1\times G_2$ acts on $(M,\omega)$ in a Hamiltonian fashion with momentum map $\Phi=\Phi_1\times\Phi_2$.\\

Let $Z_1:=\Phi_1^{-1}(0)$.  Then $Z_1$ is $G$-invariant (or equivalently, it is both $G_1$ and $G_2$-invariant).  Thus, $Z_1/G_1$ is a $G_2$-space.  Indeed, let $\pi_1:Z_1\to Z_1/G_1$ be the quotient map.  Then, for any $g_2\in G_2$ and $z\in Z_1$, we have $g_2\cdot\pi_1(z):=\pi_1(g_2\cdot z)$.  This is well-defined since the actions of $G_1$ and $G_2$ on $Z_1$ commute.  We claim that $\Phi_2|_{Z_1}$ descends to a $G_2$-equivariant map $\Phi'_2:Z_1/G_1\to\g^*_2$ defined by the relation $$\Phi_2|_{Z_1}=\Phi'_2\circ\pi_1.$$  $\Phi'_2$ is well-defined since $\pi_1$ is surjective, and $\Phi_2$ is $G_1$-equivariant.  Also, for any $g_2\in G_2$ and $z\in Z_1$, we have
\begin{align*}
\Phi'_2(g_2\cdot\pi_1(z))=&\Phi'_2(\pi_1(g_2\cdot z))\\
=&\Phi_2(g_2\cdot z)\\
=&g_2\cdot \Phi_2(z)\\
=&g_2\cdot\Phi'_2(\pi_1(z)).
\end{align*}
Thus $\Phi'_2$ is $G_2$-equivariant.  Also, for any $G_1$-orbit-type stratum $(Z_1/G_1)_{(H)}$ with symplectic structure $\omega_{(H)}$, and any $\xi\in\g_2$, we have
\begin{align*}
(\pi_{(H)})^*(\xi_{(Z_1/G_1)_{(H)}}\hook\omega_{(H)})=&(\pi_{(H)})^*((\pi_{(H)})_*\xi_{(Z_1)_{(H)}}\hook\omega_{(H)})\\
=&\xi_{(Z_1)_{(H)}}\hook i_{(H)}^*\omega\\
=&i_{(H)}^*(\xi_M\hook\omega)\\
=&-d(i_{(H)}^*\Phi_2^{\xi})\\
=&-\pi_{(H)}^*d(\langle\xi,\Phi'_2|_{(Z_1/G_1)_{(H)}}\rangle)
\end{align*}
where $\langle,\rangle$ is the pairing on $\g\times\g^*$. Since $\pi_{(H)}^*$ is injective on forms, we have that $$\xi_{(Z_1/G_1)_{(H)}}\hook\omega_{(H)}=-d(\langle\xi,\Phi'_2|_{(Z_1/G_1)_{(H)}}\rangle).$$
Thus, $\Phi'_2$ restricts to a momentum map on each stratum.  Since $\Phi'_2$ is $G_2$-invariant, $Z_{12}:=(\Phi'_2)^{-1}(0)$ is a $G_2$-space, and so the quotient $Z_{12}/G_2$ is well-defined.  Let $\pi_{12}:Z_{12}\to Z_{12}/G_2$ be the quotient map.\\

Now, define $Z:=\Phi^{-1}(0)$.  This is $G$-invariant, and $Z/G$ is the symplectic quotient coming from the $G$-action on $(M,\omega)$.  Let $\pi_Z:Z\to Z/G$ be the quotient map.  Note that $Z=\Phi_1^{-1}(0)\cap\Phi_2^{-1}(0)$.  We have the following commutative diagram.

$$\xymatrix{
Z \ar[rr] \ar[dd]_{\pi_Z} & & Z_1 \ar[d]^{\pi_1} \\
 & Z_{12} \ar[r] \ar[d]^{\pi_{12}} & Z_1/G_1 \\
Z/G & Z_{12}/G_2 \ar[l]^{\Psi} & \\
}$$

\begin{lemma}
There is a bijection $\Psi:Z_{12}/G_2\to Z/G$.
\end{lemma}

\begin{proof}
This appears in \cite{lerman-sjamaar91}, but we repeat it here. Let $x\in Z_{12}/G_2$.  Then there is some $y\in Z_{12}\subseteq Z_1/G_1$ such that $x=G_2\cdot y$ as an equivalence class.  Since $y\in Z_1/G_1$, there is some $z\in Z_1$ such that $\pi_1(z)=y$.  Define $\Psi(x)=\pi_Z(z)$.  Due to its construction $\Psi$ is well-defined.  We next show that it is injective.\\

Let $x_1,x_2\in Z_{12}/G_2$ such that $\Psi(x_1)=\Psi(x_2)$.  Then there exist $z_1,z_2\in Z$ such that $\pi_Z(z_1)=\pi_Z(z_2)$ and $\pi_{12}\circ\pi_1(z_i)=x_i$ (for $i=1,2$).  But then there is a pair $(g_1,g_2)\in G$ such that $z_2=g_2\cdot g_1\cdot z_1$.  We have
\begin{align*}
\pi_{12}\circ\pi_1(z_2)=&\pi_{12}\circ\pi_1(g_2\cdot g_1\cdot z_1)\\
=&\pi_{12}(g_2\cdot\pi_1(g_1\cdot z_1))\\
=&\pi_{12}(g_2\cdot\pi_1(z_1))\\
=&\pi_{12}\circ\pi_1(z_1).
\end{align*}
Thus, $x_1=x_2$.  Surjectivity of $\Psi$ is clear: for any $w\in Z/G$, there is some $z\in Z$ such that $w=\pi_Z(z)$.  But then $w=\Psi(\pi_{12}\circ\pi_1(z))$.
\end{proof}

\begin{proposition}\labell{p:redinstages}
$\Psi$ is a diffeological diffeomorphism.
\end{proposition}

\begin{proof}
Fix a plot $p:U\to Z_{12}/G_2$.  Then for any $u\in U$ there is an open neighbourhood $V\subseteq U$ of $u$ and a plot $q:V\to Z_{12}\subseteq Z_1/G_1$ such that $p|_V=\pi_{12}\circ q$.  Shrinking $V$ if necessary, there is a plot $r:V\to Z_1$ such that $q=\pi_1\circ r$.  But since $q(V)\subseteq(\Phi'_2)^{-1}(0)$, we have $r(V)\subseteq\Phi_2^{-1}(0)$.  Hence, $r$ is a plot of $Z$, and so $\pi_Z\circ r$ is a plot of $Z/G$.  But $\pi_Z\circ r$ is precisely $\Psi\circ p|_V$.  Thus, $\Psi\circ p$ is a plot of $Z/G$, and $\Psi$ is diffeologically smooth.  Since $\Psi$ is a bijection, it induces an injection on plots of $Z_{12}/G_2$ to plots of $Z/G$.\\

Now, fix a plot $\tilde{p}:U\to Z/G$.  Then, for any $u\in U$ there is some open neighbourhood $V\subseteq U$ of $u$ and some plot $q:V\to Z$ such that $p|_V=\pi_Z\circ q$.  But then $\pi_{12}\circ\pi_1\circ q$ is a plot of $Z_{12}/G_2$.  But this is equal to $\Psi^{-1}\circ p|_V$.  Hence, $\Psi^{-1}$ is diffeologically smooth.
\end{proof}

Thus, up to diffeomorphism, the symplectic quotient $Z/G$ can be obtained in stages by first taking the symplectic quotient with respect to $G_1$, and then with respect to $G_2$.  All of the above arguments also work when the roles of $G_1$ and $G_2$ are switched, result in another quotient space $Z_{21}/G_1$.

\begin{remark}
In \cite{lerman-sjamaar91}, Sjamaar-Lerman show that the three symplectic quotients $Z_{12}/G_2$, $Z_{21}/G_1$ and $Z/G$ are all functionally diffeomorphic with respect to the quotient differential structures on the spaces (that is, they are all isomorphic to $\CIN(Z/G)$).  Moreover, these diffeomorphisms preserve the Poisson structures on these differential structures (see \hypref{d:poissympquot}{Definition}).
\end{remark}

Looking carefully at the definition, Sjamaar differential forms are not defined on the quotients $Z_{12}/G_1$ and $Z_{21}/G_2$. However, due to \hypref{t:sjamaar}{Theorem} and \hypref{p:redinstages}{Proposition}, in the case that $0\in\g_1^*\times\g_2^*$ is a regular value for $\Phi$, diffeological forms defined on these two quotients are naturally isomorphic to Sjamaar forms defined on $Z/G$.

\section{Orbifolds}\labell{s:orbifolds}

In this section we explore how differential forms on orbifolds compare to forms defined diffeologically.  The original definition of an orbifold is given by Satake \cite{satake56}, \cite{satake57}, although he calls them ``$V$-manifolds''.  A more general definition is used by Haefliger in \cite{haefliger84}.  A $V$-manifold, or an orbifold in the sense of Haefliger, are both ``diffeological orbifolds'' as defined by Iglesias-Zemmour, Karshon, and Zadka in \cite{ikz10}.  We give this definition here.

\begin{definition}[Diffeological Orbifold]
A \emph{diffeological orbifold} of dimension $n$ is a diffeological space which is locally diffeologically diffeomorphic to a quotient $\RR^n/\Gamma$ for some finite group $\Gamma$ acting linearly on $\RR^n$.  Maps between diffeological orbifolds are given by diffeologically smooth maps.
\end{definition}

In their paper, the Iglesias-Zemmour, Karshon, and Zadka show that $V$-manifolds and orbifolds in the sense of Haefliger are naturally diffeological orbifolds.  The reason we use Haefliger's definition for an orbifold below instead of that of a diffeological orbifold is because the common definition of differential forms on an orbifold is defined in the literature in terms of Satake's and Haefliger's definitions, and not in terms of diffeological orbifolds.  The goal of this section is to show that orbifold differential forms are naturally isomorphic to diffeological forms on the space.\\

Fix a topological space $X$ and a non-negative integer $n$.

\begin{definition}[Orbifold (Haefliger)]
\noindent
\begin{enumerate}
\item An $n$-dimension \emph{orbifold chart} on $X$ is a triplet $(\tilde{U},\Gamma,\phi)$ where $\tilde{U}\subseteq\RR^n$ is an open ball, $\Gamma$ is a finite group of smooth automorphisms of $\tilde{U}$, and a $\Gamma$-invariant map $\phi:\tilde{U}\to X$ that induces a homeomorphism $\tilde{U}/\Gamma\to\varphi(\tilde{U})$.
\item An \emph{embedding} $\lambda:(\tilde{U},\Gamma,\phi)\to(\tilde{V},\Gamma',\psi)$ between two charts is a smooth embedding $\lambda:\tilde{U}\to\tilde{V}$ such that $\psi\circ\lambda=\phi$.
\item An $n$-dimensional orbifold atlas on $X$ is a family $\mathcal{U}$ of $n$-dimensional orbifold charts that (a) cover $X$; (b) are locally compatible, meaning that for any two charts $(\tilde{U},\Gamma,\phi)$ and $(\tilde{V},\Gamma',\psi)$ in $\mathcal{U}$ we have for any $x\in\phi(\tilde{U})\cap\psi(\tilde{V})$ an open neighbourhood $W\subseteq\phi(\tilde{U})\cap\psi(\tilde{V})$ of $x$ and a chart $(\tilde{W},\Gamma'',\mu)$ for $W$ such that there exist embeddings $(\tilde{W},\Gamma'',\mu)\to(\tilde{U},\Gamma,\phi)$ and $\tilde{W},\Gamma'',\mu)\to(\tilde{V},\Gamma',\psi)$.
\item An orbifold atlas $\mathcal{U}$ \emph{refines} another orbifold atlas $\mathcal{V}$ if for any chart in $\mathcal{U}$, there is an embedding of the chart into a chart of $\mathcal{V}$.  If there exists a common refinement of $\mathcal{U}$ and $\mathcal{V}$, then we say that the two atlases are \emph{equivalent}. This forms an equivalence relation on all atlases of $X$.
\item An \emph{(effective) orbifold} $(X,[\mathcal{U}])$ of dimension $n$ is a paracompact, Hausdorff space $X$ equipped with an equivalence class of $n$-dimensional atlases $[\mathcal{U}]$.  Here, for each chart $(\tilde{U},\Gamma,\phi)$, $\Gamma$ acts smoothly and effectively on $\tilde{U}$.  Note that there exists a unique maximal atlas in each equivalence class, and we will now assume that we are working with this atlas, instead of the corresponding equivalence class.
\item Let $(X,\mathcal{U})$ and $(Y,\mathcal{V})$ be orbifolds.  Then a map $F:X\to Y$ is \emph{orbifold smooth} if for any $x\in X$, there exist charts $(\tilde{U},\Gamma,\phi)$ about $x$ and $(\tilde{V},\Gamma',\psi)$ about $F(x)$ such that $F(\phi(\tilde{U}))\subseteq\psi(\tilde{V})$ and there exists a map $\tilde{F}:\tilde{U}\to\tilde{V}$ such that $\psi\circ\tilde{F}=F\circ\phi$.  If $F$ is orbifold smooth and invertible with orbifold smooth inverse, then $F$ is an \emph{orbifold diffeomorphism}.
\end{enumerate}
\end{definition}

\begin{definition}[Tangent and Cotangent Bundles of Orbifolds]
Fix a chart $(\tilde{U},\Gamma,\phi)\in\mathcal{U}$ of an orbifold $(X,\mathcal{U})$.  Note that $\Gamma$ acts on $T\tilde{U}$.  We thus get a map $p:T\tilde{U}/G\to\phi(\tilde{U})$ that factors through the homeomorphism induced by $\phi$.  The fibres are given by $$p^{-1}(x)=\{\Gamma\cdot(y,v)~|~(y,v)\in T_y\tilde{U},~\phi(y)=x\}\cong T_y\tilde{U}/\Gamma_y$$ where $\Gamma_y$ is the stabiliser of $y$. The fibre over $x$ is called the \emph{tangent space} at $x$, denoted $T_xX$, and the collection of tangent spaces over all $x\in X$ is called the \emph{tangent bundle}, denoted $TX$.  This is a $2n$-dimensional orbifold, and the map $p:TX\to X$ sending a point in a tangent space to its base point, is orbifold smooth.  Similarly, the \emph{cotangent bundle} $T^*X$ is defined by taking the dual approach: the fibre over $x$ is isomorphic to $T^*_y\tilde{U}/\Gamma_y$ where $\phi(y)=x$.  We can also take wedge powers: $\bigwedge^kT^*X$ has fibres given by $$\bigwedge^kT_x^*X\cong\left(\bigwedge^kT_y^*\tilde{U}\right)\Big/\Gamma_y.$$
\end{definition}

\begin{definition}[Orbifold Differential Forms]
A \emph{$k$-form} $s$ is a section of $\bigwedge^kT^*X$.  That is, a map $s:X\to\bigwedge^kT^*X$ satisfying the following two conditions.
\begin{enumerate}
\item For each chart $(\tilde{U},\Gamma,\phi)$ there is an equivariant smooth section $\tilde{s}:\tilde{U}\to\bigwedge^kT^*\tilde{U}$ which descends to $s$ on $\phi(\tilde{U})$.
\item For any two charts $(\tilde{U},\Gamma,\phi)$ and $(\tilde{V},\Gamma',\psi)$ with $x\in\phi(\tilde{U})\cap\psi(\tilde{V})$, the two equivariant sections $\tilde{s}$ and $\tilde{s}'$ of $\bigwedge^kT^*\tilde{U}$ and $\bigwedge^kT^*\tilde{V}$, respectively, are compatible.  That is, for a chart $(\tilde{W},\Gamma'',\mu)$ about $x$ with embeddings $\lambda:\tilde{W}\to\tilde{U}$ and $\lambda':\tilde{W}\to\tilde{V}$, and an equivariant section $\tilde{s}'':\tilde{W}\to\bigwedge^kT^*\tilde{W}$, we have $\lambda^*\tilde{s}=\tilde{s}''$ and $(\lambda')^*\tilde{s}'=\tilde{s}''$.
\end{enumerate}
\end{definition}

Let $\Gamma$ be a finite group acting smoothly on a manifold $M$.  Then, $M/\Gamma$ is an orbifold, and by the definition above, differential forms on $M/\Gamma$ correspond to invariant (basic) forms on $M$.  But then, by \hypref{t:geomquot}{Theorem}, these correspond to diffeological forms on $M/\Gamma$ equipped with the quotient diffeology; that is, the orbifold viewed as a diffeological orbifold.\\

Also, note that if $M$ is a symplectic manifold and $\Gamma$ acts symplectically on $M$, then the action is automatically Hamiltonian (trivially, since the Lie algebra to $\Gamma$ is the one point vector space $\{0\}$, and so the momentum map is constant).  The symplectic quotient is equal to the geometric quotient.  By \hypref{t:sjamaar}{Theorem}, in this case the Sjamaar forms are naturally isomorphic to the diffeological forms on the geometric quotient, which by the above are naturally isomorphic to the orbifold differential forms. 
\chapter{Vector Fields on Subcartesian Spaces}\labell{ch:vectorfields}

In this chapter we focus on vector fields on subcartesian spaces.  Recall that a subcartesian space is a differential space that is locally diffeomorphic to subsets of Euclidean spaces (see \hypref{d:subcart}{Definition}).  Now, for manifolds, the Lie algebra of derivations of the ring of smooth functions is canonically isomorphic to the Lie algebra of smooth sections of the tangent bundle.  Moreover, each of these derivations induces a local flow on the manifold. For subcartesian spaces, this is no longer the case.  While it is still true that derivations of the ring of smooth functions match smooth sections of the Zariski tangent bundle, not all of these derivations will induce local flows.  We reserve the term ``vector field'' for a derivation that does.\\

\'Sniatycki shows in \cite{sniatycki03} that the set of all vector fields on a subcartesian space $S$ induces a partition of $S$ into manifolds.  This generalises a well-known theorem in control theory, called ``the orbit theorem'' (see \cite{jurdjevic97}).  \'Sniatycki applies this to smooth stratified spaces, and in particular, to geometric quotients coming from compact Lie group actions.\\

The main idea behind this chapter is to apply this theory of vector fields to geometric and symplectic quotients, and to show that smooth stratified spaces form a full subcategory of subcartesian spaces equipped with locally complete families of vector fields (see \hypref{c:orbitalcat}{Corollary}).  I construct locally complete Lie algebras of vector fields whose orbits induce the orbit-type stratifications on $M$ and $M/G$ in the case of a compact Lie group action $G\circlearrowright M$.  The former of these is new (see \hypref{d:A}{Definition} and \hypref{t:singfolA}{Theorem}) whereas the latter is due to the work of \'Sniatycki \cite{sniatycki03}.  I also give similar families of vector fields that induce the orbit-type stratifications on the symplectic quotient and the zero-level of the momentum map for a compact Hamiltonian action.  The family on the zero-level is new (see \hypref{d:AZ}{Definition} and \hypref{p:vectZorbits}{Proposition}) whereas the family on the symplectic quotient is the work of Sjamaar-Lerman \cite{lerman-sjamaar91}.\\

If $0$ is a regular value of the momentum map, then it is not hard to show that the orbit-type stratification of the symplectic quotient is intrinsic to the differential structure on the symplectic quotient (see \hypref{t:regvalue}{Theorem}).  But it is unknown if this is still the case for $0$ a critical value of the momentum map.  We state this as an open question later in the chapter (see \hypref{q:critvalue}{Question}).\\

For the purposes of this chapter, by ``smooth'' map between subcartesian spaces, I mean ``functionally smooth''.

\section{The Zariski Tangent Bundle}\labell{s:zariski}

In this section, we review the basics of subcartesian space theory.  Let $S$ be a subcartesian space.

\begin{definition}[Zariski Tangent Bundle]
Given a point $x\in S$, a \emph{derivation} of $\CIN(S)$ at $x$ is a linear map $v:\CIN(S)\to\RR$ that satisfies Leibniz' rule: for all $f,g\in\CIN(S)$, $$v(fg)=f(x)v(g)+g(x)v(f).$$  The set of all derivations of $\CIN(S)$ at $x$ forms a vector space, called the \emph{(Zariski) tangent space} of $x$, and is denoted $T_xS$. Define the \emph{(Zariski) tangent bundle} $TS$ to be the (disjoint) union $$TS:=\bigcup_{x\in S}T_xS.$$  Denote the canonical projection $TS\to S$ by $\tau$.
\end{definition}

$TS$ is a subcartesian space with its differential structure generated by functions $f\circ\tau$ and $df$ where $f\in\CIN(S)$ and $d$ is the differential operator $df(v):=v(f)$.  The projection $\tau$ is smooth with respect to this differential structure.  Given a chart $\varphi:U\to\tilde{U}\subseteq\RR^n$ on $S$, $(\varphi\circ\tau,\varphi_*|_{\varphi\circ\tau}):TS\to T\RR^n\cong\RR^{2n}$ is a fibrewise linear chart on $TS$.  See my paper with Lusala and \'Sniatycki \cite{lsw10} for more details.  We will denote this chart by $\varphi_*$ henceforth.

\begin{definition}[Pushforward]
Let $R$ and $S$ be subcartesian spaces, and let $F:R\to S$ be a smooth map.  Then there is an induced fibrewise linear smooth map $F_*:TR\to TS$ defined by $$(F_*v)f=v(F^*f)$$ for all $v\in TR$ and $f\in\CIN(S)$.  $F_*$ satisfies the following commutative diagram.

$$\xymatrix{
TR \ar[r]^{F_*} \ar[d]_{\tau} & TS \ar[d]^{\tau} \\
R \ar[r]_{F} & S \\
}$$

$F_*$ is called the \emph{pushforward} of $F$, and is sometimes denoted as $dF$ or $TF$.
\end{definition}

We recall some notation.  For a subset $A\subseteq\RR^n$, let $\mathfrak{n}(A)$ be the ideal of the ring of smooth functions on $\RR^n$ consisting of functions that vanish on $A$.

\begin{proposition}[Local Representatives of Vectors]\labell{p:charTS}
Let $x\in S$ and let $\varphi:U\to\tilde{U}\subseteq\RR^n$ be a chart about $x$.  Then, $\tilde{v}\in T_{\varphi(x)}\RR^n$ is equal to $\varphi_*v$ for some $v\in T_xS$ if and only if $\tilde{v}(\mathfrak{n}(\varphi(U)))=\{0\}$.
\end{proposition}
\begin{proof}
See \cite{lsw10}
\end{proof}

\begin{definition}[Derivations of $\CIN(S)$]
A \emph{(global) derivation} of $\CIN(S)$ is a linear map $X:\CIN(S)\to\CIN(S)$ that satisfies Leibniz' rule: for any $f,g\in\CIN(S)$, $$X(fg)=fX(g)+gX(f).$$  Denote the $\CIN(S)$-module of all derivations by $\Der\CIN(S)$.
\end{definition}

\begin{proposition}[$\Der\CIN(S)$ is a Lie Algebra]
The set of derivations of $\CIN(S)$ is a Lie algebra under the commutator bracket, and can be identified with the smooth sections of $\tau:TS\to S$.
\end{proposition}
\begin{proof}
See \cite{lsw10}.
\end{proof}

\begin{proposition}[Local Representatives of Derivations]\labell{p:charDer}
Let $x\in S$, and let $\varphi:U\to\tilde{U}\subseteq\RR^n$ be a chart about $x$.  Let $\tilde{X}\in\Der\CIN(\RR^n)$.  Then $\tilde{X}$ satisfies $$\varphi_*(X|_U)=\tilde{X}|_{\tilde{U}}$$ for some derivation $X\in\Der\CIN(S)$ if and only if $$\tilde{X}(\mathfrak{n}(\tilde{U}))\subseteq\mathfrak{n}(\tilde{U}).$$ Moreover, for any $X\in\Der\CIN(S)$, there exist an open neighbourhood $V\subseteq U$ of $x$ and $\tilde{X}\in\Der\CIN(\RR^n)$ satisfying $\varphi_*(X|_V)=\tilde{X}|_{\varphi(V)}$. We call $\tilde{X}$ a \emph{local extension} or a \emph{local representative} of $X$ with respect to $\varphi$.
\end{proposition}
\begin{proof}
See \cite{lsw10}.
\end{proof}

\begin{definition}[Locally Trivial Surjections]
Let $R$ and $S$ be subcartesian spaces, and let $f$ be a surjective smooth map between them.  Then $f:R\to S$ is \emph{locally trivial} if for every $x\in S$ there exist an open neighbourhood $U\subseteq S$ of $x$, a subcartesian space $F$, and a diffeomorphism $\psi:f^{-1}(U)\to U\times F$ such that the following diagram commutes ($\operatorname{pr}_1$ being the projection of the first component.)

$$\xymatrix{
f^{-1}(U) \ar[rr]^{\psi} \ar[dr]_f & & U\times F \ar[dl]^{\operatorname{pr}_1}\\
& U & \\
}$$
\end{definition}

\begin{theorem}[Local Triviality of $TS$]\labell{t:TSloctriv}
Let $S$ be a subcartesian space.  There exists an open dense subset $U\subseteq S$ such that $\tau|_U:TS|_U\to U$ is locally trivial.
\end{theorem}
\begin{proof}
See \cite{lsw10}.
\end{proof}

\begin{corollary}
The $k$th exterior product of fibres of $TS$ over $S$ are also locally trivial over an open dense subset of $S$.
\end{corollary}

\section{Vector Fields on Subcartesian Spaces}

In this section we review the theory of vector fields on subcartesian spaces, developed by \'Sniatycki in \cite{sniatycki03}.

\begin{definition}[Integral Curves]
Fix $X\in\Der\CIN(S)$ and $x\in S$.  A \emph{maximal integral curve} $\exp(\cdot X)(x)$ of $X$ through $x$ is a smooth map from a connected subset $I^X_x\subseteq\RR$ containing 0 to $S$ such that $\exp(0X)(x)=x$, the following diagram commutes,
$$\xymatrix{
TI^X_x \ar[rr]_{\exp(\cdot X)(x)_*} && TS\\
I^X_x \ar[u]^{\frac{d}{dt}} \ar[rr]^{\exp(\cdot X)(x)} && S \ar[u]_{X}
}$$
and such that $I_x^X$ is maximal among the domains of all such curves.  In particular, for all $f\in\CIN(S)$ and $t\in I^X_x$, $$\frac{d}{dt}(f\circ\exp(tX)(x))=(Xf)(\exp(tX)(x)).$$ We adopt the convention that the map $c:\{0\}\to S:0\mapsto c(0)$ is an integral curve of every global derivation of $\CIN(S)$.
\end{definition}

\begin{theorem}[ODE Theorem for Subcartesian Spaces -- \'Sniatycki]\labell{t:ode}
Let $S$ be a subcartesian space, and let $X\in\Der\CIN(S)$.  Then, for any $x\in S$, there exists a unique maximal integral curve $\exp(\cdot X)(x)$ through $x$.
\end{theorem}
\begin{proof}
See \cite{sniatycki02} and \S4 Theorem 1 of \cite{sniatycki03}.
\end{proof}

\begin{proposition}[Local Representatives of Integral Curves]\labell{p:intcurve}
Let $\varphi:U\to\tilde{U}\subseteq\RR^n$ be a chart on $S$, $X\in\Der\CIN(S)$ and $\tilde{X}\in\Der\CIN(\RR^n)$ such that $$\varphi_*(X|_U)=\tilde{X}|_{\tilde{U}}.$$  Then for all $x\in S$ and $t\in I^{X|_U}_x$,  $$\varphi(\exp(tX)(x))=\exp(t\tilde{X})(\varphi(x)).$$
\end{proposition}

\begin{proof}
Denote by $J$ the open subset of $I^X_x$ such that for every $t\in J$, $\exp(tX)(x)\in U$. Define $\gamma:J\to\tilde{U}:t\mapsto\varphi(\exp(tX)(x))$.  Then,
\begin{align*}
\frac{d}{dt}\Big|_{t=0}(\gamma(t))=&\varphi_*(X|_x)\\
=&\tilde{X}|_{\varphi(x)}.
\end{align*}
Applying the ODE theorem, $\gamma(t)=\exp(t\tilde{X})(\varphi(x))$.
\end{proof}

Fix a derivation $X\in\Der\CIN(S)$.  Let $A^X\subseteq\RR\times S$ be defined as $$A^X:=\coprod_{x\in S}I_x^X.$$  Then there is an induced smooth map $A^X\to S$ whose restriction to each fibre $A^X\cap(\RR\times\{x\})$ is the domain $I^X_x$ of the maximal integral curve $\exp(\cdot X)(x)$.  %We thus have the following commutative diagram.
%$$\xymatrix{
%TA^X \ar[rr]_{\exp(\cdot X)(\cdot)_*} && TS\\
%A^X \ar[u]^{\frac{d}{dt}} \ar[rr]^{\exp(\cdot X)(\cdot)} && S \ar[u]_{X}
%}$$

\begin{definition}[Local Flows]
Let $D$ be a subset of $\RR\times S$ containing $\{0\}\times S$ such that $D\cap(\RR\times\{x\})$ is connected for each $x\in S$.  A map $\phi:D\to S$ is a \emph{local flow} if $D$ is open, $\phi(0,x)=x$ for each $x\in S$, and $\phi(t,\phi(s,x))=\phi(t+s,x)$ for all $x\in S$ and $s,t\in\RR$ for which both sides are defined.
\end{definition}

\begin{remark}
If $S$ is a smooth manifold, then every derivation $X$ admits a local flow $\exp(\cdot X)(\cdot)$ sending $(t,x)$ to $\exp(tX)(x)$.  This is not the case with subcartesian spaces.  Indeed, consider the closed ray $[0,\infty)$, and the global derivation $X=\partial_x$.  Then the domain $D$ of $\exp(\cdot X)(\cdot)$ is not an open subset of $\RR\times[0,\infty)$.  Indeed, $D\cap(\RR\times\{x\})=[-x,\infty)\times\{x\}$ for each $x\in\RR$.  Thus, $D=\{(t,x)\in\RR^2~|~t\geq-x,~x\geq0\}$.  This motivates the following definition.
\end{remark}

\begin{definition}[Vector Fields]
A \emph{vector field} on $S$ is a derivation $X$ of $\CIN(S)$ such that $A^X$ is open in $\RR\times S$.  Equivalently, the map $(t,x)\mapsto\exp(tX)(x)$ defined on $A^X$ is a local flow.  Here let us emphasise that $\exp(tX)(x)$ is the maximal integral curve through $x$.  Denote the set of all vector fields on $S$ by $\Vect(S)$.
\end{definition}

\begin{remark}
Given a vector field $X$ on $S$, since $A^X$ is open, the domain of each of its maximal integral curves is open.  Note, however, that the converse is not true: if $X$ is a global derivation and each of its maximal integral curves has an open domain, it is not necessarily true that $X$ is a vector field. For a counterexample, see \hypref{x:loccmpctneeded}{Example}.
\end{remark}

For the important proposition to come, we recall the concepts of ``locally closed'' and ``locally compact''.  In the literature (for example, \cite{sniatycki03}), the notion of locally closed is used for subsets of $\RR^n$ (in particular, for differential subspaces of $\RR^n$).  ``Locally compact'', however, can be used for subcartesian spaces (or any topological space), not just differential subspaces of $\RR^n$.  It also tends to be more widely used in the literature.  We show in the following lemma that, for differential subspaces of $\RR^n$, these two concepts coincide. Before stating and proving the lemma, we recall the definitions of locally compact and locally closed subsets.\\

\begin{itemize}
\item Let $S\subseteq\RR^n$.  $S$ is \emph{locally compact} if for every $x\in S$ there exist a relatively open neighbourhood $U\subseteq S$ of $x$ and a compact set $K\subseteq S$ such that $U\subseteq K$.

\item Let $S\subseteq\RR^n$.  $S$ is \emph{locally closed} if for every $x\in S$ there exist an open neighbourhood $V\subseteq\RR^n$ of $x$ and a closed set $C\subseteq\RR^n$ such that $V\cap C$ is a relatively open neighbourhood of $x$ in $S$.
\end{itemize}

\begin{lemma}
Let $S\subseteq\RR^n$.  Then $S$ is locally closed if and only if $S$ is locally compact.
\end{lemma}

\begin{proof}
Assume that $S$ is locally compact, and fix $x\in S$.  Then, there exist an open neighbourhood $U\subseteq S$ of $x$ and a compact $K\subseteq S$ such that $U\subseteq K$.  There exists an open neighbourhood $V\subseteq\RR^n$ of $x$ such that $U=V\cap S$, and $K$ is a compact subset of $\RR^n$ and hence closed.  $$V\cap K\subseteq V\cap S=U\subseteq V\cap K$$
and so $V\cap K=U$.  Hence, $S$ is locally closed.\\

Conversely, assume $S$ is locally closed, and fix $x\in S$.  There exist an open neighbourhood $V\subseteq\RR^n$ of $x$ and a closed subset $C\subseteq\RR^n$ such that $V\cap C$ is an open neighbourhood of $x$ contained in $S$.  Let $B\subseteq\RR^n$ be the open ball of radius $\epsilon>0$ centred at $x$.  Then, $B\cap V\cap C$ is an open neighbourhood of $x$ in $S$.  Choosing $\epsilon$ to be sufficiently small so that $\overline{B}\subseteq V$, we have $B\cap V\cap C=B\cap C$ and $\overline{B}\cap C\subseteq S$.  Since $\overline{B}$ and $C$ are closed subsets of $\RR^n$, their intersection is closed.  Since this intersection is contained in $S$, $\overline{B}\cap C$ is a closed subset of $S$.  Moreover, since $\overline{B}$ is compact in $\RR^n$, $\overline{B}\cap C$ is compact in $\RR^n$ as well.\\

Now, let $\{W_\alpha\}_{\alpha\in A}$ be an open cover of $\overline{B}\cap C$ in $S$. Then, for each $\alpha$, there exists an open set $\tilde{W}_\alpha$ such that $W_\alpha=\tilde{W}_\alpha\cap S$. Thus, the collection of open sets $\{\tilde{W}_\alpha\}_{\alpha\in A}$ forms an open cover of $\overline{B}\cap C$ in $\RR^n$.  Since $\overline{B}\cap C$ is compact in $\RR^n$, there is a finite subcover $\{\tilde{W}_{\alpha_i}\}$ of $\{\tilde{W}_\alpha\}$ that covers $\overline{B}\cap C$.  But then for each $\alpha_i$, $W_{\alpha_i}:=\tilde{W}_{\alpha_i}\cap S$ is an open subset of $S$, contained in $\{W_\alpha\}$.  We conclude that the collection $\{W_{\alpha_i}\}$ is a finite subcover of $\{W_\alpha\}$ covering $\overline{B}\cap C$, and hence $\overline{B}\cap C$ is compact as a subset of $S$.
\end{proof}

Note that a subcartesian space can be locally compact, which extends the notion of local closedness beyond differential subspaces of $\RR^n$.

\begin{proposition}[Integral Curve Domains -- \'Sniatycki]\labell{p:opendomain}
Let $S$ be a locally compact subcartesian space.  A derivation $X$ of $\CIN(S)$ is a vector field if and only if the domain of each of its maximal integral curves is open.
\end{proposition}
\begin{proof}
See \S4 Proposition 3 of \cite{sniatycki03}.
\end{proof}

%Thus, for a vector field $X$ on $S$,
%$$\xymatrix{
%TD^X \ar[rr]_{\exp(\cdot X)(\cdot)_*} &&TS\\
%D^X \ar[u]^{\frac{d}{dt}} \ar[rr]^{\exp(\cdot X)(\cdot)} && S \ar[u]_{X}
%}$$
%is a commutative diagram and
%$$\xymatrix{
%TD^X \ar[d]_{\tau_D} \ar[rr]_{\exp(\cdot X)(\cdot)_*} && TS \ar[d]^{\tau_S}\\
%D^X \ar[rr]^{\exp(\cdot X)(\cdot)} && S
%}$$
%is ...

\begin{example}[\'Sniatycki \cite{sniatycki03}]\labell{x:loccmpctneeded}
Let $S$ be the differential subspace of $\RR^2$ given by $$S=\{(x,y)~|~x^2+(y-1)^2<1\}\cup\{(x,y)~|~y=0\}.$$  Consider the global derivation given by the restriction of $\partial_x$ to $S$.  Then the domain of each maximal integral curve of $\partial_x$ is open; however, at the origin, the integral curve does not induce a local diffeomorphism.
\end{example}

\section{Locally Complete Families of Vector Fields and Smooth Stratified Spaces}

We are interested in using families of vector fields in order to obtain a ``nice'' partition of a subcartesian space.  The condition needed to achieve this on these families is defined next.  We then give examples.

\begin{definition}[Locally Complete Families of Vector Fields]
A family of vector fields $\mathcal{F}\subseteq\Vect(S)$ is \emph{locally complete} if for every $X,Y\in\mathcal{F}$, every $x\in S$ and every $t\in\RR$ such that $(\exp(tX)_*Y)|_x$ is well-defined, there exist an open neighbourhood $U$ of $x$ and a vector field $Z\in\mathcal{F}$ such that $\exp(tX)_*Y|_U=Z|_U$.
\end{definition}

\begin{remark}
Note that for $f\in\CIN(S)$ and $X,Y\in\mathcal{F}$ where $\mathcal{F}$ is a locally complete family of vector fields, $x\in S$ and $s,t\in\RR$, we have (where it is defined)
\begin{align*}
\frac{d}{ds}f(\exp(tX)(\exp(sY)(x)))=&(\exp(tX)_*(Y|_{\exp(sY)(x)}))f\\
=&((\exp(tX)_*Y)f)(\exp(tX)(\exp(sY)(x))).
\end{align*}
For fixed $t$, there exists an open neighbourhood $U$ of $x$ on which the local flow of $(\exp(tX)_*Y)|_U$ is equal to $s\mapsto\exp(tX)(\exp(sY)(y))$ for $y\in U$.
\end{remark}

\begin{example}[Not Locally Complete]\labell{x:lcexample}
Consider $S=\RR^2$, and let $\mathcal{F}$ be the family of all $\RR$-linear combinations of the two vector fields $\partial_x$ and $x\partial_y$.  This family is not locally complete, as one can check that $\exp(tx\partial_y)_*\partial_x=\partial_x+t\partial_y$ is not contained in $\mathcal{F}$ for any $t\neq 0$.
\end{example}

\begin{proposition}[\'Sniatycki]\labell{p:vectloccompl}
$\Vect(S)$ is locally complete.
\end{proposition}
\begin{proof}
See \S4 Theorem 2 of \cite{sniatycki03}.
\end{proof}

We will later give examples of subcartesian spaces equipped with smooth stratifications on which we apply the theory of vector fields, but in order to do this we will need some more terminology coming from the theory of stratified and decomposed spaces.  An introduction to decomposed and stratified spaces is given in \hypref{s:stratified}{Section}.

\begin{definition}[Refinements of Decomposed Spaces]
Fix a differential space $X$ with smooth decompositions $\mathcal{D}_1$ and $\mathcal{D}_2$.  $\mathcal{D}_1$ is a \emph{refinement} of $\mathcal{D}_2$, denoted $\mathcal{D}_1\geq\mathcal{D}_2$ if for every piece $P_1\in\mathcal{D}_1$, there exists $P_2\in\mathcal{D}_2$ such that $P_1\subseteq P_2$.  This induces a partial ordering on the set of decompositions on $X$. We say that $\mathcal{D}$ is \emph{minimal} if for any $\mathcal{D}'$ such that $\mathcal{D}\geq\mathcal{D}'$, we have $\mathcal{D}=\mathcal{D}'$.
\end{definition}

\begin{example}
The square $[0,1]^2$ with the decomposition given in \hypref{x:square}{Example} is minimal.
\end{example}

\begin{theorem}[Bierstone]\labell{t:bierstone}
If $G$ is a compact Lie group acting on a manifold $M$, then the orbit-type stratification of $M/G$ is minimal.
\end{theorem}
\begin{proof}
See \cite{bierstone75} and \cite{bierstone80}.
\end{proof}

\begin{definition}[Stratified Vector Fields]
Let $S$ be a smooth stratified space.  Let $X\in\Vect(S)$.  If for each stratum $P$ of $S$, $X|_P$ is a smooth vector field on $P$ as a smooth manifold, then we call $X$ \emph{stratified}. Denote the set of all stratified vector fields on $S$ by $\Vect_{strat}(S)$.
\end{definition}

\begin{remark}
Different terminology and definitions appear in the literature.  For example, in \cite{sniatycki03}, \'Sniatycki defines a smooth stratified space as a (topological) decomposed space equipped with a special atlas of charts such that the pieces obtain their smooth structures from the atlas.  As a theorem, he proves that these are subcartesian spaces.  He does not require the smooth local triviality condition in the definition of a smooth stratified space, although he requires it in many of the lemmas and theorems.  Also, in the same article, a stratified vector field on a smooth stratified space $S$ is not necessarily a smooth section of the Zariski tangent bundle, but a continuous section that is smooth on the strata.  Instead, a strongly stratified vector field is an element of $\Vect(S)$ that restricts to a smooth section of each strata.
\end{remark}

\begin{proposition}[\'Sniatycki]\labell{p:stratloccompl}
Let $S$ be a smooth stratified space.  Then $\Vect_{strat}(S)$ forms a locally complete family.
\end{proposition}
\begin{proof}
See \S6 Lemma 11 of \cite{sniatycki03}
\end{proof}

Now, let $G$ be a compact Lie group acting on a manifold $M$. Denote by $\Vect(M)^G$ the invariant vector fields on $M$. Note that each invariant vector field $X$ induces a $G$-invariant local flow: $$g\cdot\exp(tX)(x)=\exp(tX)(g\cdot x).$$

\begin{proposition}\labell{p:vectGloccompl}
$\Vect(M)^G$ is a locally complete Lie subalgebra of $\Vect(M)$.
\end{proposition}

\begin{proof}
For any two invariant vector fields $X$ and $Y$, we have for all $g\in G$, $$g_*[X,Y]=[g_*X,g_*Y]=[X,Y],$$ and for $x\in M$,  $$g\cdot\exp(tX)(\exp(sY)(x))=\exp(tX)(\exp(sY)(g\cdot x))$$ for $s,t$ such that the composition of the curves is defined. Thus $\exp(tX)_*Y$ is locally defined about $G$-orbits.  Since $\Vect(M)$ is locally complete, for any $x\in M$ there exist a vector field $Z$ on $M$ and an open neighbourhood $U$ of $x$ such that $\exp(tX)_*Y$ is defined on $U$ and $(\exp(tX)_*Y)|_U=Z|_U$.  Since $\exp(tX)_*Y$ is invariant about $x$, we can choose $U$ to be a $G$-invariant open neighbourhood.  Let $V\subset U$ be a $G$-invariant open neighbourhood of $x$ such that $\overline{V}\subset U$.  Let $b:M\to\RR$ be a $G$-invariant smooth bump function with support in $U$ and $b|_V=1$.  Then, $bZ\in\Vect(M)^G$ extends $(\exp(tX)_*Y)|_V$ to a invariant vector field on $M$.
\end{proof}

\begin{definition}
Identify $\g$ with the invariant (under left multiplication) vector fields on $G$. Let $\rho:\g\to\Vect(M)$ be the $\g$-action induced by the $G$-action.
\end{definition}

\begin{proposition}\labell{p:rhogloccompl}
$\rho(\g)$ is a locally complete Lie subalgebra of $\Vect(M)$.
\end{proposition}

\begin{proof}
Let $\xi,\zeta\in\g$, and let $\xi_M=\rho(\xi)$ and $\zeta_M=\rho(\zeta)$.  Then, $\exp(t\xi_M)_*\zeta_M=(\Ad_{\exp(t\xi)}\zeta)_M$.
\end{proof}

Recall that for a compact Lie group $G$, its Lie algebra decomposes as a direct sum of the derived Lie subalgebra and the centre of $\g$: $$\g=[\g,\g]\oplus\mathfrak{z}(\g).$$

\begin{corollary}\labell{c:rhogloccompl}
$\rho([\g,\g])$ and $\rho(\mathfrak{z}(\g))$ are locally complete Lie subalgebras of $\Vect(M)$.
\end{corollary}

\begin{proof}
Since $[\g,\g]$ and $\mathfrak{z}(\g)$ are themselves Lie algebras corresponding to compact Lie groups, this corollary is immediate from the above lemma.
\end{proof}

\begin{definition}\labell{d:A}
Define $\mathcal{A}$ to be the smallest Lie subalgebra of $\Vect(M)$ containing $\Vect(M)^G$ and $\rho(\g)$.
\end{definition}

\begin{remark}
Note that $\mathcal{A}$, $\Vect(M)^G$ and $\rho(\g)$ are not necessarily closed under multiplication by functions in $\CIN(M)$, but $\Vect(M)^G$ is closed under multiplication by $G$-invariant smooth functions.
\end{remark}

\begin{proposition}\labell{p:A}
$\mathcal{A}$ is locally complete, and it is a direct sum of Lie algebras: $$\mathcal{A}=\rho([\g,\g])\oplus\Vect(M)^G.$$
\end{proposition}

\begin{proof}
Let $\xi\in\g$ and $X\in\Vect(M)^G$.  Then, $$[\xi_M,X]=\underset{t\to0}{\lim}\frac{\exp(t\xi_M)_*(X|_{\exp(-t\xi_M)(x)})-X|_x}{t}=0$$ since $\exp(t\xi_M)_*(X|_{\exp(-t\xi_M)(x)})=X|_x$ by left-invariance.  Thus,
\begin{equation}\labell{e:flowscommute}
\exp(t\xi_M)\circ\exp(sX)=\exp(sX)\circ\exp(t\xi_M).
\end{equation}
Now, let $\xi\in\g$ and assume for all $g\in G$ and $x\in M$, we have $$g_*(\xi_M|_x)=\xi_M|_{g\cdot x};$$ that is, $\xi_M$ is invariant.  Then, $$\frac{d}{dt}\Big|_{t=0}(g\cdot\exp(t\xi_M)(x))=\frac{d}{dt}\Big|_{t=0}\exp(t\xi_M)(g\cdot x).$$  The uniqueness property of $\exp$ implies that $$g\cdot\exp(t\xi_M)(x)=\exp(t\xi_M)(g\cdot x).$$  Hence  $(g\exp(t\xi))\cdot x=(\exp(t\xi)g)\cdot x$.  Since this is true for all $g\in G$, $\exp(t\xi)$ must be in the centre of $G$, and hence $\xi\in\mathfrak{z}(\g)$.  Thus, $$\rho(\g)\cap\Vect(M)^G=\rho(\mathfrak{z}(\g)).$$  Since $\rho$ is a Lie algebra homomorphism, from \hypref{e:flowscommute}{Equation}: $\rho(\g)=\rho([\g,\g])\oplus\rho(\mathfrak{z}(\g))$, and we obtain the direct sum structure of $\mathcal{A}$.\\
To show local completeness, by \hypref{p:vectGloccompl}{Proposition} and \hypref{p:rhogloccompl}{Proposition} it suffices to show that for any $\xi\in\g$ and $X\in\Vect(M)^G$, $\exp(t\xi_M)_*X\in\mathcal{A}$ and $\exp(tX)_*\xi_M\in\mathcal{A}$.  The former is immediate since $X$ is invariant.  The latter follows from \hypref{e:flowscommute}{Equation}:
\begin{align*}
\exp(tX)_*(\xi_M|_x)=&\frac{d}{ds}\Big|_{s=0}\exp(tX)(\exp(s\xi_M)(x))\\
=&\frac{d}{ds}\Big|_{s=0}\exp(s\xi_M)(\exp(tX)(x))\\
=&\xi_M|_{\exp(tX)(x)}.
\end{align*}
\end{proof}

\begin{definition}
A \emph{Poisson bracket} on a differential structure $\mathcal{F}$ on a differential space $X$ is a Lie bracket $\pois{}{}$ satisfying for any $f,g,h\in\mathcal{F}$: $$\pois{f}{gh}=h\pois{f}{g}+g\pois{f}{h}.$$
\end{definition}

We return to Hamiltonian group actions in order to give examples of Poisson structures.  Let $G$ be a compact Lie group acting in a Hamiltonian fashion on a connected symplectic manifold $(M,\omega)$ with momentum map $\Phi$ and $Z:=\Phi^{-1}(0)$.  We utilise again the following commutative diagram.

$$\xymatrix{
Z \ar[d]_{\pi_Z} \ar[r]^i & M \ar[d]^{\pi} \\
Z/G \ar[r]_j & M/G \\
}$$

\begin{definition}[Hamiltonian Vector Fields on $M$]
A vector field $X\in\Vect(M)$ is \emph{Hamiltonian} if there exists a function $f\in\CIN(M)$ such that $$X\hook\omega=-df.$$ In this case, we usually denote $X$ by $X_f$. Note that $f$ is unique up to a constant.
\end{definition}

\begin{example}\labell{x:stdpois}
Define $\pois{}{}$ on $(M,\omega)$ by $$\pois{f}{g}:=\omega(X_f,X_g).$$  This is the standard Poisson structure on a symplectic manifold.
\end{example}

\begin{example}
Since the $G$-action on $M$ is symplectic, for any $f,g\in\CIN(M)^G$, we have $\pois{f}{g}\in\CIN(M)^G$.  In particular, this descends to a Poisson structure $\pois{\cdot}{\cdot}_{M/G}$ on $\CIN(M/G)$.
\end{example}

Recall the orbit-type stratifications on $Z$ and $Z/G$ (refer to \hypref{d:symplquotstrat}{Definition}).

\begin{theorem}[Lerman-Sjamaar]
For each closed subgroup $H\leq G$ such that $Z_{(H)}$ is nonempty, the manifold $(Z/G)_{(H)}$ admits a symplectic form $\omega_{(H)}\in\Omega^2((Z/G)_{(H)})$ satisfying $$(\pi_{(H)})^*\omega_{(H)}=(i_{(H)})^*\omega.$$
\end{theorem}
\begin{proof}
See \cite{lerman-sjamaar91}.
\end{proof}

Since these manifolds $(Z/G)_{(H)}$ are symplectic, their rings of functions admit Poisson structures $\pois{\cdot}{\cdot}_{(H)}$ as in \hypref{x:stdpois}{Example}.  In fact, we can define a Poisson bracket on all of $Z/G$ as follows.

\begin{definition}\labell{d:poissympquot}
Let $f,g\in\CIN(Z/G)$, and let $x\in(Z/G)_{(H)}$ for some $H\leq G$.  Then define $$\pois{f}{g}_{Z/G}(x):=\pois{f|_{(Z/G)_{(H)}}}{g|_{(Z/G)_{(H)}}}_{(H)}(x).$$
\end{definition}

\begin{proposition}[Lerman-Sjamaar]
The above bracket defines a Poisson bracket on $\CIN(Z/G)$.
\end{proposition}
\begin{proof}
See \cite{lerman-sjamaar91}.
\end{proof}

\begin{definition}[Hamiltonian Vector Fields on $Z/G$]
A vector field $X\in\Vect(Z/G)$ is called \emph{Hamiltonian} if there exists $h\in\CIN(Z/G)$ such that $X=\pois{h}{\cdot}_{Z/G}$.  We will usually denote $X$ by $X_h$, and the set of all Hamiltonian vector fields by $\ham(Z/G)$.
\end{definition}

\begin{lemma}
For any $h\in\CIN(Z/G)$, the derivation $\{h,\cdot\}_{Z/G}$ is a Hamiltonian vector field.
\end{lemma}

\begin{proof}
Sjamaar-Lerman prove the existence and uniqueness of maximal integral curves of these derivations, and that they remain in the orbit-type strata (see \cite{lerman-sjamaar91}).  Since these strata are manifolds, the maximal integral curves have open domains.  Hence, by \hypref{p:opendomain}{Proposition}, they are vector fields.
\end{proof}

\begin{proposition}
$\ham(Z/G)$ is a locally complete family in $\Vect(Z/G)$.
\end{proposition}
\begin{proof}
See \S7 Proposition 4 in \cite{sniatycki03}.
\end{proof}

\section{The Orbital Tangent Bundle}

In this section, for a fixed family of vector fields we introduce a ``subbundle'' of the Zariski tangent bundle consisting of vectors that are fibrewise linear combinations of vectors in the images of vector fields in the family.  We show that the family of \emph{all} vector fields yields such a tangent bundle that is locally trivial on an open dense subset, as well as equal to the Zariski tangent bundle over an open dense subset.

\begin{definition}[Orbital Tangent Bundle]
Let $\mathcal{F}$ be a family of vector fields on $S$.  For each $x\in S$, denote by $\widehat{T}^\mathcal{F}_xS$ the linear subspace of $T_xS$ spanned by all vectors $v\in T_xS$ such that there exists a vector field $X\in\mathcal{F}$ with $v=X|_x$.  If $\mathcal{F}=\Vect(S)$, then we will denote this space by $\widehat{T}_xS$.  We will call $\widehat{T}^\mathcal{F}_xS$ the \emph{orbital tangent space} of $S$ at $x$ with respect to $\mathcal{F}$.  Let $\widehat{T}^\mathcal{F}S$ be the (disjoint) union $$\widehat{T}^\mathcal{F}S:=\bigcup_{x\in S}\widehat{T}^\mathcal{F}_xS.$$  We will call $\widehat{T}^\mathcal{F}S$ the \emph{orbital tangent bundle} with respect to $\mathcal{F}$.  It is a differential subspace of $TS$.  Denote by $\widehat{\tau}_\mathcal{F}$ the restriction of $\tau:TS\to S$ to $\widehat{T}^\mathcal{F}S$ and by $\delta_\mathcal{F}(x)$ the dimension $\dim(\widehat{T}^\mathcal{F}_xS)$.
\end{definition}

\begin{remark}
Since $\widehat{T}^\mathcal{F}S$ is a differential subspace of $TS$, a chart $\varphi:U\to\tilde{U}\subseteq\RR^n$ on $S$ induces a chart $(\varphi\circ\widehat{\tau}_{\mathcal{F}},\varphi_*|_{\varphi\circ\widehat{\tau}_{\mathcal{F}}})$ on $\widehat{T}^\mathcal{F}S$, which we shall denote simply as $\varphi_*$.  This is just a restriction of the corresponding chart on $TS$.  It makes the following diagram commute.

$$\xymatrix{
\widehat{T}^\mathcal{F}S|_U \ar[d]_{\widehat{\tau}_{\mathcal{F}}} \ar[r]^{\varphi_*} & T\RR^n \ar[d]^{\tau} \\
U \ar[r]_{\varphi} & \RR^n
}$$

This extends to (fibred) exterior powers of $\widehat{T}^\mathcal{F}S$ in the natural way; \emph{i.e.} to $$\bigwedge_S^k\widehat{T}^\mathcal{F}S:=\bigcup_{x\in S}\bigwedge^k\widehat{T}^\mathcal{F}_xS.$$
\end{remark}

\begin{lemma}\labell{l:lwrsemicont}
The map $\delta_\mathcal{F}:S\to\ZZ$ is lower semicontinuous.
\end{lemma}

\begin{proof}
Define $S_i:=\{x\in S~|~\delta_{\mathcal{F}}(x)\geq i\}$.  The goal is to show that $S_i$ is open for each $i$.  Let $y\in S_i$.  Then there exist $Y_1,...,Y_k\in\mathcal{F}$, where $k\geq i$, such that $\{Y_1|_y,...,Y_k|_y\}$ is a basis for $\widehat{T}^\mathcal{F}_yS$.  Linear independence is an open condition, and so there exists an open neighbourhood $U$ of $y$ such that $\{Y_1|_z,...,Y_k|_z\}$ is linear independent for all $z\in U$.  Hence, $\widehat{T}^\mathcal{F}_zS$ contains the span of $\{Y_1|_z,...,Y_k|_z\}$ as a linear subspace for each $z\in U$.  Thus, $\delta_\mathcal{F}(z)\geq k\geq i$.  Thus, $U\subseteq S_i$.
\end{proof}

\begin{proposition}[Local Triviality of $\widehat{T}^\mathcal{F}S$]\labell{p:checkTSloctriv}
There exists an open dense subset $U\subseteq S$ such that $\widehat{\tau}_\mathcal{F}|_{U}:\widehat{T}^\mathcal{F}S|_U\to U$ is locally trivial.
\end{proposition}

\begin{proof}
We will show that for any point $x\in S$ and any open set $U$ containing $x$, there is a point $z\in U$ and an open neighbourhood $V\subseteq U$ of $z$ so that $\widehat{\tau}_\mathcal{F}^{-1}(V)\cong V\times F$ for some vector space $F$.\\

Fix $x\in S$.  Define $S_i$ as in the proof of \hypref{l:lwrsemicont}{Lemma}.  Define $$m:=\underset{V\backepsilon x}{\inf}\{\sup\{k~|~S_k\cap V\neq\emptyset\}\}$$ where $V$ runs through all open neighbourhoods of $x$.  There exists an open neighbourhood $W$ of $x$ such that $\sup_{z\in W}\{\delta_\mathcal{F}(z)\}=m$.  Now fix $z\in W$ such that $\delta_\mathcal{F}(z)=m$.  Then, there are vector fields $Y_1,...,Y_m\in\mathcal{F}$ such that $\{Y_1|_z,...,Y_m|_z\}$ spans $\widehat{T}^\mathcal{F}_zS$.  Since linear independence is an open condition and $m$ is maximal, there is an open neighbourhood $V\subseteq W$ of $z$ such that $\{Y_1|_y,...,Y_m|_y\}$ spans $\widehat{T}^\mathcal{F}_yS$ for all $y\in V$.  Hence, $\widehat{T}^\mathcal{F}S$ is locally trivial over $V$.\\

Now, let $U$ be any open subset containing $x$.  We claim that there exists some $z\in W\cap U$ such that $\delta_\mathcal{F}(z)=m$.  Assume otherwise.  If $\sup_{z\in W\cap U}(\delta_\mathcal{F}(z))>m$, then this contradicts the definition of $W$.  If $\sup_{z\in W\cap U}\{\delta_\mathcal{F}(z)\}<m$, then this contradicts the definition of $m$.  Now, choose an open neighbourhood $V\subseteq W\cap U$ of $z$ as above, and the result follows.
\end{proof}

\begin{corollary}\labell{c:checkTopendense}
Let $\mathcal{F}$ be a locally complete family of vector fields, and let $U\subseteq S$ be an open dense subset on which $\widehat{T}^\mathcal{F}S$ is locally trivial.  Then, $\widehat{\tau}_\mathcal{F}^{-1}(U)$ is open and dense in $\widehat{T}^\mathcal{F}S$.
\end{corollary}

\begin{proof}
By continuity, $\widehat{\tau}_\mathcal{F}(U)$ is open.  Let $x\in S\smallsetminus U$, and let $Y_1,...,Y_k\in\mathcal{F}$ such that $\{Y_1|_x,...,Y_k|_x\}$ forms a basis of $\widehat{T}^\mathcal{F}_xS$.  Since linear independence is an open condition, there is an open neighbourhood $V$ of $x$ on which $\{Y_1|_y,...,Y_k|_y\}$ is linear independent for all $y\in V$, and their span is a subset of $\widehat{T}^\mathcal{F}_yS$.  Hence, $\widehat{T}^\mathcal{F}_xS\subseteq\overline{\widehat{\tau}_\mathcal{F}^{-1}(U)}$.
\end{proof}

\begin{remark}
The above corollary extends to exterior powers of the fibres of  $\widehat{T}^{\mathcal{F}}S$; that is, there exists an open dense subset $U\subseteq S$ on which $\bigwedge_S^k\widehat{T}^\mathcal{F}S\Big|_U\to U$ is locally trivial.
\end{remark}

\begin{proposition}[Zariski Versus Orbital Tangent Bundles]\labell{p:checkZar}
Let $S$ be a locally compact subcartesian space. Then there exists an open dense subset $U\subseteq S$ such that for each $x\in U$, $$\widehat{T}_xS=T_xS.$$
\end{proposition}

\begin{proof}
By \hypref{t:TSloctriv}{Theorem} and \hypref{p:checkTSloctriv}{Proposition}, there exists an open dense subset $U\subseteq S$ on which $TS$ and $\widehat{T}S$ are locally trivial.  Let $x\in U$, and let $\varphi:V\to\tilde{V}\subseteq\RR^n$ be a chart about $x$ where $V\subseteq U$ and $n=\dim(T_xS)$ (see \cite{lsw10}).  Then the derivations $\partial_1,...,\partial_n$ on $V$ arising from coordinates on $\RR^n$ give a local trivialisation of $TV$ (again, see \cite{lsw10}).  Let $W_1$ and $W_2$ be open neighbourhoods of $x$ satisfying $\overline{W_1}\subset W_2\subset\overline{W_2}\subset V$.  Let $b:S\to\RR$ be a smooth bump function that is equal to 1 on $W_1$ and 0 outside of $W_2$.  Then $b\partial_1,...,b\partial_n$ extend to derivations on all of $S$, and we claim that they are vector fields.\\

Now, for $i=1,...,n$, shrinking $V$ if necessary, there exist $\tilde{X}_1,...,\tilde{X}_n\in\Der\CIN(\RR^n)$ satisfying $\varphi_*(b\partial_i)=\tilde{X}_i|_{\tilde{V}}$.  Each $\tilde{X}_i$ gives rise to a local flow $\exp(\cdot \tilde{X}_i)(\cdot)$, such that for each $y\in \tilde{V}$, $\exp(\cdot \tilde{X}_i)(\varphi(y))$ has an open domain.  By \hypref{p:intcurve}{Proposition}, $\exp(\cdot \tilde{X}_i)(\varphi(y))=\varphi(\exp(t b\partial_i)(y))$ for all $t\in I^{b\partial_i}_y$ for which the integral curve lies in $V$.  But since $b$ is supported in $V$, the entire curve $\exp(\cdot b\partial_i)(y)$ is in $V$.  Hence, $\exp(t \tilde{X}_i)(\varphi(y))\in\tilde{V}$ for all $t\in I^{\tilde{X}_i}_{\varphi(y)}$.  Since $\tilde{X}_i$ is a vector field on $\RR^n$, $I^{\tilde{X}_i}_{\varphi(y)}$ is open, and consequently so is $I^{b\partial_i}_y$.  Thus, by \hypref{p:opendomain}{Proposition} $b\partial_i$ is a vector field on $V$, and since it has been extended as 0 to the rest of $S$, it is a vector field on $S$.  Finally, since $(b\partial_i)|_{W_1}=\partial_i|_{W_1}$ for each $i$, we see that $\widehat{T}_yS=T_yS$ for all $y\in W_1$, since $T_yS$ is the span over $\RR$ of $\{\partial_1|_y,...,\partial_n|_y\}.$
\end{proof}

\section{Orbits of Families of Vector Fields}

I review the theory of orbits of families of vector fields, including the Orbit Theorem for subcartesian spaces, proven by \'Sniatycki in \cite{sniatycki03}.  I show that the natural topology on the orbits discussed below comes from a diffeology induced by the local flows.

\begin{definition}[Orbits]
Let $S$ be a subcartesian space, and let $\mathcal{F}$ be a family of vector fields.  The \emph{orbit} of $\mathcal{F}$ through a point $x$, denoted $O^{\mathcal{F}}_x$ or just $O_x$ if $\mathcal{F}=\Vect(S)$, is the set of all points $y\in S$ such that there exist vector fields $X_1,...,X_k\in\mathcal{F}$ and real numbers $t_1,...,t_k\in\RR$ satisfying $$y=\exp(t_1X_1)\circ...\circ\exp(t_kX_k)(x).$$
Denote by $\mathcal{O}_{\mathcal{F}}$, or just $\mathcal{O}$ if $\mathcal{F}=\Vect(S)$, the set of all orbits $\{O^\mathcal{F}_x~|~x\in S\}$.  Note that $\mathcal{O}_{\mathcal{F}}$ induces a partition of $S$ into connected differential subspaces.
\end{definition}

Given a family of vector fields $\mathcal{F}$ on $S$, there exists a natural topology on the orbits that in general is finer than the subspace topology.  We define this topology here using similar notation as found in \cite{sniatycki03} and \cite{sussmann73}.  Let $X_1,...,X_k\in\mathcal{F}$.  Let $\xi:=(X_1,...,X_k)$ and $T=(t_1,...,t_k)$, and define $\xi_T(x):=\exp(t_kX_k)\circ...\circ\exp(t_1X_1)(x).$  $\xi_T(x)$ is well-defined for all $(T,x)$ in an open neighbourhood $U(\xi)$ of $(0,x)\in\RR^k\times S$.  Define $U_x(\xi)$ to be the set of all $T\in\RR^k$ such that $\xi_T(x)$ is well-defined; that is, $U_x(\xi)=U(\xi)\cap(\RR^k\times\{x\})$.  Let $i:O^\mathcal{F}_x\hookrightarrow S$ be the inclusion map.  Fix $y\in i(O^\mathcal{F}_x)$ and let $\varphi:V\to\tilde{V}\subseteq\RR^n$ be a chart of $S$ about $y$.  We give $W:=i^{-1}(V\cap i(O^\mathcal{F}_x))$ the strongest topology such that for each $\xi$ and $y\in i(W)$ the map $$\rho_{\xi,y}:U_y(\xi)\to\RR^n:T\mapsto \varphi\circ\xi_T(y)$$ is continuous.  This extends to a topology $\mathcal{T}$ on all of $O^\mathcal{F}_x$, which matches on overlaps (see \cite{sniatycki03}).

\begin{remark}\labell{r:vectdiffeol}
The \emph{$D$-topology} on a diffeological space $(X,\mathcal{D})$ is the strongest topology on $X$ such that all plots are continuous.  Let $\mathcal{D}_\mathcal{F}$ be the diffeology on $S$ generated by the maps $T\mapsto\xi_T(x)$ for all $X_1,...,X_k\in\mathcal{F}$, $\xi=(X_1,...,X_k)$, and $x\in S$.  Then the $D$-topology generated by $\mathcal{D}_\mathcal{F}$ on $S$ induces the topology described above on each orbit.
\end{remark}

\begin{lemma}\labell{l:orbitaltop}
With respect to the $D$-topology on $S$ induced by $\mathcal{D}_\mathcal{F}$, the orbits are connected and pairwise disjoint.
\end{lemma}

\begin{proof}
Fix $x\in S$, and choose $y\in O^\mathcal{F}_x$.  Then there exist $X_1,...,X_k\in\mathcal{F}$ and $t_1,...,t_k\in\RR$ such that $$y=\exp(t_kX_k)\circ...\circ\exp(t_1X_1)(x).$$  Let $T=(t_1,...,t_k)$ and $\xi=(X_1,...,X_k)$.  Then $$y=\xi_T(x).$$  Since $T\mapsto\xi_T(x)$ is continuous with respect to the $D$-topology, as it is a plot, and $U_x(\xi)$ is connected, its image is connected.  Hence $x$ and $y$ are in the same connected component of $S$ with respect to the $D$-topology.\\

We now show that each orbit is open and closed in the $D$-topology.  Since the preimage of any orbit is open in the domain of any plot, each orbit is open in the strongest topology such that each map $\rho_{\xi,x}$ is continuous.  Moreover, since the complement of any orbit is the union of orbits, and hence open, each orbit is closed.
\end{proof}

\begin{example}[Irrational Flow on Torus]
Let $S$ be the torus $\RR^2/\ZZ^2$ and let $\pi:\RR^2\to S$ be the quotient map. Consider the one-element family $\{X\}$ where $X=\pi_*(\partial_1+\sqrt{2}\partial_2)$.  Then for any $x\in S$, $\exp(tX)(x)$ has domain $\RR$, and the orbit is dense in $S$.  $\mathcal{T}$ in this case is such that $O^{\{X\}}_x$ is diffeomorphic to $\RR$.  This is strictly stronger than the subspace topology on the orbit.
\end{example}

\begin{theorem}[Orbit Theorem]\labell{t:singfol}
Let $S$ be a subcartesian space.  Then for any locally complete family of vector fields $\mathcal{F}$, $\mathcal{O}_{\mathcal{F}}$ induces a partition of $S$ into orbits $O^\mathcal{F}_x$, each of which when equipped with the topology $\mathcal{T}$ described above has a smooth manifold structure.  The inclusion $i:O^\mathcal{F}_x\hookrightarrow S$ is smooth, and $i_*:TO^\mathcal{F}_x\to TS$ is a fibrewise linear isomorphism onto $\widehat{T}^\mathcal{F}S|_{O^\mathcal{F}_x}$.
\end{theorem}
\begin{proof}
See \S5 Theorem 3 of \cite{sniatycki03}.
\end{proof}

\begin{remark}
This theorem generalises the corresponding ``orbit theorem'' in control theory (see, for example, \cite{jurdjevic97}).
\end{remark}

\begin{example}
In \hypref{x:lcexample}{Example} the orbital tangent space has dimension $\dim(\widehat{T}^\mathcal{F}_{(0,y)}\RR^2)=1$ for all $y$, whereas $\widehat{T}^\mathcal{F}_{(x,y)}\RR^2=T_{(x,y)}\RR^2$ for $x\neq0$. But there is only one orbit: all of $\RR^2$.  So the family of vector fields given by the $\RR$-span of $\{\partial_x,x\partial_y\}$ does not satisfy the conclusion of \hypref{t:singfol}{Theorem}.  (Recall that this family is not locally complete.)
\end{example}

\begin{theorem}[Ordering on Orbit Partitions]\labell{t:singfolorder}
Orbits of any family of vector fields $\mathcal{F}$ are contained within orbits of $\Vect(S)$.
\end{theorem}
\begin{proof}
See \S5 Theorem 4 of \cite{sniatycki03}.
\end{proof}

\begin{theorem}[Stratification Induced by $\Vect(S)$]\labell{t:stratorb1}
Let $S$ be a smooth stratified space.  Then the orbits on $S$ induced by $\Vect(S)$ form a smooth decomposition of $S$.
\end{theorem}
\begin{proof}
See \S6 Theorem 8 of \cite{sniatycki03}.
\end{proof}

\begin{remark}
Note that it is not known whether the induced decomposition satisfies the ``local triviality'' condition of a stratified space.  However, it is known that this decomposition satisfies the Whitney A condition (see \cite{lusala-sniatycki11}).  We recall what this conditions is.  Let $N_1$ and $N_2$ be two submanifolds of $\RR^k$ such that $N_1$ is contained in the closure of $N_2$ in $\RR^k$, and let $(x_i)$ be a sequence of points in $N_2$ with limit $x\in N_1$.  Then the sequence of tangent spaces $T_{x_i}N_2$ converges to a linear subspace $L$ of $T_x\RR^k$.  We say that the pair $(N_1,N_2)$ satisfies the Whitney A condition if $L$ contains $T_xN_1$.
\end{remark}

\begin{theorem}[Orbits of Stratified Vector Fields]\labell{t:stratorb2}
Let $S$ be a smooth stratified space.  Then the orbits of $\Vect_{strat}(S)$ are exactly the strata of $S$.
\end{theorem}
\begin{proof}
See \S6 Theorem 12 of \cite{sniatycki03}.
\end{proof}

\begin{theorem}[$\Vect(M/G)$ and the Orbit-Type Stratification]\labell{t:stratvect}
Given a compact Lie group $G$ acting on a connected manifold $M$, the strata of the orbit-type stratification on $M/G$ are precisely the orbits in $\mathcal{O}$ induced by $\Vect(M/G)$.
\end{theorem}

\begin{proof}
The proof can be found in \cite{sniatycki03} and \cite{lusala-sniatycki11}.  The idea is the following.  By \hypref{t:bierstone}{Theorem} the orbit-type stratification on $M/G$ is minimal.  The family of stratified vector fields of this stratification is locally complete by \hypref{p:stratloccompl}{Proposition} and its orbits are the strata by \hypref{t:stratorb2}{Theorem}. By \hypref{t:singfolorder}{Theorem}, these strata lie in orbits of $\Vect(M/G)$.  But, the set of orbits $\mathcal{O}$ induced by $\Vect(M/G)$ themselves form a stratification of $M/G$ by \hypref{t:stratorb1}{Theorem}.  So by minimality, we must have that these two stratifications are equal.
\end{proof}

\begin{proposition}\labell{p:jorbital}
Given a Hamiltonian action of a compact Lie group $G$ on a connected symplectic manifold $(M,\omega)$ with momentum map $\Phi$, let $Z$ be the zero set of $\Phi$.  The orbits of $\ham(Z/G)$ are the orbit-type strata of $Z/G$.
\end{proposition}

\begin{proof}
Lerman and Sjamaar showed in \cite{lerman-sjamaar91} that the maximal integral curves of any Hamiltonian vector field on $Z/G$ is confined to a symplectic stratum.  Moreover, we can construct these vector fields so that their orbits are exactly the connected components of the orbit-type strata of $Z/G$.
\end{proof}

\begin{theorem}\labell{t:regvalue}
If $0\in\g^*$ is a regular value of the momentum map $\Phi$, then the orbits induced by $\ham(Z/G)$ are exactly the orbits induced by $\Vect(Z/G)$, which gives a minimal stratification.
\end{theorem}

\begin{proof}
Assume that $0\in\g^*$ is a regular value of $\Phi$. Then $Z$ is a $G$-manifold, and by \hypref{t:stratvect}{Theorem} the orbit-type stratification is minimal, and the strata are exactly the orbits induced by $\Vect(Z/G)$.  By \hypref{p:jorbital}{Proposition} the orbits of $\Vect(Z/G)$ and $\ham(Z/G)$ coincide.
\end{proof}

\begin{question}\labell{q:critvalue}
Does the above theorem hold in general?  That is, even if $0\in\g^*$ is a critical value of $\Phi$?
\end{question}

\section{Lie Algebras of Vector Fields}

Our goal for this section is to establish that for a locally compact subcartesian space $S$, $\Vect(S)$ is a Lie algebra under the commutator bracket.  For a subset $A\subseteq S$ we shall denote by $\mathfrak{n}(A)$ the set of functions $\{f\in\CIN(S)~|~f|_{A}=0\}.$  Recall that for a family $\mathcal{F}$ of vector fields on $S$ and $x\in S$, $\widehat{T}^\mathcal{F}_xS$ is the linear subspace of $T_xS$ spanned by all vectors $X|_x$ for $X\in\mathcal{F}$.

\begin{proposition}[Characterisation of Orbital Vectors]\labell{p:charTcheck}
Let $S$ be a subcartesian space and $\mathcal{F}$ a locally complete family of vector fields. Let $x\in S$ and $v\in T_xS$.  Then, $v\in\widehat{T}^\mathcal{F}_xS$ if and only if for every open neighbourhood $U\subseteq O^\mathcal{F}_x$ of $i^{-1}(x)$, where $i$ is the inclusion of $O^\mathcal{F}_x$ into $S$, we have $v(\mathfrak{n}(i(U)))=\{0\}$.
\end{proposition}

\begin{proof}
Let $v\in\widehat{T}^\mathcal{F}_xS$.  Then by \hypref{t:singfol}{Theorem} $v=i_*w$ for some $w\in TO^\mathcal{F}_x$.  For any open neighbourhood $U$ of $i^{-1}(x)$ and for any $f\in\mathfrak{n}(i(U))$, $$vf=w(i^*f)=0.$$
Conversely, let $v\in T_xS$ and let $\varphi:V\to\tilde{V}\subseteq\RR^n$ be a chart about $x$. Then, $\varphi(V\cap i(O^\mathcal{F}_x))$ is a differential subspace of $\RR^n$, and in fact since $\varphi\circ i|_{i^{-1}(V)}$ is smooth with $d(\varphi\circ i|_{i^{-1}(V)})$ one-to-one (by \hypref{t:singfol}{Theorem}), we have that $\varphi\circ i|_{i^{-1}(V)}$ is an immersion.  Hence by the rank theorem there exists an open neighbourhood $U\subseteq i^{-1}(V)$ of $i^{-1}(x)$ such that $\tilde{U}:=\varphi\circ i(U)$ is an embedded submanifold of $\RR^n$.\\

Now, $v$ has a unique extension to a vector $\tilde{v}=\varphi_*v\in T_x\RR^n$.  Suppose $vf=0$ for all $f\in \mathfrak{n}(i(U))$.  Then for each such $f$, by \hypref{p:charTS}{Proposition}, $\tilde{v}\tilde{f}=0$ for any local representative $\tilde{f}$ of $f$. But then, also by \hypref{p:charTS}{Proposition}, we have that $\tilde{v}$ is the unique local extension of a vector $\tilde{w}\in T_{\varphi(x)}\tilde{U}$ since $\tilde{f}|_{\tilde{U}}=0$. Since $\tilde{U}$ is an embedded submanifold, there exists a unique $w\in T_{i^{-1}(x)}U$ such that $(\varphi\circ i)_*w=\tilde{w}$.  Identify $\tilde{w}$ with $\tilde{v}$.  By \hypref{t:singfol}{Theorem} and uniqueness, $i_*w=v$.  Thus, $v\in\widehat{T}^\mathcal{F}_xS$.\\

Since any open neighbourhood $W$ of $i^{-1}(x)$ contains a smaller open neighbourhood $U\subseteq i^{-1}(V)\cap W$ in which $\varphi\circ i(U)$ is an embedded submanifold of $\RR^n$, and also $\mathfrak{n}(i(W))\subseteq\mathfrak{n}(i(U))$, we can apply the above argument, obtaining our result.
\end{proof}

\begin{proposition}[Characterisation of Vector Fields]\labell{p:charvect}
Let $S$ be a locally compact subcartesian space.  A derivation $X\in\Der\CIN(S)$ is a vector field if and only if for every $x\in S$ and every open neighbourhood $U$ of $i^{-1}(x)$, $$X(\mathfrak{n}(i(U)))\subseteq\mathfrak{n}(i(U)).$$
\end{proposition}

\begin{proof}
Let $X$ be a vector field.  Then for any $x\in S$ and any open neighbourhood $U$ of $i^{-1}(x)$, $X|_{i(U)}$ is a vector field on $i(U)$.  By \hypref{p:charTcheck}{Proposition} for any $f\in\mathfrak{n}(i(U))$, $$(Xf)|_{i(U)}=0.$$
Conversely, let $X$ be a derivation of $\CIN(S)$ satisfying the property that for any open neighbourhood $U$ of $i^{-1}(x)$, $X(\mathfrak{n}(i(U)))\subseteq\mathfrak{n}(i(U))$ for all orbits $O_x$ with inclusion $i:O_x\to S$.  By \hypref{p:opendomain}{Proposition}, it is enough to show that each maximal integral curve of $X$ has an open domain.\\

Assume otherwise: there exists a maximal integral curve $\exp(tX)(x)$ through a point $x\in S$ with a closed or half-closed domain $I^X_x$.  If $X|_x=0$, then $\exp(tX)(x)$ is a constant map, and its maximal integral curve has $\RR$ as its domain, which is open. So assume $X|_x\neq0$.  Let $a\in I^X_x$ be an endpoint of $I^X_x$ and let $y:=\exp(aX)(x)$.  Then for any open neighbourhood $U\subseteq O_y$ of $i^{-1}(y)$, $$(Xf)|_{i(U)}=0$$ for all $f\in \mathfrak{n}(i(U))$.  In particular, $X|_zf=0$ for all $f\in \mathfrak{n}(i(U))$ and all $z\in i(U)$.  By \hypref{p:charTcheck}{Proposition}, $X|_z\in\widehat{T}_zS$ for all $z\in i(U)$.  Note that since $X|_x\neq0$, we have that $X|_y\neq0$, and so there exists an open neighbourhood $V\subseteq i(U)$  of $y$ such that $X|_z\neq 0$ for all $z\in V$.\\

Since $X|_V$ is a smooth section of $TV\subseteq TS$, by \hypref{t:singfol}{Theorem} we have constructed a vector field $Y\in\Vect(V)$ such that $Y|_z=X|_z$.  But note that by \hypref{p:intcurve}{Proposition} these integral curves locally are restrictions of integral curves in $\RR^n$, and so we can apply the ODE theorem, and obtain that since $X|_V=Y$, we have $\exp(tX)(y)=\exp(tY)(y)$ for $t$ in some domain $I_y$.  But, shrinking $V$ if necessary so that it is an embedded submanifold of $S$ (which exists by the rank theorem), since $Y$ is a vector field on the manifold $V$, $I_y$ is open and contains $0$, whereas since $\exp(tX)(y)=\exp((t+a)(X))(x)$, by assumption $I_y$ has $0\in I_y$ as an endpoint.  This is a contradiction.  Thus, $I^X_x$ does not contain any endpoints, and hence is open.
\end{proof}

\begin{corollary}[$\Vect(S)$ is a Lie Algebra]\labell{c:charvect}
Let $S$ be a locally compact subcartesian space.  Then $\Vect(S)$ is a Lie subalgebra of $\Der\CIN(S)$ and is a $\CIN(S)$-module.
\end{corollary}

\begin{proof}
Let $x\in S$, $X,Y\in\Vect(S)$, $U\subseteq O_x$ any open neighbourhood of $i^{-1}(x)$ and $f\in\mathfrak{n}(i(U))$ and $g\in\CIN(S)$.  Applying \hypref{p:charvect}{Proposition}, we have $(X+Y)f|_{i(U)}=Xf|_{i(U)}+Yf|_{i(U)}=0$, $(gX)f|_{i(U)}=0$ and $[X,Y](f)|_{i(U)}=X(Yf)|_{i(U)}-Y(Xf)|_{i(U)}=0$.  Thus, $X+Y$, $gX$ and $[X,Y]$ are vector fields.
\end{proof}

\begin{remark}
By the above corollary, for any $x\in S$ and any $v\in\widehat{T}_xS$, there is a vector field $X$ such that $X|_x=v$.  In other words, we did not need to take the linear span in the definition of $\widehat{T}_xS$.
\end{remark}

We again return to the situation of a Hamiltonian $G$-action on $(M,\omega)$.  We have shown that $\Vect(Z)$ is a Lie algebra.  Denote by $\Vect(Z)^G$ the Lie subalgebra of $G$-invariant vector fields on $Z$.

\begin{proposition}[Invariant Local Extensions for $\Vect(Z)^G$]\labell{p:vectZext}
Let $X\in\Vect(Z)^G$ and let $x\in Z\subseteq M$.  Then there exist a $G$-invariant open neighbourhood $U\subseteq M$ of $x$ and $\tilde{X}\in\Vect(M)^G$ such that $$X|_{U\cap Z}=\tilde{X}|_{U\cap Z}.$$
\end{proposition}

\begin{proof}
There exist an open neighbourhood $V\subseteq M$ of $x$ and $\tilde{X}_0\in\Vect(M)$ such that $\tilde{X}_0|_{V\cap Z}=X|_{V\cap Z}$.  Let $g_0=e\in G$ and let $g_i$ be elements of $G$ for $i=1,...,k$ such that $G\cdot x\subseteq M$ is covered by open sets $g_i\cdot V$.  Let $\{\zeta_i\}$ be a partition of unity subordinate to this cover, and define $$\tilde{X}:=\sum_{i=0}^k\zeta_i g_{i*}\tilde{X}_0.$$  Then, letting $W:=\bigcup_{i=0}^kg_i\cdot V$, we have that for any $y\in W\cap Z$
\begin{align*}
\tilde{X}|_y=&\sum_{i=0}^k\zeta_i(y)g_{i*}(\tilde{X}_0|_{g_i^{-1}\cdot y})\\
=&\sum_{i=0}^k\zeta_i(y)g_{i*}(X|_{g_i^{-1}\cdot y})\\
=&\sum_{i=0}^k\zeta_i(y)X|_y\\
=&X|_y.
\end{align*}
Thus, $\tilde{X}\in\Vect(M)$ is a local extension of $X$ on $W\cap Z$.  Averaging $\tilde{X}$ and letting $U$ be a $G$-invariant open neighbourhood of $G\cdot x$ contained in $W$, we are done.
\end{proof}

\begin{proposition}[$\Vect(Z)^G$ is a Lie Algebra]\labell{p:vectZloccompl}
$\Vect(Z)^G$ is a locally complete Lie subalgebra of $\Vect(Z)$.
\end{proposition}

\begin{proof}
Since diffeomorphisms commute with the commutator bracket, we have that $\Vect(Z)^G$ is a Lie subalgebra of $\Vect(Z)$.  For any two invariant vector fields $X$ and $Y$, we have for all $g\in G$ and $x\in Z$ $$g\cdot\exp(tX)(\exp(sY)(x))=\exp(tX)(\exp(sY)(g\cdot x))$$ for $s,t$ such that the composition of the curves is defined. Thus $\exp(tX)_*Y$ is locally defined about $G$-orbits.  Since $\Vect(Z)$ is locally complete, for any $x\in Z$ there exist a vector field $\Xi$ on $Z$ and an open neighbourhood $U$ of $x$ such that $\exp(tX)_*Y$ is defined on $U$ and $(\exp(tX)_*Y)|_U=\Xi|_U$.  Since $\exp(tX)_*Y$ is invariant about $x$, we can choose $U$ to be a $G$-invariant open neighbourhood.  Let $V\subset U$ be a $G$-invariant open neighbourhood of $x$ such that $\overline{V}\subset U$.  Let $b:M\to\RR$ be a $G$-invariant smooth bump function with support in $U$ and $b|_V=1$.  Then, $b\Xi\in\Vect(Z)^G$ extends $(\exp(tX)_*Y)|_V$ to a invariant vector field on $Z$.
\end{proof}

\begin{definition}
Let $\rho_Z:\g\to\Der\CIN(Z)$ be the $\g$-action induced by the action of $G$ on $Z$.  Note that by \hypref{p:opendomain}{Proposition}, $\rho_Z(\g)\subseteq\Vect(Z)$.  In fact, for any $\xi\in\g$, $\xi_Z:=\rho_Z(\xi)$ is just the restriction of $\xi_M$ to $Z$.
\end{definition}

\begin{proposition}[$\rho_Z(\g)$ is a Lie Algebra]\labell{p:rhoZloccompl}
$\rho_Z(\g)$ is a locally complete Lie subalgebra of $\Vect(Z)$.
\end{proposition}

\begin{proof}
Let $\xi,\zeta\in\g$, and let $\xi_Z=\rho_Z(\xi)$ and $\zeta_Z=\rho_Z(\zeta)$.  Then, $\exp(t\xi_Z)_*\zeta_Z=(\Ad_{\exp(t\xi)}\zeta)_Z$.  Thus $\rho_Z(\g)$ is locally complete, and since $\rho_Z$ is a Lie algebra homomorphism, its image is a Lie algebra.
\end{proof}

\begin{corollary}
$\rho_Z([\g,\g])$ and $\rho_Z(\mathfrak{z}(\g))$ are both locally complete Lie subalgebras of $\Vect(Z)$.
\end{corollary}

\begin{proof}
This is immediate from the above lemma.
\end{proof}

\begin{definition}\labell{d:AZ}
Define $\mathcal{A}_Z$ to be the smallest Lie subalgebra of $\Vect(Z)$ that contains both $\rho_Z(\g)$ and $\Vect(Z)^G$.
\end{definition}

\begin{proposition}\labell{p:AZ}
$\mathcal{A}_Z$ is locally complete and is equal to the direct sum of Lie subalgebras $$\mathcal{A}_Z=\rho_Z([\g,\g])\oplus\Vect(Z)^G.$$
\end{proposition}

\begin{proof}
By \hypref{p:vectZext}{Proposition}, for any $X\in\Vect(Z)^G$ and for any $x\in Z$, there exist a $G$-invariant open neighbourhood $U\subseteq M$ of $x$ and $\tilde{X}\in\Vect(M)^G$ such that $$X|_{U\cap Z}=\tilde{X}|_{U\cap Z}.$$  Hence, $$[\xi_Z,X]|_{U\cap Z}=[\xi_M,\tilde{X}]|_{U\cap Z}=0$$ by \hypref{p:A}{Proposition}.  Thus, applying \hypref{p:intcurve}{Proposition} and \hypref{e:flowscommute}{Equation} in the proof of \hypref{p:A}{Proposition}, we have that
\begin{equation}\labell{e:flowscommute2}
\exp(t\xi_Z)\circ\exp(sX)=\exp(sX)\circ\exp(t\xi_Z).
\end{equation}

Now, let $\xi\in\g$ and assume for all $g\in G$ and $x\in Z$, we have $$g_*(\xi_Z|_x)=\xi_Z|_{g\cdot x}.$$ Then, $$\frac{d}{dt}\Big|_{t=0}(g\cdot\exp(t\xi_Z)(x))=\frac{d}{dt}\Big|_{t=0}\exp(t\xi_Z)(g\cdot x).$$  The uniqueness property of $\exp$ implies that $$g\cdot\exp(t\xi_Z)(x)=\exp(t\xi_Z)(g\cdot x).$$  Hence  $(g\exp(t\xi))\cdot x=(\exp(t\xi)g)\cdot x$.  Since this is true for all $g\in G$, $\exp(t\xi)$ must be in the centre of $G$, and hence $\xi\in\mathfrak{z}(\g)$.  Thus, $$\rho_Z(\g)\cap\Vect(Z)^G=\rho_Z(\mathfrak{z}(\g)).$$  Since $\rho_Z$ is a Lie algebra homomorphism, from \hypref{e:flowscommute2}{Equation}: $\rho_Z(\g)=\rho_Z([\g,\g])\oplus\rho_Z(\mathfrak{z}(\g))$, and we obtain the direct sum structure of $\mathcal{A}_Z$.\\

To show local completeness, by \hypref{p:vectZloccompl}{Proposition} and \hypref{p:rhoZloccompl}{Proposition} it suffices to show that for any $\xi\in\g$ and $X\in\Vect(Z)^G$, $\exp(t\xi_Z)_*X\in\mathcal{A}_Z$ and $\exp(tX)_*\xi_Z\in\mathcal{A}_Z$.  The former is immediate since $X$ is invariant.  The latter follows from \hypref{e:flowscommute2}{Equation}:
\begin{align*}
\exp(tX)_*(\xi_Z|_x)=&\frac{d}{ds}\Big|_{s=0}\exp(tX)(\exp(s\xi_Z)(x))\\
=&\frac{d}{ds}\Big|_{s=0}\exp(s\xi_Z)(\exp(tX)(x))\\
=&\xi_Z|_{\exp(tX)(x)}.
\end{align*}
\end{proof}

\begin{proposition}[$\ham(Z/G)$ is a Lie Algebra]
$\ham(Z/G)$ is a locally complete Lie subalgebra of $\Vect(Z/G)$.
\end{proposition}

\begin{proof}
For any $f,g,h\in\CIN(Z/G)$ and $a,b\in\RR$, $\{af+bg,h\}_{Z/G}=a\{f,h\}_{Z/G}+b\{g,h\}_{Z/G}$, and so $aX_f+bX_g=X_{af+bg}$.  Thus $\ham(Z/G)$ is a real vector space.  Next, the Jacobi identity for the Poisson bracket gives $$\{\{f,g\}_{Z/G},h\}_{Z/G}=-\{g,\{f,h\}_{Z/G}\}_{Z/G}+\{f,\{g,h\}_{Z/G}\}_{Z/G}.$$ This translates to $$X_{\{f,g\}_{Z/G}}h=X_fX_gh-X_gX_fh=[X_f,X_g]h.$$

To show local completeness, fix $f,g\in\CIN(Z/G)$ and let $X_f$ and $X_g$ be their corresponding Hamiltonian vector fields.  For sufficiently small $t$, we want to show that $\exp(tX_f)_*X_g$ is a Hamiltonian vector field.  Consider $X_{\exp(-tX_f)^*g}$.  For any $h\in\CIN(Z/G)$, we have
\begin{align*}
X_{\exp(-tX_f)^*g}h=&\pois{\exp(-tX_f)^*g}{h}_{Z/G}\\
=&\exp(-tX_f)^*\pois{g}{\exp(tX_f)^*h}\\
=&\exp(-tX_f)^*(X_g(\exp(tX_f)^*h))\\
=&(\exp(tX_f)_*X_g)(h).
\end{align*}
This completes the proof.
\end{proof}

\section{Orbital Maps}

In general, a smooth map between subcartesian spaces does not lift to a map between the corresponding orbital tangent bundles.  This is illustrated in the following example.

\begin{example}
Let $S=\{(x,y)\in\RR^2~|~xy=0\}$, and let $\gamma:\RR\to S$ be a curve passing through $(0,0)\in S$ at time $t=0$ such that $$u:=\frac{d}{dt}\Big|_{t=0}\gamma(t)\neq0.$$  Then $u\notin\widehat{T}_{(0,0)}S$ since $\widehat{T}_{(0,0)}S=\{0\}$, but $\frac{d}{dt}\Big|_{t=0}\in\widehat{T}_0\RR=T_0\RR$.
\end{example}

To remedy this lack of the functoriality of $\widehat{T}$, we introduce a special kind of smooth map.

\begin{definition}[Orbital Maps]
Let $R$ and $S$ be subcartesian spaces and let $F:R\to S$ be a smooth map between them.  Let $\mathcal{F}$ and $\mathcal{G}$ be families of vector fields on $R$ and $S$, respectively.  $F$ is \emph{orbital} with respect to $\mathcal{F}$ and $\mathcal{G}$ if for any $x\in R$, $F(O^\mathcal{F}_x)\subseteq O^\mathcal{G}_{F(x)}$.  That is, for any $X\in\mathcal{F}$, $x\in R$, and $t\in I_x^X$, there exist $Y_1,...,Y_k\in\mathcal{G}$ and $t_1,...,t_k\in\RR$ such that
\begin{equation*}
F(\exp(tX)(x))=\exp(t_kY_k)\circ...\circ\exp(t_1Y_1)(F(x)).
\end{equation*}
If $\mathcal{F}=\Vect(R)$ and $\mathcal{G}=\Vect(S)$, then we simply call $F$ \emph{orbital}.
\end{definition}

\begin{proposition}[Charts, Smooth Functions, Diffeomorphisms]\labell{p:orb}
Charts, real-valued smooth functions, and diffeomorphisms between subcartesian spaces are orbital.
\end{proposition}

\begin{proof}
Since $\RR^k$ only has one orbit for each $k\geq0$, charts and smooth functions are trivially orbital.  Since a diffeomorphism $F:R\to S$ induces an isomorphism of Lie algebras $F_*:\Der\CIN(R)\to\Der\CIN(S)$, and hence $F(\exp(tX)(x))=\exp(tF_*X)(F(x))$ for all $X\in\Vect(R)$ and $x\in R$, we are done.
\end{proof}

\begin{proposition}[Orbital Pushforwards]\labell{p:Torbital}
Let $R$ and $S$ be subcartesian spaces, and let $F$ be an orbital map between them with respect to locally complete families of vector fields $\mathcal{F}$ on $R$ and $\mathcal{G}$ on $S$.  Then the restriction of the pushforward $F_*$ to $\widehat{T}^\mathcal{F}R$ has image in $\widehat{T}^\mathcal{G}S$.
\end{proposition}

\begin{proof}
This is immediate from \hypref{t:singfol}{Theorem} and the definition of an orbital map.
\end{proof}

\begin{remark}
Subcartesian spaces equipped with locally complete families of vector fields, along with orbital maps with respect to these families, form a category.  We will call the objects of this category \emph{orbital subcartesian spaces}.
\end{remark}

\begin{proposition}
Let $R$ and $S$ be smooth stratified spaces.  Then a smooth map $F:R\to S$ is stratified if and only if it is orbital with respect to $\Vect_{strat}(R)$ and $\Vect_{strat}(S)$.
\end{proposition}

\begin{proof}
This is immediate from \hypref{t:stratorb2}{Theorem}.
\end{proof}

\begin{corollary}\labell{c:orbitalcat}
The category of smooth stratified spaces, along with smooth stratified maps, forms a full subcategory of orbital subcartesian spaces.
\end{corollary}

The following theorem is a result of Schwarz; see \cite{schwarz77} and \cite{schwarz80} (\cite{schwarz80} Chapter 1 Theorem 4.3 for full details).  Let $D$ be the Lie subgroup of $\Diff(M)^G$ consisting of $G$-equivariant diffeomorphisms of $M$ that act trivially on $\CIN(M)^G$ (that is, they send each $G$-orbit to itself), and let $\mathfrak{d}$ denote the Lie algebra of $D$.

\begin{notation}
For brevity, we will often use the notation $\mathcal{V}:=\Vect(M)^G$ in the future.
\end{notation}

\begin{theorem}[Schwarz]\labell{t:vectquot}
The following is a split short exact sequence.
\begin{equation}\labell{e:ses}
\xymatrix{
0 \ar[r] & \mathfrak{d} \ar[r] & \Vect(M)^G \ar[r]^{\pi_*} & \Vect(M/G) \ar[r] & 0
}
\end{equation}
\end{theorem}

\begin{remark}
Actually, Schwarz showed that $\pi_*$ mapped $\Vect(M)^G$ onto stratified vector fields of $M/G$ with its orbit-type stratification. But by \hypref{t:stratvect}{Theorem}, this family of vector fields is exactly $\Vect(M/G)$.
\end{remark}

\begin{remark}
Since diffeomorphisms in $D$ keep $G$-orbits invariant, we have $\widehat{T}^{\mathfrak{d}}M\subseteq\widehat{T}^{\rho(\g)}M$. In fact, if $G$ is abelian then we have $\rho(\g)\subset\mathfrak{d}$, and so we obtain $$\widehat{T}^{\mathfrak{d}}M=\widehat{T}^{\rho(\g)}M.$$  However, in the non-abelian case, $\widehat{T}^{\mathfrak{d}}M$ may be a strict subset of $\widehat{T}^{\rho(\g)}M$.  For example, consider $\SO(3)$ acting by rotations on $\RR^3$.  For any nonzero $x\in\RR^3$ and any nonzero $\xi\in\mathfrak{so}(3)$, $\xi_{\RR^3}|_x$ is tangent to the $\SO(3)$-orbit through $x$, but this vector is not in the image of any invariant vector field.  For if it was, then the stabiliser at $x$ would fix the vector, and this is not the case.
\end{remark}

\begin{corollary}\labell{c:vectquot}
The image of $\pi_*$ restricted to $\widehat{T}^\mathcal{A}M$ is $\widehat{T}(M/G)$.
\end{corollary}

\begin{proof}
$\pi_*$ will map any vector in $\widehat{T}^{\rho(\g)}M$ to 0, and so it is enough to consider vectors in $\widehat{T}^\mathcal{V}M$ (where we set $\mathcal{V}:=\Vect(M)^G$ for brevity).  Let $x\in M$ and $v\in\widehat{T}_x^\mathcal{V}M$.  Then, there exists a invariant vector field $X\in\mathcal{V}$ such that $X|_x=v$.  By \hypref{t:vectquot}{Theorem} there exists $Y\in\Vect(M/G)$ such that $Y|_{\pi(x)}=\pi_*(X|_x)$.\\

Now, let $w\in\widehat{T}_{\pi(x)}(M/G)$.  There exists a vector field $Y\in\Vect(M/G)$ such that $Y|_{\pi(x)}=w$.   Again by \hypref{t:vectquot}{Theorem} there is a vector field $X\in\mathcal{V}$ such that $\pi_*X=Y$, and so $\pi_*(X|_x)=w$.
\end{proof}

\begin{corollary}\labell{c:piorbital}
$\pi$ is orbital with respect to $\mathcal{A}$ and $\Vect(M/G)$.
\end{corollary}

\begin{proof}
Since $\pi_*$ will map any vector field in $\rho(\g)$ to the zero vector field on $M/G$, and local flows of $\Vect(M)^G$ and $\rho(\g)$ commute, it is enough to check that $\pi$ is orbital with respect to $\Vect(M)^G$ and $\Vect(M/G)$.  Let $X\in\Vect(M)^G$.  Then by \hypref{t:vectquot}{Theorem}, there is a vector field $Y\in\Vect(M/G)$ such that $\pi_*X=Y$.  Fix $x\in M$.  Then $$\frac{d}{dt}\pi(\exp(tX)(x))=\pi_*(X|_x)=Y|_{\pi(x)}=\frac{d}{dt}\exp(tY)(\pi(x)).$$
By the ODE theorem, we have that $$\pi(\exp(tX)(x))=\exp(tY)(\pi(x))$$
for all $t$ where it is defined.  Hence orbits in $\mathcal{O}_\mathcal{A}$ are mapped via $\pi$ to orbits of $M/G$ induced by $\Vect(M/G)$.
\end{proof}

\begin{corollary}\labell{c:flowlift}
A local flow of $M/G$ lifts to a $G$-equivariant local flow of $M$.
\end{corollary}

\begin{proof}
Fix a vector field $Y\in\Vect(M/G)$.  By \hypref{t:vectquot}{Theorem} there is a vector field $X\in\Vect(M)^G$ such that $\pi_*X=Y$. From the ODE theorem we have that $$\pi(\exp(tX)(x))=\exp(tY)(\pi(x))$$ for all $x\in M$ and $t\in I^X_x$.
\end{proof}

\begin{theorem}\labell{t:singfolA}
The orbits in $\mathcal{O}_\mathcal{A}$ are exactly the orbit-type strata on $M$.
\end{theorem}

\begin{proof}
Fix $x\in M$, and let $H\leq G$ be a closed subgroup of $G$ such that $x\in M_{(H)}$.  Choose $y\in O^\mathcal{A}_x$.  Then, there exist vector fields $X_1,...,X_k\in\mathcal{A}$ and $t_1,...,t_k\in\RR$ such that $$y=\exp(t_1X_1)\circ...\circ\exp(t_kX_k)(x).$$  But then, by \hypref{c:piorbital}{Corollary} and \hypref{t:vectquot}{Theorem}, there exist $Y_1,...,Y_k\in\Vect(M/G)$ such that $$\pi(y)=\exp(t_1Y_1)\circ...\circ\exp(t_kY_k)(\pi(x)).$$  Hence, $\pi(x)$ and $\pi(y)$ are in the same orbit $O_{\pi(x)}$.  But this is a stratum of the orbit-type stratification of $M/G$ by \hypref{t:stratvect}{Theorem}, and so $y\in M_{(H)}$.  Thus $O^\mathcal{A}_x\subseteq M_{(H)}$.\\

Now, let $z$ be a point in the same connected component of $M_{(H)}$ as $x$.  Then again by \hypref{t:stratvect}{Theorem}, $\pi(y)$ and $\pi(x)$ are in the same orbit $O_{\pi(x)}$, and hence there exist vector fields $Y_1,...,Y_k$ and $t_1,...,t_k\in\RR$ such that $\pi(y)=\exp(t_1X_1)\circ...\circ\exp(t_kX_k)(\pi(x))$.  By \hypref{c:flowlift}{Corollary}, there are vector fields $X_1,...,X_k\in\mathcal{A}$ such that $y=\exp(t_1X_1)\circ...\circ\exp(t_kX_k)(x)$.
\end{proof}

We again return to the case where $G$ is a compact Lie group now acting on a connected symplectic manifold $(M,\omega)$ in a Hamiltonian fashion, with $Z$ the zero set of the momentum map $\Phi$.

$$\xymatrix{
Z \ar[r]^i \ar[d]_{\pi_Z} & M \ar[d]^{\pi} \\
Z/G \ar[r]_j & M/G \\
}$$

Recall that $\mathcal{A}=\rho(\g)+\Vect(M)^G$ and $\mathcal{A}_Z=\rho_Z(\g)+\Vect(Z)^G$ (see \hypref{d:A}{Definition} and \hypref{d:AZ}{Definition}).

\begin{proposition}\labell{p:iorbit}
$i$ is orbital with respect to $\mathcal{A}_Z$ and $\mathcal{A}$.
\end{proposition}

\begin{proof}
Let $X\in\mathcal{A}_Z$ and fix $z\in Z\subseteq M$.  Then by \hypref{p:vectZext}{Proposition} there exist a $G$-invariant open neighbourhood $U\subseteq M$ of $z$ and $\tilde{X}\in\mathcal{A}$ such that $X|_{U\cap Z}=\tilde{U\cap Z}$.  Applying the ODE theorem, we are done.
\end{proof}

\begin{proposition}\labell{p:piZorbit}
$\pi_Z$ is orbital with respect to $\mathcal{A}_Z$ and $\Vect(Z/G)$.
\end{proposition}

\begin{proof}
By \hypref{p:AZ}{Proposition}, it is enough to show this separately for $\rho_Z(\g)$ and $\Vect(Z)^G$.  For the first subalgebra, $$\pi(\exp(t\xi_Z)(z))=\pi(z)=\exp(0)(\pi(z))$$ for all $z\in Z$ and $t$ for which the integral curve is defined.\\

Now fix $X\in\Vect(Z)^G$.  Using \hypref{p:vectZext}{Proposition} cover $Z/G$ with a locally finite open cover $\{V_\alpha\}_{\alpha\in A}$ such that for every $\alpha\in A$, there exist $\tilde{X}^\alpha\in\Vect(M)^G$ satisfying $i_*(X|_{\pi_Z^{-1}(V_\alpha)})=\tilde{X}^\alpha|_{\pi^{-1}(j(V_\alpha))}.$ Note that for any $\alpha\in A$, $x\in V_\alpha$, $z\in\pi_Z^{-1}(x)$ and $f\in\mathfrak{n}(j(Z/G))$,
\begin{align*}
(\pi_*\tilde{X}^\alpha)|_{j(x)}f=&\tilde{X}^\alpha|_{i(z)}\pi^*f\\
=&X|_zi^*\pi^*f\\
=&X|_z\pi_Z^*j^*f\\
=&0.
\end{align*}
Let $\{\zeta_\alpha\}_{\alpha\in A}$ be a partition of unity subordinate to $\{V_\alpha\}$, and for each $\alpha\in A$, let $\tilde{\zeta}_\alpha$ be an extension of $\zeta_{\alpha}$ to $M/G$.  Define $$\tilde{Y}:=\sum_{\alpha}(\tilde{\zeta}_{\alpha}(\pi_*\tilde{X}^\alpha))|_{j(Z/G)}.$$ From the above, we have that $\tilde{Y}(f)=0$ for all $f\in\mathfrak{n}(j(Z/G))$, and so in particular, $\tilde{Y}$ restricts to a global derivation $Y\in\Der\CIN(S)$.  Also, for any $z\in Z$,
\begin{align*}
j_*\pi_{Z*}(X|_z)=&\sum_{\alpha}\tilde{\zeta}_\alpha j_*\pi_{Z*}(X|_z)\\
=&\sum_{\alpha}\tilde{\zeta}_\alpha\pi_*(\tilde{X}^\alpha|_{i(z)})\\
=&\tilde{Y}|_{\pi(i(z))}\\
=&j_*Y|_{\pi_Z(z)}.
\end{align*}
Thus, $\pi_{Z*}(X|_z)=Y|_{\pi_Z(z)}.$ Finally, we need to show that $Y$ is a vector field, and we shall do so by appealing to \hypref{p:opendomain}{Proposition}.  Fix $z\in Z$, and define $\gamma(t):=\pi_Z(\exp(tX)(z))$.  Differentiating, we see that $\gamma$ is an integral curve of $Y$ through $\pi_Z(z)$.  But $\gamma$ has an open domain and $\pi_Z$ is surjective, and so $\gamma$ is maximal.  Thus $Y$ is a vector field.
\end{proof}

\begin{proposition}\labell{p:jorbital2}
$j$ is orbital with respect to $\ham(Z/G)$ and $\Vect(M/G)$.
\end{proposition}

\begin{proof}
By \hypref{p:jorbital}{Proposition} orbits of $\ham(Z/G)$ are exactly the orbit-type strata of $Z/G$, which in turn are contained in the orbit-type strata of $M/G$. By \hypref{t:stratvect}{Theorem}, connected components of the orbit-type strata of $M/G$ are the orbits induced by $\Vect(M/G)$.
\end{proof}

\begin{lemma}\labell{l:GhamtangZ}
Vector fields in $\ham(M)^G$ are tangent to level sets of $\Phi$.
\end{lemma}

\begin{proof}
Fix $X\in\ham(M)^G$.  There exists $f\in\CIN(M)^G$ such that $X=X_f$.  It is enough to show that for any $\xi\in\g$, we have $X(\Phi^\xi)=0$.  Fix $\xi\in\g$.  Then
\begin{align*}
X(\Phi^\xi)=&d\Phi^\xi(X)\\
=&\omega(X,\xi_M)\\
=&-df(\xi_M)=0.
\end{align*}
\end{proof}

\begin{lemma}\labell{l:hamsurj}
There is a surjective Lie algebra homomorphism $H:\ham(M)^G\to\ham(Z/G)$ sending $X\in\ham(M)^G$ to $(\pi_Z)_*(X|_Z)$.
\end{lemma}

\begin{proof}
Fix $X\in\ham(M)^G$, and let $f\in\CIN(M)^G$ such that $X=X_f$.  By \hypref{l:GhamtangZ}{Lemma}, we have that $X|_Z$ is tangent to $Z$.  By \hypref{p:opendomain}{Proposition}, since the integral curves of $X$ through points of $Z$ are contained in $Z$, these integral curves when restricted to $Z$ have open domains, and hence $X|_Z\in\Vect(Z)^G$.\\

We now need to show that $(\pi_Z)_*(X|_Z)$ is a smooth vector field on $Z/G$.  Define $$h:=j^*((\pi^*)^{-1}(f)).$$  We claim that $X_h\in\ham(Z/G)$ is exactly $(\pi_Z)_*(X|_Z)$.  By \hypref{p:jorbital}{Proposition} it is enough to show this on each stratum $(Z/G)_{(H)}$ of $Z/G$.  Since $X$ is $G$-invariant, it is in fact tangent to $Z_{(H)}$ by \hypref{t:singfolA}{Theorem} and \hypref{l:GhamtangZ}{Lemma} for each $H\leq G$.  Fix a nonempty $Z_{(H)}$.  Then  $Y:=(\pi_{(H)})_*(X|_{Z_{(H)}})$ is a smooth vector field on $(Z/G)_{(H)}$.  Let $g\in\CIN(M/G)$ such that $\pi^*g=f$.  We have
\begin{align*}
\pi_{(H)}^*(Y\hook\omega_{(H)})=&X_f|_{Z_{(H)}}\hook i_{(H)}^*\omega\\
=&i_{(H)}^*(X_f\hook\omega)\\
=&i_{(H)}^*(-df)\\
=&(\pi\circ i_{(H)})^*(-dg)\\
=&(j\circ\pi_{(H)})^*(-dg)\\
=&\pi_{(H)}^*(-dj^*g|_{(Z/G)_{(H)}})\\
=&\pi_{(H)}^*(-d(h|_{(Z/G)_{(H)}})).
\end{align*}

Now, since $Z_{(H)}$ is a $G$-manifold with quotient manifold $(Z/G)_{(H)}$, it is known that $\pi_{(H)}^*$ is an isomorphism of complexes between differential forms on $(Z/G)_{(H)}$ and basic differential forms on $Z_{(H)}$.  Hence, $$Y\hook\omega_{(H)}=-d(h|_{(Z/G)_{(H)}}).$$  Thus, $Y=X_h|_{(Z/G)_{(H)}}$.  Thus the map $H$ is well-defined.\\

To show that this map is surjective, it is enough to show that there is a surjective map sending $f\in\CIN(M)^G$ to $j^*((\pi^*)^{-1}(f))$.  But $\pi^*$ is an isomorphism between $\CIN(M/G)$ and $\CIN(M)^G$, and since $Z/G$ is closed in $M/G$, we have that $j^*$ is a surjection from $\CIN(M/G)$ onto $\CIN(Z/G)$ by \hypref{p:closedsubset}{Proposition}.\\

We now check that this is a Lie algebra homomorphism.  It is clearly $\RR$-linear.  Let $f,g\in\CIN(M)^G$.  Then $$(\pi^*)^{-1}(\pois{f}{g})=\pois{(\pi^*)^{-1}f}{(\pi^*)^{-1}g}_{M/G},$$ and $j^*$ is a Poisson morphism.  Thus, $$H(X_{\pois{f}{g}})=\pois{H(X_f)}{H(X_g)}_{Z/G}.$$
\end{proof}

\begin{proposition}\labell{p:vectZorbits}
The orbits of $\mathcal{A}_Z$ are contained in the orbit-type strata of $Z$.  Moreover, if $G$ is connected, the orbits are exactly the orbit-type strata.
\end{proposition}

\begin{proof}
By \hypref{p:iorbit}{Proposition}, $i$ is orbital with respect to $\mathcal{A}_Z$ and $\mathcal{A}$.  Thus, orbits of $\mathcal{A}_Z$ are mapped into orbits of $\mathcal{A}$, which by \hypref{t:singfolA}{Theorem} are exactly the orbit-type strata on $M$.  Thus, the orbits of $\mathcal{A}_Z$ are contained in the orbit-type strata on $M$ intersected with $Z$.  But these are precisely the orbit-type strata of $Z$.\\

For the opposite inclusion, assume that $G$ is connected.  Let $x,y$ be in the same orbit-type stratum in $Z$.  Then $\pi_Z(x)$ and $\pi_Z(y)$ are in the same orbit-type stratum in $Z/G$.  Thus by \hypref{p:jorbital}{Proposition}, there exist $f_1,...,f_k\in\CIN(Z/G)$ and $t_1,...,t_k\in\RR$ such that the Hamiltonian vector fields $X_{f_1},...,X_{f_k}$ satisfy $$\pi_Z(y)=\exp(t_1X_{f_1})\circ...\circ\exp(t_kX_{f_k})(\pi_Z(x)).$$
By \hypref{l:hamsurj}{Lemma}, there exist $Y_1,...,Y_k\in\ham(M)^G$ such that $(\pi_Z)_*(Y_i|_Z)=X_{f_i}$ for each $i=1,...,k$.  So, we have $$\pi_Z(y)=\pi_Z(\exp(t_1Y_1|_Z)\circ...\circ\exp(t_kY_k|_Z)(\pi_Z(x)).$$ In particular, $$z:=\exp(t_1Y_1|_Z)\circ...\circ\exp(t_kY_k)(x)$$ is contained in the same $G$-orbit as $y$.  Thus there is some $g\in G$ such that $g\cdot z=y$.  Since $G$ is compact and connected, there is some $\tau\in\RR$ and $\xi\in\g$ such that $y=g\cdot z=\exp(\tau\xi_Z)(z)$.  Thus, $x$ and $y$ are in the same orbit of $\mathcal{A}_Z$.
\end{proof}

\begin{proposition}
If $0\in\g^*$ is a regular value of $\Phi$, then $j$ is orbital (with respect to $\Vect(Z/G)$ and $\Vect(M/G)$).
\end{proposition}

\begin{proof}
By \hypref{t:regvalue}{Theorem} we know that in this case, orbits of $\ham(Z/G)$ and $\Vect(Z/G)$ coincide.  Thus, applying \hypref{p:jorbital2}{Proposition} we are done.
\end{proof}

\begin{remark}
Let $G$ be a compact Lie group acting on a connected manifold $M$, and let $Z$ be an invariant closed subset of $M$.  Then it is not true that the inclusion $j:Z/G\to M/G$ is orbital with respect to $\Vect(Z/G)$ and $\Vect(M/G)$.  Indeed, let $G=\SS^1$ and $M=\RR\times\RR^2$.  $G$ acts on $M$ diagonally, trivially on $\RR$ and by rotations on $\RR^2$.  The geometric quotient $M/G$ is diffeomorphic to the closed half-plane $\RR\times[0,\infty)$.  Set coordinates $(x,y,z)$ on $M$, where $x$ describes $\RR$ and $(y,z)$ describes $\RR^2$ in the product $\RR\times\RR^2$.  Let $Z$ be the cone in $M$ given by $\{x^2=y^2+z^2\}.$  This is $G$-invariant, and $Z/G$ is diffeomorphic to $\RR$.  $\partial_x$ is a vector field on $Z/G$ whose orbit is all of $Z/G$, yet $j(Z/G)$ is not contained in one orbit of $\Vect(M/G)$.
\end{remark}

Given the above remark, zero sets of momentum maps are still special invariant closed subsets.  We are left with the following question.

\begin{question}
Is $j$ orbital in the general case when $0\in\g^*$ is a critical value of $\Phi$?
\end{question} 
\chapter{Orientation-Preserving Diffeomorphisms of $\SS^2$}\labell{ch:sphere}

This chapter is a report on the paper "The orientation-preserving diffeomorphism group of $\SS^2$ deforms to $\SO(3)$ smoothly" \cite{li-watts10}.  In this paper, Jiayong Li and I use techniques introduced by Smale in his article "Diffeomorphisms of the 2-sphere" \cite{smale59} to show that the identity component of the diffeomorphism group of $\SS^2$ has a diffeologically smooth strong deformation retraction onto the subgroup of rotations $\SO(3)$, that is equivariant with respect to rotations.  More precisely, let $I=[0,1]$.  We construct a diffeologically smooth map $P:I\times\Diff_0(\SS^2)\to\Diff_0(\SS^2)$ with respect to the standard functional diffeology on $\Diff_0(\SS^2)$ such that for each $(t,f)\in I\times\Diff_0(\SS^2)$ and $A\in\SO(3)$,
\begin{enumerate}
\item $P_0(f)=f$,
\item $P_1(f)\in\SO(3)$,
\item $P_t(A)=A$,
\item $P_t(A\circ f)=A\circ P_t(f)$.
\end{enumerate}
Smale constructed a continuous strong deformation retraction with respect to the $C^k$-topology on the diffeomorphism group ($1<k\leq\infty$).  I will go through the main ideas of our paper, emphasising which parts of the joint work were done more exclusively by Jiayong, or myself, where possible.\\

We begin with the following theorem.
\begin{theorem}[3.1 of \cite{li-watts10}]\labell{t:smale1}
The space of diffeomorphisms of the square $[0,1]^2$ that are equal to the identity in a neighbourhood of the boundary has a diffeologically smooth contraction to a point.
\end{theorem}

Proving this theorem was my job.  The proof follows Smale's analogous theorem (a continuous contraction to a point)  closely.  Much of the work was translating his arguments into the diffeological language; however, some modifications are required in order to obtain diffeological smoothness.  Indeed, one argument in particular (in the proof of Theorem 3.3 of \cite{li-watts10}) did not necessarily achieve smoothness, only continuity, so I used a different argument.\\
Let $x_0$ be the south pole of $\SS^2$, and define $\Omega_1$ to be the space of all diffeomorphisms $\varphi$ of $\SS^2$ that fix $x_0$ and such that $d\varphi|_{x_0}$ is the identity map on $T_{x_0}\SS^2$.  The next step in the paper is the following theorem.
\begin{theorem}[1.8 of \cite{li-watts10}]\labell{t:smale2}
$\Omega_1$ has a diffeologically smooth contraction to a point.
\end{theorem}
The proof of the analogous theorem in Smale's paper is not entirely clear as written. If one looks closely, $\SS^2$ is embedded into $\RR^3$, and Smale uses a linear interpolation between points of a neighbourhood of $x_0$ to build a contraction, but this is not well-defined, as a linear interpolation may move points off the embedded sphere.  While it may be understood that such a neighbourhood is diffeomorphic to an open convex subset of the plane, where such an interpolation is okay, Jiayong and I both use a stereographic projection near $x_0$ instead, avoiding the issue altogether.\\

Also, in the same proof, Smale needs a positive continuous function $\epsilon:\Diff_0(\SS^2)\to\RR$ which, for a given diffeomorphism $\varphi$, returns a radius $\epsilon{\varphi}$ of a neighbourhood of $x_0$ (using a stereographic projection) on which $|d\varphi(v)-v|<1$ for all $v\in T_{x_0}\SS^2$.  This is easy to construct: choose the supremum over all possible such radii.  Then this yields a positive lower semicontinuous function on $\Diff_0(\SS^2)$.  Now one can use a partition of unity to construct a positive continuous such function.  However, this argument does not work when trying to achieve diffeological smoothness.  Jiayong used a combination of the mean-value theorem and a Sobolev inequality to obtain a smooth function $\epsilon:\Diff_0(\SS^2)\to\RR$, which completed the proof of the theorem (see Lemma 2.4 of \cite{li-watts10}).\\

What is important about \hypref{t:smale2}{Theorem} is that the contraction has two stages: given a diffeomorphism $\varphi\in\Diff_0(\SS^2)$, the first stage modifies $\varphi$ in a small neighbourhood about $x_0$ so that it becomes equal to the identity near $x_0$.  The second stage uses a stereographic projection and \hypref{t:smale1}{Theorem} to deform $\varphi$ away from $x_0$ into the identity there.  I proved the second stage, which consisted of translating what Smale did into the diffeological language, much like the proof of \hypref{t:smale1}{Theorem}.  Jiayong proved the first stage.\\

To take advantage of \hypref{t:smale2}{Theorem}, we have the following theorem.  Embed $\SS^2$ into $\RR^3$, which induces an orthonormal basis $\{e_1,e_2\}$ on $T_{x_0}\SS^2$.  Let $\Omega_2$ be the space of diffeomorphisms $\varphi$ of $\SS^2$ such that $d\varphi(e_1)$ and $d\varphi(e_2)$ form an orthonormal basis of $T_{\varphi(x_0)}\SS^2$.
\begin{lemma}[1.6 of \cite{li-watts10}]\labell{l:smale3}
There is a diffeologically smooth homotopy, equivariant under the action of $\SO(3)$, from $\Diff_0(\SS^2)$ onto $\Omega_2$. \end{lemma}
The proof of the analogous statement in Smale's paper is missing some details.  Jiayong and I both went through the proof carefully and filled in these details.  In particular, for a diffeomorphism $\varphi$, a linear interpolation deforms $d\varphi(e_1)$ and $d\varphi(e_2)$ into an orthonormal pair.  But one needs to control this deformation so that it occurs only locally near $\varphi(x_0)$.  We used a compactly supported vector field and a stereographic projection to achieve this.  Otherwise, translating Smale's proof into the diffeological language was required, and we both did this for this part of the paper.\\

At this point for $\varphi\in\Omega_2$, there is a unique rotation $A(\varphi)$ in $\SO(3)$ that sends $(\varphi(x_0),d\varphi(e_1),d\varphi(e_2))$ to $(x_0,e_1,e_2)$.  Composing $\varphi$ with such a rotation yields a diffeomorphism in $\Omega_1$.  But we know that $\Omega_1$ contracts to a point (the identity map on $\SS^2$), and so our result follows. 
\chapter{Appendix: Subcartesian Spaces in the Literature}\labell{ch:appendix}

In the literature, one typically finds two different definitions of a subcartesian space: one which makes use of an atlas, and one which makes use of a differential structure as we used in \hypref{s:differentialspaces}{Section}.  Maps between subcartesian spaces are also defined differently.  The purpose of this section is to show that the two different definitions lead to isomorphic categories.  To the author's knowledge, this isomorphism does not appear in the literature, and while it may not be surprising, it should still be written down.\\

We first begin by introducing the classical definition of a subcartesian space as an object of a category we shall refer to as $\mathcal{S}_0$.  We then proceed to define its morphisms, and then show its equivalence to the category of subcartesian spaces as defined in \hypref{s:differentialspaces}{Section}.

\begin{definition}[Objects of $\mathcal{S}_0$]
Let $S$ be a Hausdorff, paracompact, second-countable topological space locally homeomorphic to subspaces of $\RR^n$ ($n=0,1,...$).  That is, for any $x\in S$ there exist an open neighbourhood $U$ of $x$, a positive integer $n$, a subset (not necessarily open) $\tilde{U}\subseteq\RR^n$, and a homeomorphism $\varphi:U\to\tilde{U}$.  We refer to such homeomorphisms as \emph{charts}.  Let $\varphi:U\to\tilde{U}\subseteq\RR^m$ and $\psi:V\to\tilde{V}\subseteq\RR^n$ be two charts on $S$.  Then, $\varphi$ and $\psi$ are \emph{compatible} if for every $x\in U\cap V$, there is an open neighbourhood $W^\varphi\subseteq\RR^m$ of $\varphi(x)$ and an open neighbourhood $W^\psi\subseteq\RR^n$ of $\psi(x)$ satisfying:

\begin{enumerate}
\item there exists a smooth map $\zeta\in\CIN(\RR^m,\RR^n)$ such that $\zeta|_{\varphi(U\cap V)\cap W^\varphi}=\psi\circ\varphi^{-1}|_{\varphi(U\cap V)\cap W^\varphi}$,
\item there exists a smooth map $\xi\in\CIN(\RR^n,\RR^m)$ such that $\xi|_{\psi(U\cap V)\cap W^\psi}=\varphi\circ\psi^{-1}|_{\psi(U\cap V)\cap W^\psi}$.
\end{enumerate}
This is easily summarised by saying that $\psi\circ\varphi^{-1}$ and its inverse locally extend to smooth functions between Cartesian spaces.
A \emph{(maximal) atlas} $\mathcal{A}$ on $S$ is a maximal collection of charts on $S$ such that any two charts in the collection are compatible.  An object $(S,\mathcal{A})$ of $\mathcal{S}_0$ is a Hausdorff, paracompact, second-countable topological space $S$ which is locally homeomorphic to subspaces of Cartesian space, equipped with an atlas $\mathcal{A}$ of charts.
\end{definition}

When the atlas of an object in $\mathcal{S}_0$ is understood, we will often drop it from the notation.

\begin{definition}[Morphisms of $\mathcal{S}_0$]
Let $R$ and $S$ be two objects of $\mathcal{S}_0$.  A map $F:R\to S$ is a \emph{morphism} in $\mathcal{S}_0$ if it is continuous and it satisfies the following condition.  For every $x\in R$, there is a chart $\varphi:U\to\tilde{U}\subseteq\RR^m$ about $x$ and a chart $\psi:V\to\tilde{V}\subseteq\RR^n$ about $F(x)$ such that $\psi\circ F\circ\varphi^{-1}:\tilde{U}\to\tilde{V}$ extends to a smooth map $\tilde{F}:\RR^m\to\RR^n$. An \emph{isomorphism} in the category $\mathcal{S}_0$ is a homeomorphism such that it and its inverse are morphisms.
\end{definition}

\begin{theorem}[$\mathcal{S}_0$ is Isomorphic to the Category of Subcartesian Spaces]\labell{t:subcart}
The category $\mathcal{S}_0$ is isomorphic to the category of subcartesian spaces.
\end{theorem}

We will break up the proof of this theorem into parts, each part a separate lemma.  Denote the category of subcartesian spaces by $\mathbf{Subcart}$.

\begin{lemma}
There is a functor $\mathbf{\Phi}$ from $\mathcal{S}_0$ to $\mathbf{Subcart}$ such that the underlying topological space of an object $(S,\mathcal{A})$ of $\mathcal{S}_0$ is the same as that of $\mathbf{\Phi}((S,\mathcal{A}))$, and $\mathbf{\Phi}{F}=F$ for any morphism $F$ of $\mathcal{S}_0$.
\end{lemma}

\begin{proof}
Let $(S,\mathcal{A})$ be a smooth subcartesian space in $\mathcal{S}_0$.  Define $\CIN(S)$ to be the set of all functions $f:S\to\RR$ such that for every $x\in S$ there is a chart $\varphi:U\to\tilde{U}\subseteq\RR^m$ about $x$ contained in $\mathcal{A}$ and a function $\tilde{f}\in\CIN(\RR^m)$ satisfying $f|_U=\varphi^*\tilde{f}$.  We claim that $(S,\CIN(S))$ is an object in $\mathbf{Subcart}$.  To prove this, we need to show that $\CIN(S)$ is a differential structure on $S$, and that $S$ is locally diffeomorphic to differential subspaces of Cartesian spaces.  We begin with the differential structure conditions.\\

We first need to show that the topology generated by sets of the form $$\{f^{-1}((a,b))~|~(a,b)\subseteq\RR,~f\in\CIN(S)\}$$ is contained in the topology on $S$.  Fix $(a,b)\subseteq\RR$, $f\in\CIN(S)$, let $U=f^{-1}((a,b))$ and fix $x\in U$.  Let $\varphi:V\to\tilde{V}\subseteq\RR^m$ be a chart about $x$ in $\mathcal{A}$ and $\tilde{f}\in\CIN(\RR^m)$ such that $\varphi^*\tilde{f}=f|_V$.  Let $W=\tilde{f}^{-1}((a,b))$.  Then $\varphi^{-1}(W)$ is open in $S$.  But $\varphi^{-1}(W)=(f|_V)^{-1}((a,b))\subseteq U\cap V$.  Since $x\in U$ is arbitrary, we conclude that $U$ is open in $S$.\\

To show that the topology on $S$ is contained in the topology induced by $\CIN(S)$, it is sufficient to show that the domain of any chart is contained in this latter topology. Let $\varphi:V\to\tilde{V}\subseteq\RR^m$ be a chart in $\mathcal{A}$, and fix $x\in V$.  Define $\tilde{f}:\RR^m\to[0,1]$ to be a smooth bump function on $\RR^m$ that is equal to 1 on an open neighbourhood $W_1$ of $\varphi(x)$, and with support in an open neighbourhood $W_0$ of $\varphi(x)$ small enough such that $\overline{\varphi^{-1}(W_0)}\subset V$. Define $U=\varphi^{-1}\circ\tilde{f}^{-1}((\frac{1}{4},2))$. Then $U$ is open in $V$.  Now $\tilde{f}\circ\varphi$ extends to all of $S$ by setting it equal to zero on the complement of $V$.  This extension is in $\CIN(S)$.  Since $x$ is arbitrary, every point of $V$ is contained in a set of the form $f^{-1}((a,b))\subseteq V$ for some $f\in\CIN(S)$ and some interval $(a,b)$.  Hence, $V$ is open in the topology generated by $\CIN(S)$, and we see that the two topologies are equal.\\

Next, we show that $(S,\CIN(S))$ satisfies the smooth compatibility condition of a differential space.  Let $f_1,...,f_k\in\CIN(S)$, and let $F\in\CIN(\RR^k)$.  Fix $x\in S$.  For $i=1,...,k$ there exist compatible charts $\varphi_i:U_i\to\tilde{U}_i\subseteq\RR^{m_i}$ and functions $\tilde{f}_i\in\CIN(\RR^{m_i})$ such that $f_i|_{U_i}=\varphi_i^*\tilde{f}_i$.  Define $U=\bigcap_{i=1}^kU_i$.  Then
\begin{align*}
F(f_1,...,f_k)|_U=&F(f_1|_U,...,f_k|_U)\\
=&F(\tilde{f}_1\circ\varphi_1|_U,...,\tilde{f}_k\circ\varphi_k|_U)\\
=&F(\tilde{f}_1\circ\varphi_1|_U,\tilde{f}_2\circ\varphi_2\circ\varphi^{-1}_1\circ\varphi_1|_U,...,
\tilde{f}_k\circ\varphi_k\circ\varphi_1^{-1}\circ\varphi_1|_U).
\end{align*}
Now, since the charts $\varphi_i$ are pairwise compatible, for $i=2,...,k$ there exist open neighbourhoods $W_i\subseteq\RR^{m_1}$ of $\varphi_1(x)$ and smooth functions $\zeta_i:\RR^{m_1}\to\RR^{m_i}$ satisfying $\zeta_i|_{W_i\cap\tilde{U}_1}=\varphi_i\circ\varphi^{-1}_1|_{W_i\cap\tilde{U}_1}$. Let $W=\bigcap_{i=2}^k\varphi^{-1}_1(W_i)$.  Then, $$F(f_1,...,f_k)|_W=F(\tilde{f}_1,\tilde{f}_2\circ\zeta_2,...,\tilde{f}_k\circ\zeta_k)\circ\varphi_1|_W.$$ But, the right hand side is the pullback of a smooth function defined on $\RR^{m_1}$ by the chart $\varphi_1|_W$.  We conclude that $F(f_1,...,f_k)\in\CIN(S)$.\\

We next show that $(S,\CIN(S))$ satisfies the locality condition of a differential structure.  Let $f:S\to\RR$ be a function having the property that for any $x\in S$ there is an open neighbourhood $U\subseteq S$ of $x$ and a function $g\in\CIN(S)$ satisfying $g|_U=f|_U$.  Fix $x\in S$, and let $U$ and $g$ satisfy the above condition.  Since $g\in\CIN(S)$, by definition, for any $y\in U$, there is a smooth chart $\varphi:V\to\tilde{V}\subseteq\RR^m$ about $y$ and a function $\tilde{g}\in\CIN(\RR^m)$ such that $g|_V=\varphi^*\tilde{g}$.  Hence, $f|_{U\cap V}=\varphi|_{U\cap V}^*\tilde{g}$.  Since $\varphi|_{U\cap V}$ is a smooth chart, $f\in\CIN(S)$.  Thus, $\CIN(S)$ is a differential structure.\\

We now show that $(S,\CIN(S))$ is a subcartesian space; that is, the charts are diffeomorphisms onto differential subspaces of Cartesian spaces.  Let $\varphi:U\to\tilde{U}\subseteq\RR^m$ be a chart and let $f\in\CIN(\tilde{U})$.  Then, for every $x\in U$, there is a neighbourhood $\tilde{V}\subseteq\RR^m$ of $\varphi(x)$ and a function $\tilde{f}\in\CIN(\RR^m)$ such that $\tilde{f}|_{\tilde{U}\cap\tilde{V}}=f|_{\tilde{U}\cap\tilde{V}}.$ Let $V:=\varphi^{-1}(\tilde{U}\cap\tilde{V})$. Thus, $(\varphi^*f)|_V=(\varphi|_V)^*\tilde{f}$. So $\varphi^*f\in\CIN(U)$.  Conversely, let $g\in\CIN(U)$.  Then for each $x\in U$ there is a smooth chart $\psi:V\to\tilde{V}\subseteq\RR^n$ about $x$ (where $V\subseteq U$) and a function $\tilde{g}\in\CIN(\RR^n)$ such that $g|_V=\psi^*\tilde{g}$.  Now, $\varphi$ and $\psi$ are compatible, and so there is an open neighbourhood $W\subseteq\RR^m$ of $\varphi(x)$ and a smooth function $\zeta\in\CIN(\RR^m,\RR^n)$ satisfying $$\psi\circ\varphi^{-1}|_{W\cap\tilde{U}}=\zeta|_{W\cap\tilde{U}}.$$ Thus, $$g\circ\varphi^{-1}|_{W\cap\tilde{U}\cap\tilde{V}}=\tilde{g}\circ\zeta|_{W\cap\tilde{U}\cap\tilde{V}}.$$ Since $\tilde{g}\circ\zeta\in\CIN(\RR^m)$, we conclude that $(\varphi^{-1})^*g\in\CIN(\tilde{U})$. Thus, $\varphi$ is a diffeomorphism, and $(S,\CIN(S))$ is an object in $\mathbf{Subcart}$.  Define the map $\mathbf{\Phi}$ between objects of $\mathcal{S}_0$ and objects of $\mathbf{Subcart}$ by $\mathbf{\Phi}((S,\mathcal{A}))=(S,\CIN(S))$.\\

To show that $\mathbf{\Phi}$ is a functor we need to say what it does to morphisms in $\mathcal{S}_0$.  Let $(R,\mathcal{A})$ and $(S,\mathcal{B})$ be objects in $\mathcal{S}_0$, and let $F$ be a morphism between them.  Let $f\in\CIN(S)$.  Then, for every $x\in R$, there exist a chart $\psi:V\to\tilde{V}\subseteq\RR^n$ about $F(x)$, a chart $\varphi:U\to\tilde{U}\subseteq\RR^m$ about $x$ with $F(U)\subseteq V$ and a smooth map $\tilde{F}:\RR^m\to\RR^n$ and a smooth function $\tilde{f}\in\CIN(\RR^n)$ such that $f|_V=\psi^*\tilde{f}$ and
\begin{align*}
(F^*f)|_U=&f\circ\psi^{-1}\circ\tilde{F}\circ\varphi\\
=&\tilde{f}\circ\tilde{F}\circ\varphi.
\end{align*}
But, $\tilde{f}\circ\tilde{F}\in\CIN(\RR^m)$ and so $\tilde{f}\circ\tilde{F}\circ\varphi=(F^*f)|_U\in\varphi^*\CIN(\RR^m)$.  Hence, $F^*f\in\CIN(R)$.  Thus, $F$ is also functionally smooth. Defining $\mathbf{\Phi}(F)=F$, we thus have a well-defined functor $\mathbf{\Phi}$.\\
\end{proof}

\begin{lemma}
There is a functor $\mathbf{\Psi}$ from $\mathbf{Subcart}$ to $\mathcal{S}_0$ such that the topological space of an object $(S,\CIN(S))$ of $\mathbf{Subcart}$ is the same as that of $\mathbf{\Psi}((S,\CIN(S)))$, and $\mathbf{\Psi}{F}=F$ for any smooth map $F$ of $\mathbf{Subcart}$.
\end{lemma}

\begin{proof}
Let $(S,\CIN(S))$ be a subcartesian space in $\mathbf{Subcart}$.  Define $\mathcal{A}$ to be the collection of all diffeomorphisms of open differential subspaces of $S$ to differential subspaces of Cartesian spaces.  We claim that this is a maximal atlas.  It is clear that these are homeomorphisms onto subsets of Cartesian spaces, so we only need to check compatibility.  Let $\varphi:U\to\tilde{U}\subseteq\RR^m$ be a diffeomorphism onto $\tilde{U}$.  Let $\psi:U\to\tilde{V}\subseteq\RR^n$ be another such diffeomorphism onto a differential subspace $\tilde{V}$ of $\RR^n$.  Denote by $x^1,...,x^m$ the standard coordinates on $\RR^m$ and set $q^i:=\varphi^*x^i$.  Then $\varphi=(q^1,...,q^m)$.  Set $p^i:=q^i\circ\psi^{-1}$.  Since $\psi$ is a diffeomorphism, we have that $p^i\in\CIN(\tilde{V})$.  In other words, for each $i=1,...,m$ and each $y\in U$, there exists an open neighbourhood $W_i\subseteq\RR^n$ of $\psi(y)$ and a function $\tilde{p}^i\in\CIN(\RR^n)$ such that $\tilde{p}^i|_{W_i\cap\tilde{V}}=p^i|_{W_i\cap\tilde{V}}$.  Define $W:=\bigcap_{i=1}^mW_i$.  Then $\xi:=(\tilde{p}^1,...,\tilde{p}^m)$ is a smooth map from $\RR^n$ to $\RR^m$ and $$\xi|_{W\cap\tilde{V}}=(p^1|_{W\cap\tilde{V}},...,p^m|_{W\cap\tilde{V}})=\varphi\circ\psi^{-1}|_{W\cap\tilde{V}}.$$ Hence, $\varphi\circ\psi^{-1}$ locally extends to smooth maps between Cartesian spaces.  A similar argument shows that $\psi\circ\varphi^{-1}$ also locally extends in this manner.  Hence, $\varphi$ and $\psi$ are compatible charts on $S$. So $(S,\mathcal{A})$ is an object in $\mathcal{S}_0$.  Define $\mathbf{\Psi}((S,\CIN(S)))$ to be $(S,\mathcal{A})$.\\

Now, let $(R,\CIN(R))$ and $(S,\CIN(S))$ be two subcartesian spaces.  Let $F:R\to S$ be a functionally smooth map: $F^*\CIN(S)\subseteq\CIN(R)$.  Fix $x\in R$, and let $\psi:V\to\tilde{V}\subseteq\RR^n$ be a diffeomorphism about $F(x)$.  Let $q^1,...,q^n$ be coordinate functions on $V$; that is, $q^i:=\psi^*x^i$ for $i=1,...,n$.  Then, $q^i\in\CIN(V)$, and so $F|_{F^{-1}(V)}^*q^i\in\CIN(F^{-1}(V))$.  Let $\varphi:U\to\tilde{U}\subseteq\RR^m$ be a diffeomorphism about $x$ such that $F(U)\subseteq V$, and let $p^1,...,p^n\in\CIN(\RR^m)$ such that $F|_U^*q^i=\varphi^*p^i$ (shrinking $U$ if necessary).  Define $\tilde{F}:=(p^1,...,p^n)\in\CIN(\RR^m,\RR^n)$.  Then, $$\psi^{-1}\circ\tilde{F}\circ\varphi=\psi^{-1}\circ((F|_U)^*q^1,...,(F|_U)^*q^n)=F|_U.$$ Thus $\tilde{F}|_{\tilde{U}}=\psi\circ F\circ\varphi^{-1}.$  We conclude that $F$ is a morphism between $\mathbf{\Psi}((R,\CIN(R)))$ and $\mathbf{\Psi}((S,\CIN(S)))$.  Defining $\mathbf{\Psi}(F)$ to be $F$, we have a well-defined functor.
\end{proof}

\begin{lemma}
The functors $\mathbf{\Phi}$ and $\mathbf{\Psi}$ are inverses of one another.
\end{lemma}

\begin{proof}
This is clear for maps, and so we need only show it for objects in the two categories $\mathcal{S}_0$ and $\mathbf{Subcart}$. Let $(S,\mathcal{A})$ be an object in $\mathcal{S}_0$.  Let $\mathcal{A}'$ be the maximal atlas of $\mathbf{\Psi}\circ\mathbf{\Phi}((S,\mathcal{A}))$, as described in the proof of the lemma above.  Since charts and local diffeomorphisms to differential subspaces are identified by these functors, these two maximal atlases are the same, and $\mathbf{\Psi}\circ\mathbf{\Phi}$ is the identity on $\mathcal{S}_0$.\\

We next want to show that the reversed composition is the identity functor on objects in $\mathbf{Subcart}$.  To this end let $(S,\CIN(S))$ be a subcartesian space.  Let $(S,\mathcal{F}):=\mathbf{\Phi}\circ\mathbf{\Psi}((S,\CIN(S)))$. Then we want to show that $\mathcal{F}=\CIN(S)$.  Let $f\in\CIN(S)$.  Then for every $x\in S$ there exists an open neighbourhood $U\subseteq S$ of $x$, a diffeomorphism $\varphi:U\to\tilde{U}\subseteq\RR^m$ and a function $\tilde{f}\in\CIN(\RR^m)$ such that $f|_U=\varphi^*\tilde{f}$. Hence, $f|_U\in\varphi^*\CIN(\RR^n)$.  By definition of $\mathcal{F}$, $f\in\mathcal{F}$.\\

Conversely, let $f\in\mathcal{F}$.  By definition, for every $x\in S$ there exists an open neighbourhood $U\subseteq S$ of $x$, a diffeomorphism $\varphi:U\to\tilde{U}\subseteq\RR^m$ and a function $\tilde{f}\in\CIN(\RR^m)$ such that $f|_U=\varphi^*\tilde{f}$.  Now, $\varphi$ is a diffeomorphism onto a differential subspace $\tilde{U}$ for $\RR^m$.  Since $\tilde{f}\in\CIN(\RR^m)$, we know that $\tilde{f}|_{\tilde{U}}\in\CIN(\tilde{U})$, and hence $\varphi^*(\tilde{f}|_{\tilde{U}})\in\CIN(U)$.  But $\varphi^*(\tilde{f}|_{\tilde{U}})=f|_U$. Thus, by definition of $\CIN(U)$, for every $y\in U$ there exists a function $g_y\in\CIN(S)$ and an open neighbourhood $W\subseteq U$ of $y$ such that $g_y|_W=f|_W$.  Since $x$ and $U$ are arbitrary, by definition of a differential structure, $f\in\CIN(S)$.  This completes the proof.
\end{proof}

\begin{proof}[Proof of Theorem \ref{t:subcart}]
The proof of the theorem follows from the three lemmas above.
\end{proof} 

\addcontentsline{toc}{chapter}{Bibliography}

\bibliographystyle{amsalpha}
\bibliography{references}

\end{document}